\documentclass[10pt, a4paper, oneside, reqno]{amsart}

\makeatletter
\def\@settitle{\begin{center}%
		\baselineskip14\p@\relax
		\normalfont\LARGE\scshape\bfseries
		\@title
	\end{center}%
}
\makeatother

\makeatletter

\def\subsection{\@startsection{subsection}{2}%
	\z@{.5\linespacing\@plus.7\linespacing}{.5\linespacing}%
	{\normalfont\large\bfseries}}

\def\subsubsection{\@startsection{subsubsection}{3}%
	\z@{.5\linespacing\@plus.7\linespacing}{.5\linespacing}%
	{\normalfont\itshape}}

\usepackage[usenames, dvipsnames]{color}
\definecolor{darkblue}{rgb}{0.0, 0.0, 0.45}

\usepackage[colorlinks	= true,
raiselinks	= true,
linkcolor	= darkblue, 
citecolor	= Mahogany,
urlcolor	= ForestGreen,
pdfauthor	= {Peyman Mohajerin Esfahani},
pdftitle	= {},
pdfkeywords	= {},
pdfsubject	= {},
plainpages	= false]{hyperref}

\usepackage{dsfont,amssymb,amsmath,subfigure, graphicx,enumitem} 
\usepackage{amsfonts,dsfont,mathtools, mathrsfs,amsthm} 
\usepackage[amssymb, thickqspace]{SIunits}
\usepackage{algorithm,algorithmicx,fancyhdr}

\allowdisplaybreaks
\date{\today}

\theoremstyle{theorem}
\newtheorem{Thm}{Theorem}[section]
\newtheorem{Prop}[Thm]{Proposition}

\newtheorem{Lem}[Thm]{Lemma}
\newtheorem{Cor}[Thm]{Corollary}
\newtheorem{As}[Thm]{Assumption}

\newtheorem{Def}[Thm]{Definition}

\newtheorem{Rem}[Thm]{Remark}

\theoremstyle{remark}
\newtheorem{Ex}{Example}

\definecolor{green}{rgb}{0,.6,0}

\newcommand{\Max}{\max\limits_}
\newcommand{\Min}{\min\limits_}
\newcommand{\Sup}{\sup\limits_}
\newcommand{\Inf}{\inf\limits_}

\newcommand{\PP}{\mathds{P}}
\newcommand{\EE}{\mathds{E}}
\newcommand{\R}{\mathbb{R}}
\newcommand{\Ru}{\overline {\R}}

\newcommand{\N}{\mathbb{N}}
\newcommand{\borel}{\mathfrak{B}}

\newcommand{\ra}{\rightarrow}
\newcommand{\da}{\downarrow}

\newcommand{\ind}[1]{\mathds{1}_{#1}}
\newcommand{\Let}{\coloneqq}

\newcommand{\diff}{\mathrm{d}}

\newcommand{\ol}[1]{\overline{#1}}
\newcommand{\ul}[1]{\underline{#1}}

\newcommand{\wh}{\widehat}

\newcommand{\tr}{^\intercal}

\newcommand{\eps}{\varepsilon}

\newcommand{\X}{\mathbb{X}}

\newcommand{\amb}{\wh{\mathcal{P}}}
\newcommand{\data}{\wh{\xi}}
\newcommand{\Pem}{\wh{\PP}_N}
\newcommand{\Xiem}{\wh{\Xi}_N}
\newcommand{\Q}{\mathds{Q}}
\newcommand{\M}{\mathcal{M}}
\newcommand{\Lip}{\mathcal{L}}
\newcommand{\dir}[1]{\delta_{#1}}

\newcommand{\Wass}[2]{d_{\mathrm W}\big(#1,#2\big)}
\newcommand{\inner}[2]{\big \langle #1, #2 \big \rangle }

\DeclareMathOperator{\cl}{cl}
\newcommand{\set}[1]{\mathbb{#1}}

\newcommand{\normT}[1]{\left\| #1 \right\|_{\rm T}}

\newcommand{\J}{J^\star}

\newcommand{\xsaa}{\wh{x}_{\rm SAA}}
\newcommand{\Jsaa}{\wh{J}_{\rm SAA}}
\newcommand{\xdd}{\wh{x}_N} 
\newcommand{\Jdd}{\wh{J}_N} 

\newcommand{\xacc}{x^\star}

\newcommand{\ball}[2]{\mathbb{B}_{#2}(#1)} 

\newcommand{\st}{\normalfont \text{s.t.}}
\newcommand{\Oo}{\mathcal{O}}
\newcommand{\indI}[1]{\chi_{#1}}

\newcommand{\One}{e}

\newcommand{\cvar}{{\rm CVaR}}

\title[Data-Driven Distributionally Robust Optimization Using the Wasserstein Metric]
{Data-Driven Distributionally Robust Optimization Using the Wasserstein Metric: Performance Guarantees and Tractable Reformulations}

\author{Peyman Mohajerin Esfahani and Daniel Kuhn}%
	\thanks{The authors are with the Delft Center for Systems and Control, TU Delft, The Netherlands ({\tt P.MohajerinEsfahani@tudelft.nl}), and the Risk Analytics and Optimization Chair, EPFL, Switzerland (\texttt{daniel.kuhn@epfl.ch}).}

\begin{document} 
\maketitle

\begin{abstract}
	We consider stochastic programs where the distribution of the uncertain parameters is only observable through a finite training dataset. Using the Wasserstein metric, we construct a ball in the space of (multivariate and non-discrete) probability distributions centered at the uniform distribution on the training samples, and we seek decisions that perform best in view of the worst-case distribution within this Wasserstein ball. The state-of-the-art methods for solving the resulting distributionally robust optimization problems rely on global optimization techniques, which quickly become computationally excruciating. In this paper we demonstrate that, under mild assumptions, the distributionally robust optimization problems over Wasserstein balls can in fact be reformulated as finite convex programs---in many interesting cases even as tractable linear programs. Leveraging recent measure concentration results, we also show that their solutions enjoy powerful finite-sample performance guarantees. Our theoretical results are exemplified in mean-risk portfolio optimization as well as uncertainty quantification.
\end{abstract}

\section{Introduction} 
\label{sec:introduction}
Stochastic programming is a powerful modeling paradigm for optimization under uncertainty. The goal of a generic single-stage stochastic program is to find a decision $x\in \R^n$ that minimizes an expected cost~$\EE^\PP[h(x,\xi)]$, where the expectation is taken with respect to the distribution $\PP$ of the {continuous random vector} $\xi\in\R^m$. However, classical stochastic programming is challenged by the large-scale decision problems encountered in today's increasingly interconnected world. First, the distribution $\PP$ is never observable but must be inferred from data. However, if we calibrate a stochastic program to a given dataset and evaluate its optimal decision on a different dataset, then the resulting out-of-sample performance is often disappointing---even if the two datasets are generated from the same distribution. This phenomenon is termed the {\em optimizer's curse} and is reminiscent of overfitting effects in statistics \cite{ref:SmiRob-06}. Second, in order to evaluate the objective function of a stochastic program for a fixed decision $x$, we need to compute a multivariate integral, which is \#P-hard  even if $h(x,\xi)$ constitutes the positive part of an affine function, while $\xi$ is uniformly distributed on the unit hypercube \cite[Corollary~1]{ref:GraKunWie-15}.

Distributionally robust optimization is an alternative modeling paradigm, where the objective is to find a decision $x$ that minimizes the {\em worst-case} expected cost $\sup_{\Q \in \mathcal P} \EE^\Q [ h(x,\xi)]$. Here, the worst-case is taken over an ambiguity set $\mathcal P$, that is, a family of distributions characterized through certain known properties of the unknown data-generating distribution $\PP$. Distributionally robust optimization problems have been studied since Scarf's seminal treatise on the ambiguity-averse newsvendor problem in 1958 \cite{ref:Scarf-58}, but the field has gained thrust only with the advent of modern robust optimization techniques in the last decade \cite{ref:BenElNem-09,ref:BertSim-04}. Distributionally robust optimization has the following striking benefits. First, adopting a worst-case approach regularizes the optimization problem and thereby mitigates the optimizer's curse characteristic for stochastic programming. Second, distributionally robust models are often tractable even though the corresponding stochastic model with the true data-generating distribution (which is generically continuous) are $\#P$-hard. So even if the data-generating distribution was known, the corresponding stochastic program could not be solved efficiently.

The ambiguity set $\mathcal P$ is a key ingredient of any distributionally robust optimization model. A good ambiguity set should be rich enough to contain the true data-generating distribution with high confidence. On the other hand, the ambiguity set should be small enough to exclude pathological distributions, which would incentivize overly conservative decisions. The ambiguity set should also be easy to parameterize from data, and---ideally---it should facilitate a tractable reformulation of the distributionally robust optimization problem as a structured mathematical program that can be solved with off-the-shelf optimization software.

Distributionally robust optimization models where $\xi$ has finitely many realizations are reviewed in~\cite{ref:BenHer-13, ref:BertSAA-14, ref:PosHertMel-14}. This paper focuses on situations where $\xi$ can have a continuum of realizations. In this setting, the existing literature has studied three types of ambiguity sets. Moment ambiguity sets contain all distributions that satisfy certain moment constraints, see for example \cite{ref:DelYe-10,ref:GohSim-10,ref:WieKuhSim-14} or the references therein. An attractive alternative is to define the ambiguity set as a ball in the space of probability distributions by using a probability distance function such as the Prohorov metric~\cite{ref:ErdoIyen-06}, the Kullback-Leibler divergence \cite{ref:JiaGua-14, ref:HuHong-13}, or the Wasserstein metric~\cite{ref:PflWoz-07,ref:Woz-12} etc. Such metric-based ambiguity sets contain all distributions that are close to a {\em nominal} or {\em most likely} distribution with respect to the prescribed probability metric. By adjusting the radius of the ambiguity set, the modeler can thus control the degree of conservatism of the underlying optimization problem. If the radius drops to zero, then the ambiguity set shrinks to a singleton that contains only the nominal distribution, in which case the distributionally robust problem reduces to an ambiguity-free stochastic program. In addition, ambiguity sets can also be defined as confidence regions of goodness-of-fit tests~\cite{ref:BertSAA-14}.


In this paper we study distributionally robust optimization problems with a {\em Wasserstein ambiguity set} centered at the uniform distribution $\Pem$ on $N$ independent and identically distributed training samples. The Wasserstein distance of two distributions $\Q_1$ and $\Q_2$ can be viewed as the minimum transportation cost for moving  the probability mass from $\Q_1$ to $\Q_2$, and the Wasserstein ambiguity set contains all (continuous or discrete) distributions that are sufficiently close to the (discrete) empirical distribution $\Pem$ with respect to the Wasserstein metric. Modern measure concentration results from statistics guarantee that the unknown data-generating distribution $\PP$ belongs to the Wasserstein ambiguity set around $\Pem$ with confidence $1-\beta$ if its radius is a sublinearly growing function of $\log(1/\beta)/N$ \cite{ref:Boll-07, ref:FouGui-14}. The optimal value of the distributionally robust problem thus provides an upper confidence bound on the achievable out-of-sample cost.

While Wasserstein ambiguity sets offer powerful out-of-sample performance guarantees and enable the decision maker to control the model's conservativeness, moment-based ambiguity sets appear to display better tractability properties. Specifically, there is growing evidence that distributionally robust models with moment ambiguity sets are more tractable than the corresponding stochastic models {because the intractable high-dimensional integrals in the objective function are replaced with tractable (generalized) moment problems} \cite{ref:DelYe-10,ref:GohSim-10,ref:WieKuhSim-14}. In contrast, distributionally robust models with Wasserstein ambiguity sets are believed to be harder than their stochastic counterparts \cite{ref:PflPich-14}. Indeed, the state-of-the-art method for computing the worst-case expectation over a Wasserstein ambiguity set~$\mathcal P$ relies on global optimization techniques. Exploiting the fact that the extreme points of $\mathcal P$ are discrete distributions with a fixed number of atoms~\cite{ref:Woz-12}, one may reformulate the original worst-case expectation problem as a finite-dimensional non-convex program, which can be solved via ``difference of convex programming" methods, see \cite{ref:Woz-12} or \cite[Section~7.1]{ref:PflPich-14}. However, the computational effort is reported to be considerable, and there is no guarantee to find the global optimum. {Nevertheless, tractability results are available for special cases. Specifically, the worst case of a convex law-invariant risk measure with respect to a Wasserstein ambiguity set $\mathcal P$ reduces to the sum of the nominal risk and a regularization term whenever $h(x,\xi)$ is affine in $\xi$ and $\mathcal P$ does not include any support constraints \cite{ref:Woz-14}. Moreover, while this paper was under review we became aware of the PhD thesis~\cite{ref:Zhao-14}, which reformulates a distributionally robust two-stage unit commitment problem over a Wasserstein ambiguity set as a semi-infinite linear program, which is subsequently solved using a Benders decomposition algorithm.}

The main contribution of this paper is to demonstrate that the worst-case expectation over a Wasserstein ambiguity set can in fact be computed efficiently via convex optimization techniques for numerous loss functions of practical interest. Furthermore, we propose an efficient procedure for constructing an extremal distribution that attains the worst-case expectation---provided that such a distribution exists. Otherwise, we construct a sequence of distributions that attain the worst-case expectation asymptotically. As a by-product, our analysis shows that many interesting distributionally robust optimization problems with Wasserstein ambiguity sets can be solved in polynomial time. We also investigate the out-of-sample performance of the resulting optimal decisions---both theoretically and experimentally---and analyze its dependence on the number of training samples. We highlight the following main contributions of this paper.
\begin{itemize}
\item We prove that the worst-case expectation of an uncertain loss $\ell(\xi)$ over a Wasserstein ambiguity set coincides with the optimal value of a finite-dimensional convex program if $\ell(\xi)$ constitutes a pointwise maximum of finitely many concave functions. Generalizations to convex functions or to sums of maxima of concave functions are also discussed. We conclude that worst-case expectations can be computed efficiently to high precision via modern convex optimization algorithms.
\item We describe a supplementary finite-dimensional convex program whose optimal (near-optimal) solutions can be used to construct exact (approximate) extremal distributions for the infinite-dimensional worst-case expectation problem.
\item We show that the worst-case expectation reduces to the optimal value of an explicit linear program if the $1$-norm or the $\infty$-norm is used in the definition of the Wasserstein metric and if $\ell(\xi)$ belongs to any of the following function classes: (1) a pointwise maximum or minimum of affine functions; (2) the indicator function of a closed polytope or the indicator function of the complement of an open polytope; (3) the optimal value of a parametric linear program whose cost or right-hand side coefficients depend linearly on~$\xi$.
\item Using recent measure concentration results from statistics, we demonstrate that the optimal value of a distributionally robust optimization problem over a Wasserstein ambiguity set provides an upper confidence bound on the out-of-sample cost of the worst-case optimal decision. We validate this theoretical performance guarantee in numerical tests.
\end{itemize}

If the uncertain parameter vector $\xi$ is confined to a fixed finite subset of $\R^m$, then the worst-case expectation problems over Wasserstein ambiguity sets simplify substantially and can often be reformulated as tractable conic programs by leveraging ideas from robust optimization. An elegant second-order conic reformulation has been discovered, for instance, in the context of distributionally robust regression analysis \cite{ref:MehZhan-14}, and a comprehensive list of tractable reformulations of distributionally robust risk constraints for various risk measures is provided in \cite{ref:PosHertMel-14}. Our paper extends these tractability results to the practically relevant case where $\xi$ has uncountably many possible realizations---without resorting to space tessellation or discretization techniques that are prone to the curse of dimensionality.

When $\ell(\xi)$ is linear and the distribution of $\xi$ ranges over a Wasserstein ambiguity set without support constraints, one can derive a concise closed-form expression for the worst-case risk of $\ell(\xi)$ for various convex risk measures \cite{ref:Woz-14}. However, these analytical solutions come at the expense of a loss of generality. We believe that the results of this paper may pave the way towards an efficient computational procedure for evaluating the worst-case risk of $\ell(\xi)$ in more general settings where the loss function may be non-linear and $\xi$ may be subject to support constraints.

Among all metric-based ambiguity sets studied to date, the Kullback-Leibler ambiguity set has attracted most attention from the robust optimization community. It has first been used in financial portfolio optimization to capture the distributional uncertainty of asset returns with a Gaussian nominal distribution \cite{ref:ElOksOus-03}. Subsequent work has focused on Kullback-Leibler ambiguity sets for discrete distributions with a fixed support, which offer additional modeling flexibility without sacrificing computational tractability \cite{ref:Cal-07,ref:BenHer-13}. It is also known that distributionally robust chance constraints involving a generic Kullback-Leibler ambiguity set are equivalent to the respective classical chance constraints under the nominal distribution but with a rescaled violation probability \cite{ref:JiaGua-14,ref:HuHongSo-13}. Moreover, closed-form counterparts of distributionally robust expectation constraints with Kullback-Leibler ambiguity sets have been derived in \cite{ref:HuHong-13}. 

{However, Kullback-Leibler ambiguity sets typically fail to represent confidence sets for the unknown distribution $\PP$. To see this, assume that $\PP$ is absolutely continuous with respect to the Lebesgue measure and that the ambiguity set is centered at the discrete empirical distribution $\Pem$. Then, any distribution in a Kullback-Leibler ambiguity set around $\Pem$ must assign positive probability mass to each training sample. As $\PP$ has a density function, it must therefore reside outside of the Kullback-Leibler ambiguity set irrespective of the training samples. Thus, Kullback-Leibler ambiguity sets around $\Pem$ contain $\PP$ with probability~0. In contrast, Wasserstein ambiguity sets centered at $\Pem$ contain discrete as well as continuous distributions and, if properly calibrated, represent meaningful confidence sets for~$\PP$. We will exploit this property in Section~\ref{sec:wass} to derive finite-sample guarantees. A comparison and critical assessment of various metric-based ambiguity sets is provided in~\cite{ref:Shap-15}. Specifically, it is shown that worst-case expectations over Kullback-Leibler and other divergence-based ambiguity sets are law invariant. In contrast, worst-case expectations over Wasserstein ambiguity sets are not. The law invariance can be exploited to evaluate worst-case expectations via the sample average approximation.}

The models proposed in this paper fall within the scope of data-driven distributionally robust optimization \cite{ref:ErdoIyen-06,ref:ChehWeb-10, ref:BertSAA-14,ref:HanKuh-13}. Closest in spirit to our work is the robust sample average approximation \cite{ref:BertSAA-14}, which seeks decisions that are robust with respect to the ambiguity set of all distributions that pass a prescribed statistical hypothesis test. Indeed, the distributions within the Wasserstein ambiguity set could be viewed as those that pass a multivariate goodness-of-fit test in light of the available training samples. {This amounts to interpreting the Wasserstein distance between the empirical distribution $\Pem$ and a given hypothesis $\Q$ as a test statistic and the radius of the Wasserstein ambiguity set as a threshold that needs to be chosen in view of the test's desired significance level $\beta$. The Wasserstein distance has already been used in tests for normality~\cite{ref:Barr-99} and to devise nonparametric homogeneity tests~\cite{ref:Ram15}.}

The rest of the paper proceeds as follows. Section~\ref{sec:prob} sketches a generic framework for data-driven distributionally robust optimization, while Section~\ref{sec:wass}  introduces our specific approach based on Wasserstein ambiguity sets and establishes its out-of-sample performance guarantees. In Section~\ref{sec:dist} we demonstrate that many worst-case expectation problems over Wasserstein ambiguity sets can be reduced to finite-dimensional convex programs, and we develop a systematic procedure for constructing worst-case distributions. Explicit linear programming reformulations of distributionally robust single and two-stage stochastic programs as well as uncertainty quantification problems are derived in Section~\ref{sec:cases}. Section~\ref{sec:exten} extends the scope of the basic approach to broader classes of objective functions, and Section~\ref{sec:num} reports on numerical results.

\paragraph{\bf Notation} 
We denote by $\R_+$ the non-negative and by $\Ru \Let \R \cup \{-\infty,\infty \}$ the extended reals. Throughout this paper, we adopt the conventions of extended arithmetics, whereby $\infty\cdot0 = 0\cdot\infty = {0 / 0 } = 0$ and $\infty - \infty = -\infty + \infty = 1/0 = \infty$. The inner product of two vectors $a,b\in\R^m$ is denoted by $\inner{a}{b} \Let a\tr b$. Given a norm $\|\cdot\|$ on $\R^m$, the dual norm is defined through $\|z\|_* \Let \sup_{\|\xi\|\le 1} \inner{z}{\xi}$. A function $f:\R^m\ra \Ru$ is proper if $f(\xi)<+\infty$ for at least one $\xi$ and $f(\xi)>-\infty$ for every $\xi$ in $\R^m$. The conjugate of $f$ is defined as $f^*(z) \Let \sup_{\xi \in \R^m} \inner{z}{\xi} - f(\xi)$. Note that conjugacy preserves properness. For a set $\Xi\subseteq \R^m$, the indicator function $\ind{\Xi}$ is defined through $\ind{\Xi}(\xi)=1$ if $\xi\in \Xi$; $=0$ otherwise. Similarly, the characteristic function $\chi_\Xi$ is defined via $\chi_\Xi(\xi)=0$ if $\xi\in \Xi$; $=\infty$ otherwise. The support function of $\Xi$ is defined as $\sigma_{\Xi}(z) \Let \sup_{\xi \in \Xi} \inner{z}{\xi}$. It coincides with the conjugate of $\chi_\Xi$. We denote by $\dir{\xi}$ the Dirac distribution concentrating unit mass at $\xi\in\R^m$. The product of two probability distributions $\PP_1$ and $\PP_2$ on $\Xi_1$ and $\Xi_2$, respectively, is the distribution $\PP_1\otimes\PP_2 $ on $\Xi_1\times \Xi_2$. The $N$-fold product of a distribution $\PP$ on $\Xi$ is denoted by $\PP^N$, which represents a distribution on the Cartesian product space $\Xi^N$. {Finally, we set the expectation of $\ell:\Xi\ra\Ru$ under $\PP$ to $\EE^\PP[\ell(\xi)] = \EE^\PP\big[\max\{\ell(\xi),0\}\big] + \EE^\PP\big[\min\{\ell(\xi),0\}\big]$, which is well-defined by the conventions of extended arithmetics. 
}

\section{Data-Driven Stochastic Programming}
\label{sec:prob}

    Consider the stochastic program
        \begin{align}
        \label{Ex-true}
            \J \Let \inf_{x \in \X} \left\{ \EE^\PP \big[ h(x,\xi) \big] =  \int_{\Xi} h(x,\xi)\, \PP(\diff \xi)\right\}
        \end{align}
    with feasible set $\X \subseteq \R^n$, uncertainty set $\Xi\subseteq \R^m$ and loss function $h : \R^n \times \R^m \ra \Ru$. The loss function depends both on the decision vector $x\in\R^n$ and the random vector $\xi\in\R^m$, whose distribution $\PP$ is supported on $\Xi$. Problem~\eqref{Ex-true} can be viewed as the first-stage problem of a two-stage stochastic program, where $h(x,\xi)$ represents the optimal value of a subordinate second-stage problem \cite{ref:Shap&Dent&Rusz}. Alternatively, problem~\eqref{Ex-true} may also be interpreted as a generic learning problem in the spirit of \cite{ref:Vapnik}. 
	
	Unfortunately, in most situations of practical interest, the distribution $\PP$ is not precisely known, and therefore we miss essential information to solve problem \eqref{Ex-true} {\em exactly}. However, $\PP$ is often partially observable through a finite set of $N$ independent samples, {\em e.g.}, past realizations of the random vector $\xi$. We denote the training dataset comprising these samples by $\Xiem \Let \{\data_i\}_{i\le N} \subseteq \Xi$. We emphasize that---before its revelation---the dataset $\Xiem$ can be viewed as a random object governed by the distribution $\PP^N$ supported on $\Xi^N$.

    
    A {\em data-driven solution} for problem \eqref{Ex-true} is a feasible decision $\xdd \in \X$ that is constructed from the training dataset $\Xiem$. Throughout this paper, we notationally suppress the dependence of $\xdd$ on the training samples in order to avoid clutter. Instead, we reserve the superscript `\,$\widehat{~}$\,' for objects that depend on the training data and thus constitute random objects governed by the product distribution $\PP^N$. The {\em out-of-sample performance} of $\xdd$ is defined as $\EE^\PP \big[ h(\xdd,\xi) \big]$ and can thus be viewed as the expected cost of $\xdd$ under a new sample $\xi$ that is independent of the training dataset. As $\PP$ is unknown, however, the exact out-of-sample performance cannot be evaluated in practice, and the best we can hope for is to establish {\em performance guarantees} in the form of tight bounds. The feasibility of $\xdd$ in~\eqref{Ex-true} implies $\J\leq \EE^\PP \big[ h(\xdd,\xi) \big]$, but this lower bound is again of limited use as $\J$ is unknown and as our primary concern is to bound the costs from above. Thus, we seek data-driven solutions $\xdd$ with performance guarantees of the type
        \begin{align}
        	\label{out-of-sample}
            \PP^N\Big\{ \Xiem ~:~ \EE^\PP \big[ h(\xdd,\xi) \big] \leq \Jdd \Big\}\geq 1-\beta,
        \end{align}
	where $\Jdd$ constitutes an upper bound that may depend on the training dataset, and $\beta\in (0,1)$ is a \emph{significance parameter} with respect to the distribution $\PP^N$, which governs both $\xdd$ and $\Jdd$. Hereafter we refer to $\Jdd$ as a {\em certificate} for the out-of-sample performance of $\xdd$ and to the probability on the left-hand side of~\eqref{out-of-sample} as its {\em reliability}. Our ideal goal is to find a data-driven solution with the lowest possible out-of-sample performance. This is impossible, however, as $\PP$ is unknown, and the out-of-sample performance cannot be computed. We thus pursue the more modest but achievable goal to find a data-driven solution with a low certificate and a high reliability.

    A natural approach to generate data-driven solutions $\xdd$ is to approximate $\PP$ with the discrete empirical probability distribution
        \begin{align}
        \label{Pem}
            \Pem \Let {1 \over N} \sum_{i = 1}^{N} \dir{\data_i},
        \end{align}
    that is, the uniform distribution on $\Xiem$. This amounts to approximating the original stochastic program~\eqref{Ex-true} with the \emph{sample-average approximation} (SAA) problem
        \begin{align}
        \label{Ex_emp}
            \Jsaa \Let \inf_{x \in \X} \left\{ \EE^{\Pem} \big[ h(x,\xi) \big] =  {1 \over N} \sum_{i = 1}^{N} h(x, \data_i)\right\} .
        \end{align}
	
	If the feasible set $\X$ is compact and the loss function is uniformly continuous in $x$ across all $\xi\in\Xi$, then the optimal value and optimal solutions of the SAA problem \eqref{Ex_emp} converge almost surely to their counterparts of the true problem \eqref{Ex-true} as $N$ tends to infinity \cite[Theorem~5.3]{ref:Shap&Dent&Rusz}. Even though finite sample performance guarantees of the type \eqref{out-of-sample} can be obtained under additional assumptions such as Lipschitz continuity of the loss function (see {\em e.g.},~\cite[Theorem~1]{ref:Shap:SAA-05}), the SAA problem has been conceived primarily for situations where the distribution $\PP$ is known and additional samples can be acquired cheaply via random number generation. However, the optimal solutions of the SAA problem tend to display a poor out-of-sample performance in situations where $N$ is small and where the acquisition of additional samples would be costly. 
	

	In this paper we address problem \eqref{Ex-true} with an alternative approach that explicitly accounts for our ignorance of the true data-generating distribution $\PP$, and that offers attractive performance guarantees even when the acquisition of additional samples from $\PP$ is impossible or expensive. Specifically, we use $\Xiem$ to design an ambiguity set $\amb_N$ containing all distributions that could have generated the training samples with high confidence. { This ambiguity set enables us to define the certificate $\Jdd$ as the optimal value of a distributionally robust optimization problem that minimize the {\em worst-case} expected cost.}
            \begin{align}
            \label{DRO}
                \Jdd \Let 
                \inf\limits_{x \in \X} \sup\limits_{\Q \in \amb_N} \EE^\Q \big[ h(x,\xi) \big] 
            \end{align}
            {Following \cite{ref:PflWoz-07}, we construct $\amb_N$ as a ball around the empirical distribution~\eqref{Pem} with respect to the Wasserstein metric. In the remainder of the paper we will demonstrate that the optimal value $\Jdd$ as well as any optimal solution $\xdd$ (if it exists) of the distributionally robust problem~\eqref{DRO} satisfy the following conditions.}
	\begin{enumerate}[label=(\roman*), itemsep = 1mm, topsep = 1mm]
        \item \label{cond-conf} {\bf Finite sample guarantee:} {For a carefully chosen size of the ambiguity set, the certificate $\Jdd$ provides a $1-\beta$ confidence bound of the type~\eqref{out-of-sample} on the out-of-sample performance of $\xdd$.} 
        
        \item \label{cond-asym} {\bf Asymptotic consistency:} {As $N$ tends to infinity, the certificate $\Jdd$ and the data-driven solution $\xdd$ converge---in a sense to be made precise below---to the optimal value $\J$ and an optimizer $x^\star$ of the stochastic program~\eqref{Ex-true}, respectively.} 
        
        \item \label{cond-trac} {\bf Tractability: } {For many loss functions $h(x,\xi)$ and sets $\X$, the distributionally robust problem \eqref{DRO} is computationally tractable and admits a reformulation reminiscent of the SAA problem~\eqref{Ex_emp}.} 
	\end{enumerate}

{Conditions~\ref{cond-conf}--\ref{cond-trac} have been identified in~\cite{ref:BertSAA-14} as desirable properties of data-driven solutions for stochastic programs. Precise statements of these conditions will be provided in the remainder.  }
In Section~\ref{sec:wass} we will use the Wasserstein metric to construct ambiguity sets of the type $\amb_N$ satisfying the conditions \ref{cond-conf} and \ref{cond-asym}. In Section~\ref{sec:dist}, we will demonstrate that these ambiguity sets also fulfill the tractability condition~\ref{cond-trac}. {We see this last result as the main contribution of this paper because the state-of-the-art method for solving distributionally robust problems over Wasserstein ambiguity sets relies on global optimization algorithms~\cite{ref:PflPich-14}.}

\section{Wasserstein Metric and Measure Concentration} 
\label{sec:wass}
{  
    Probability metrics represent distance functions on the space of probability distributions. One of the most widely used examples is the Wasserstein metric, which is defined on the space $\M(\Xi)$ of all probability distributions $\Q$ supported on $\Xi$ with $\EE^\Q\big[\|\xi\|\big] = \int_\Xi \|\xi\| \,\Q(\diff\xi)<\infty$. 



	\begin{Def}[Wasserstein metric \cite{ref:KantRub-58}]
		\label{def:wass}
		 The Wasserstein metric $d_{\rm W} : \M(\Xi)\times \M(\Xi)\ra\R$ is defined via
		\begin{align*}
		\Wass{\Q_1}{\Q_2} \Let \inf \left\{ \int_{\Xi^2} \| \xi_1 -  \xi_2 \|\,  \Pi(\diff \xi_1, \diff \xi_2) ~: \begin{array}{l}\mbox{$\Pi$ is a joint distribution of $\xi_1$ and $\xi_2$} \\ \mbox{with marginals $\Q_1$ and $\Q_2$, respectively}\! \end{array}\right\}
		\end{align*}
		for all distributions $\Q_1,\Q_2\in \M(\Xi)$, where $\|\cdot\|$ represents an arbitrary norm on $\R^m$.
	\end{Def}
	
The decision variable $\Pi$ can be viewed as a \emph{transportation plan} for moving a mass distribution described by $\Q_1$ to another one described by $\Q_2$. Thus, the Wasserstein distance between $\Q_1$ and $\Q_2$ represents the cost of an optimal mass transportation plan, where the norm $\|\cdot\|$ encodes the transportation costs. We remark that a generalized $p$-Wasserstein metric for $p\geq 1$ is obtained by setting the transportation cost between $\xi_1$ and $\xi_2$ to $\|\xi_1-\xi_2\|^p$. In this paper, however, we focus exclusively on the $1$-Wasserstein metric of Definition~\ref{def:wass}, which is sometimes also referred to as the Kantorovich metric.

We will sometimes also need the following dual representation of the Wasserstein metric.

\begin{Thm}[Kantorovich-Rubinstein~\cite{ref:KantRub-58}]
\label{thm:KantorovichRubinstein}
For any distributions $\Q_1, \Q_2\in\mathcal M(\Xi)$ we have
\begin{align*}
	\Wass{\Q_1}{\Q_2} = \sup_{f \in \Lip} \Big\{ \int_{\Xi} f(\xi) \,\Q_1(\diff \xi) -  \int_{\Xi} f(\xi)\, \Q_2(\diff \xi)\Big\},
\end{align*}
where $\Lip$ denotes the space of all Lipschitz functions with $|f(\xi)-f(\xi')|\leq \|\xi-\xi'\|$ for all $\xi,\xi'\in\Xi$. 
\end{Thm}

Kantorovich and Rubinstein~\cite{ref:KantRub-58} originally established this result for distributions with bounded support. A modern proof for unbounded distributions is due to Villani~\cite[Remark~6.5, p.~107]{ref:Villani}. The optimization problems in Definition~\ref{def:wass} and Theorem~\ref{thm:KantorovichRubinstein}, which provide two equivalent characterizations of the Wasserstein metric, constitute a primal-dual pair of infinite-dimensional linear programs. The dual representation implies that two distributions $\Q_1$ and $\Q_2$ are close to each other with respect to the Wasserstein metric if and only if all functions with uniformly bounded slopes have similar integrals under $\Q_1$ and $\Q_2$. Theorem~\ref{thm:KantorovichRubinstein} also demonstrates that the Wasserstein metric is a special instance of an integral probability metric (see {\em e.g.} \cite{ref:Mull-97}) and that its generating function class coincides with a family of Lipschitz continuous functions.


}

In the remainder we will examine the ambiguity set
                \begin{align}
                    \label{eq:wasserstein-ball}
                    \ball{\Pem}{\eps} \Let \left\{ \Q \in \M(\Xi) ~:~ \Wass{\Pem}{\Q} \le \eps\right \},
                \end{align}
which can be viewed as the Wasserstein ball of radius $\eps$ centered at the empirical distribution~$\Pem$. Under a common light tail assumption on the unknown data-generating distribution $\PP$, this ambiguity set offers attractive performance guarantees in the spirit of Section~\ref{sec:prob}.

    \begin{As}[Light-tailed distribution]
    \label{a:exp}
        There exists an exponent $a > 1$ such that
        \begin{align*}
            A \Let \EE^\PP\big[ \exp(\|\xi\|^a) \big] = 
            \int_{\Xi} \exp(\|\xi\|^a)\,\PP(\diff \xi) < \infty.
        \end{align*}
    \end{As}

Assumption~\ref{a:exp} essentially requires the tail of the distribution $\PP$ to decay at an exponential rate. Note that this assumption trivially holds if $\Xi$ is compact. 
	{Heavy-tailed distributions that fail to meet Assumption~\ref{a:exp} are difficult to handle even in the context of the classical sample average approximation. Indeed, under a heavy-tailed distribution the sample average of the loss corresponding to any fixed decision $x \in \X$ may not even converge to the expected loss; see {\em e.g.}~\cite{ref:Brown-15, ref:Catoni-12}.} The following modern measure concentration result provides the basis for establishing powerful finite sample guarantees.

{
    \begin{Thm}[Measure concentration {\cite[Theorem 2]{ref:FouGui-14}}]
    \label{thm:concentration}
        If Assumption~\ref{a:exp} holds, we have
        \begin{align}
            \label{concentration}
            \PP^N \Big\{ \Wass{\PP}{\Pem} \ge \eps \Big \} \le \left\{ \begin{array}{ll} c_1 \exp\big({-c_2N\eps^{\max\{m,2\}}}\big) & \text{if } \eps \le 1, \\ c_1 \exp\big({-c_2N\eps^a}\big) & \text{if } \eps > 1,\end{array}\right.
        \end{align}
        for all $N \ge 1$, $m \neq 2$, and $\eps>0$, where $c_1, c_2$ are positive constants that only depend on $a$, $A$, and $m$.\footnote{{A similar but slightly more complicated inequality also holds for the special case $m = 2$; see \cite[Theorem 2]{ref:FouGui-14} for details.}}
    \end{Thm}


Theorem~\ref{thm:concentration} provides an a priori estimate of the probability that the unknown data-generating distribution~$\PP$ resides outside of the Wasserstein ball $\ball{\Pem}{\eps}$. Thus, we can use Theorem~\ref{thm:concentration} to estimate the radius of the smallest Wasserstein ball that contains $\PP$ with confidence $1-\beta$ for some prescribed $\beta \in (0,1)$. Indeed, equating the right-hand side of \eqref{concentration} to $\beta$ and solving for $\eps$ yields
\begin{align}
	\label{eps_N}
	\eps_N(\beta) \Let \left\{ \begin{array}{ll}  \Big({\log (c_1\beta^{-1}) \over c_2N} \Big)^{1/{\max\{m,2\}}} & \text{if } N \ge {\log(c_1\beta^{-1}) \over c_2}, \\
	\Big({\log (c_1\beta^{-1}) \over c_2N} \Big)^{1/a} & \text{if } N < {\log(c_1\beta^{-1}) \over c_2}. \end{array}\right.
\end{align}
Note that the Wasserstein ball with radius $\eps_N(\beta)$ can thus be viewed as a confidence set for the unknown true distribution as in statistical testing; see also~\cite{ref:BertSAA-14}.

\begin{Thm}[Finite sample guarantee] \label{thm:fin}
	Suppose that Assumption~\ref{a:exp} holds and that $\beta\in (0,1)$. Assume also that $\Jdd$ and $\xdd$ represent the optimal value and an optimizer of the distributionally robust program~\eqref{DRO} with ambiguity set $\amb_N = \ball{\Pem}{\eps_N(\beta)}$. Then, the finite sample guarantee \eqref{out-of-sample} holds.	
\end{Thm} 

\begin{proof}
	The claim follows immediately from Theorem~\ref{thm:concentration}, which ensures via the definition of $\eps_N(\beta)$ in~\eqref{eps_N} that $\PP^N \{ \PP\in \ball{\Pem}{\eps_N(\beta)} \} \ge 1-\beta$. Thus,  $\EE^\PP [ h(\xdd,\xi)] \leq \sup_{\Q\in\amb_N}\EE^\Q [ h(\xdd,\xi)] = \Jdd$ with probability $1-\beta$.
\end{proof}


 It is clear from~\eqref{eps_N} that for any fixed $\beta>0$, the radius $ \eps_N(\beta)$ tends to $0$ as $N$ increases. Moreover, one can show that if $\beta_N$ converges to zero at a carefully chosen rate, then the solution of the distributionally robust optimization problem~\eqref{DRO} with ambiguity set $\amb_N = \ball{\Pem}{\eps_N(\beta_N)}$ converges to the solution of the original stochastic program~\eqref{Ex-true} as $N$ tends to infinity. The following theorem formalizes this statement.


	\begin{Thm}[Asymptotic consistency] \label{thm:convergence} 
	Suppose that Assumption~\ref{a:exp} holds and that $\beta_N\in(0,1)$, $N \in \N$, satisfies $\sum_{N=1}^\infty\beta_N<\infty$ and $\lim_{N\ra\infty}\eps_N(\beta_N)=0$.\footnote{{A possible choice is $\beta_N = \exp(-\sqrt{N})$.}} Assume also that $\Jdd$ and $\xdd$ represent the optimal value and an optimizer of the distributionally robust program~\eqref{DRO} with ambiguity set $\amb_N = \ball{\Pem}{\eps_N(\beta_N)}$, $N\in\N$.
	\begin{enumerate}[label=(\roman*), itemsep = 1mm, topsep = 1mm]
						
			\item \label{thm:J-asy}
			If $h(x,\xi)$ is upper semicontinuous in $\xi$ and there exists $L\geq 0$ with $|h(x,\xi)|\leq L(1+\|\xi\|)$ for all $x\in\X$ and $\xi\in\Xi$, 
			then $\PP^\infty$-almost surely we have $\Jdd \da \J$ as $N \ra \infty$ where $\J$ is the optimal value of~\eqref{Ex-true}.
			
			\item \label{thm:x-asy}
			If the assumptions of assertion~\ref{thm:J-asy} hold, $\X$ is closed, and $h(x,\xi)$ is lower semicontinuous in $x$ for every $\xi\in\Xi$, then any accumulation point of $\{\xdd\}_{N \in \N}$ is $\PP^\infty$-almost surely an optimal solution for~\eqref{Ex-true}. 
		\end{enumerate}
	\end{Thm}
	
	The proof of Theorem~\ref{thm:convergence} will rely on the following technical lemma.
	
	\begin{Lem}[Convergence of distributions] \label{lem:asy}
	If Assumption~\ref{a:exp} holds and $\beta_N\in(0,1)$, $N \in \N$, satisfies $\sum_{N=1}^\infty\beta_N<\infty$ and $\lim_{N\ra\infty}\eps_N(\beta_N)=0$, then, any sequence $\wh\Q_N \in \ball{\Pem}{\eps_N(\beta_N)}$, $N\in\N$, where $\wh \Q_N$ may depend on the training data, converges under the Wasserstein metric (and thus weakly) to $\PP$ almost surely with respect to $\PP^\infty$, that is,
		\[\PP^{\infty} \left\{  \lim_{N \ra \infty} \Wass{\PP}{\wh\Q_N} = 0 \right\} = 1.\] 
\end{Lem} 

\begin{proof}
	As $\wh\Q_N \in \ball{\Pem}{\delta_N}$, the triangle inequality for the Wasserstein metric ensures that
		\begin{align*}
		\Wass{\PP}{\wh\Q_N } \le \Wass{\PP}{\Pem} + \Wass{\Pem}{\wh\Q_N} \le \Wass{\PP}{\Pem} + \eps_N(\beta_N).
		\end{align*}
	Moreover, Theorem~\ref{thm:concentration} implies that $\PP^N \{ \Wass{\PP}{\Pem} \le \eps_N(\beta_N)\}\geq 1-\beta_N$, and thus we have $\PP^N \{ \Wass{\PP}{\wh\Q_N } \leq 2\eps_N(\beta_N) \} \ge 1-\beta_N$. As $\sum_{N=1}^\infty\beta_N<\infty$, the Borel-Cantelli Lemma \cite[Theorem 2.18]{ref:Kallenberg-97} further implies that
		\[
			\PP^{\infty} \left\{ \Wass{\PP}{\wh\Q_N} \le \eps_N(\beta_N) ~ \text{for all sufficiently large } N \right\} = 1.
		\]
	Finally, as $\lim_{N\uparrow \infty}\eps_N(\beta_N)=0$, we conclude that $\lim_{N\uparrow\infty}\Wass{\PP}{\wh\Q_N} =0$ almost surely. Note that convergence with respect to the Wasserstein metric implies weak convergence~\cite{ref:Bois-11}.
\end{proof}

	\begin{proof}[Proof of Theorem~\ref{thm:convergence}]
		As $\wh x_N\in\X$, we have $\J \le \EE^\PP[h(\wh x_N,\xi)]$. Moreover, Theorem~\ref{thm:fin} implies that
			\[
				\PP^N \left\{ \J\le\EE^\PP[h(\xdd,\xi)] \le \Jdd \right\} \ge \PP^N \left\{ \PP\in \ball{\Pem}{\eps_N(\beta_N)} \right\} \ge 1-\beta_N,
			\]
		for all $N \in \N$. As $\sum_{N=1}^\infty\beta_N<\infty$, the Borel-Cantelli Lemma further implies that
		\begin{equation*}
			\PP^{\infty} \left\{ \J \le \EE^{\PP}[h(\xdd,\xi)] \le \Jdd ~ \text{for all sufficiently large }N \right\} = 1.
		\end{equation*}
		To prove assertion~\ref{thm:J-asy}, it thus remains to be shown that $\limsup_{N \ra \infty}\Jdd \le \J$ with probability $1$. As $h(x,\xi)$ is upper semicontinuous and grows at most linearly in $\xi$, there exists a non-increasing sequence of functions $h_k(x,\xi)$, $k\in\N$, such that $h(x,\xi)=\lim_{k\ra \infty} h_k(x,\xi)$, and $h_k(x,\xi)$ is Lipschitz continuous in $\xi$ for any fixed $x\in\X$ and $k\in\N$ with Lipschitz constant $L_k\geq 0$; see Lemma~\ref{lem:p.w.app} in the appendix. Next, choose any $\delta>0$, fix a $\delta$-optimal decision $x_\delta \in \X$ for \eqref{Ex-true} with $\EE^\PP[h(x_\delta,\xi)]\leq \J+\delta$, and for every $N\in\N$ let $\wh \Q_N \in \amb_N$ be a $\delta$-optimal distribution corresponding to $x_\delta$ with
		\[
			\sup_{\Q \in \amb_N}\EE^{\Q}[h(x_\delta,\xi)] \le \EE^{\Q_N}[h(x_\delta,\xi)] + \delta.
		\] 
		Then, we have
		\begin{align*}
		\limsup_{N\ra\infty}\Jdd \leq \limsup_{N \ra \infty} \sup_{\Q \in \amb_N}\EE^{\Q}[h(x_\delta,\xi)]  & \le  \limsup_{N \ra \infty} \EE^{\wh \Q_N} [h(x_\delta,\xi)] + \delta \\
		& \le \lim_{k \ra\infty} \limsup_{N \ra \infty} \EE^{\wh \Q_N}[h_k(x_\delta,\xi)] + \delta \\
		& \le \lim_{k \ra\infty} \limsup_{N \ra \infty} \left( \EE^{\PP}[h_k(x_\delta,\xi)] + L_k\, \Wass{\PP}{\wh \Q_N} \right) +\delta \\
		& = \lim_{k \ra \infty} \EE^{\PP}[h_k(x_\delta,\xi)] + \delta, \quad \PP^\infty\text{-almost surely}\\ 
		&= \EE^{\PP}[h(x_\delta,\xi)] + \delta \leq \J+2\delta,
		\end{align*}
		where the second inequality holds because $h_k(x,\xi)$ converges from above to $h(x,\xi)$, and the third inequality follows from Theorem~\ref{thm:KantorovichRubinstein}. Moreover, the almost sure equality holds due to Lemma~\ref{lem:asy}, and the last equality follows from the Monotone Convergence Theorem \cite[Theorem 5.5]{ref:Lang-93}, which applies because $|\EE^{\PP}[h_k(x_\delta,\xi)]| < \infty$. Indeed, recall that $\PP$ has an exponentially decaying tail due to Assumption~\ref{a:exp} and that $h_k(x_\delta,\xi)$ is Lipschitz continuous in~$\xi$. As $\delta>0$ was chosen arbitrarily, we thus conclude that  $\limsup_{N \ra \infty}\Jdd \le \J$. 		
%

%
%
		
		To prove assertion~\ref{thm:x-asy}, fix an arbitrary realization of the stochastic process $\{\data_N\}_{N \in\N}$ such that $\J = \lim_{N \ra \infty} \Jdd$ and  $\J \le \EE^{\PP}[h(\xdd,\xi)] \le \Jdd$ for all sufficiently large $N$. From the proof of assertion~\ref{thm:J-asy} we know that these two conditions are satisfied $\PP^\infty$-almost surely. Using these assumptions, one easily verifies that
		\begin{align}
		\label{pf:Jy}
			\liminf_{N \ra \infty} \EE^{\PP}[h(\wh x_{N},\xi)] \le \lim_{N \ra \infty} \Jdd=  \J.
		\end{align}
Next, let $\xacc$ be an accumulation point of the sequence $\{\xdd\}_{N \in\N}$, and note that $\xacc\in\X$ as $\X$ is closed. By passing to a subsequence, if necessary, we may assume without loss of generality that $\xacc = \lim_{N\ra \infty}\xdd$. Thus, 
\begin{align*}
\J \le \EE^{\PP}[h(\xacc,\xi)] & \le \EE^{\PP}[\liminf_{N \ra \infty} h(\wh x_{N},\xi)] \le \liminf_{N \ra \infty} \EE^{\PP}[h(\wh x_{N},\xi)] \le \J,
\end{align*}
where the first inequality exploits that $\xacc \in \X$, the second inequality follows from the lower semicontinuity of $h(x,\xi)$ in $x$, the third inequality holds due to Fatou's lemma (which applies because $h(x,\xi)$ grows at most linearly in $\xi$), and the last inequality follows from~\eqref{pf:Jy}. Therefore, we have $\EE^{\PP}[h(\xacc,\xi)] = \J$.
	\end{proof}

In the following we show that all assumptions of Theorem~\ref{thm:convergence} are necessary for asymptotic convergence, that is, relaxing any of these conditions can invalidate the convergence result. 

\begin{Ex}[Necessity of regularity conditions] \label{ex:regularity}
	\hfill
\begin{enumerate}[itemsep = 1mm, topsep = 1mm] 
	
	\item {\em Upper semicontinuity of $\xi \mapsto h(x,\xi)$ in Theorem~\ref{thm:convergence} \ref{thm:J-asy}:} \label{Ex:usc} \\
	Set $\Xi = [0,1]$, $\PP = \dir{0}$ and $h(x,\xi) = \ind{(0,1]}(\xi)$, whereby $\J = 0$. As $\PP$ concentrates unit mass at $0$, we have $\Pem=\dir{0}=\PP$ irrespective of $N\in\N$. For any $\eps > 0$, the Dirac distribution $\dir{\eps}$ thus resides within the Wasserstein ball $\ball{\Pem}{\eps}$. Hence, $\Jdd$ fails to converge to $\J$ for $\eps\rightarrow 0$ because
	\begin{align*}
		\Jdd \ge \EE^{\dir{\eps}} [h(x,\xi)] =  h(x, \eps) = 1,\quad \forall \eps>0.
	\end{align*}
	
	\item {\em Linear growth of $\xi \mapsto h(x,\xi)$ in Theorem~\ref{thm:convergence} \ref{thm:J-asy}:} \label{Ex:growth} \\ 	
	Set $\Xi = \R$, $\PP = \dir{0}$ and $h(x,\xi) = \xi^2$, which implies that $\J=0$. Note that for any  $\rho>\eps$, the two-point distribution $\Q_\rho = (1-\tfrac{\eps}{\rho})\dir{0}+\tfrac{\eps}{\rho}\dir{\rho}$ is contained in the Wasserstein ball $\ball{\Pem}{\eps}$ of radius $\eps >0$. Hence, $\Jdd$ fails to converge to $\J$ for $\eps\rightarrow 0$ because
	\begin{align*}
		\Jdd \ge \, \sup_{\rho > \eps} \,\EE^{\Q_\rho} [h(x,\xi)] =  \sup_{\rho > \eps} \, \eps \rho = \infty, \quad \forall \eps>0.
	\end{align*}
	
	\item {\em Lower semicontinuity of $x \mapsto h(x,\xi)$ in Theorem~\ref{thm:convergence} \ref{thm:x-asy}:} \label{Ex:lsc}\\	
	Set $\X = [0,1]$ and $h(x,\xi) = \ind{[0.5,1]}(x)$, whereby $\J=0$ irrespective of $\PP$. As the objective is independent of $\xi$, the distributionally robust optimization problem~\eqref{DRO} is equivalent to~\eqref{Ex-true}. Then, $\wh x_N = \tfrac{N-1}{2N}$ is a sequence of minimizers for~\eqref{DRO} whose accumulation point $x^\star = \tfrac{1}{2}$ fails to be optimal in~\eqref{Ex-true}. 
	
\end{enumerate}
\end{Ex}

A convergence result akin to Theorem~\ref{thm:convergence} for goodness-of-fit-based ambiguity sets is discussed in \cite[Section~4]{ref:BertSAA-14}. This result is complementary to Theorem~\ref{thm:convergence}. Indeed, Theorem~\ref{thm:convergence}(i) requires $h(x,\xi)$ to be upper semicontinuous in $\xi$, which is a necessary condition in our setting (see Example~\ref{ex:regularity}) that is absent in \cite{ref:BertSAA-14}. Moreover, Theorem~\ref{thm:convergence}(ii) only requires $h(x,\xi)$ to be lower semicontinuous in $x$, while \cite{ref:BertSAA-14} asks for equicontinuity of this mapping. This stronger requirement provides a stronger result, that is, the almost sure convergence of $\sup_{\Q\in\amb_N} \EE^\Q[h(x,\xi)]$ to $\EE^\PP[h(x,\xi)]$ uniformly in $x$ on any compact subset of $\X$.

Theorems~\ref{thm:fin} and \ref{thm:convergence} indicate that a careful a priori design of the Wasserstein ball results in attractive finite sample and asymptotic guarantees for the distributionally robust solutions. In practice, however, setting the Wasserstein radius to $\eps_N(\beta)$ yields over-conservative solutions for the following reasons: 

\begin{itemize}
	\item Even though the constants $c_1$ and $c_2$ in \eqref{eps_N} can be computed based on the proof of \cite[Theorem~2]{ref:FouGui-14}, the resulting Wasserstein ball is larger than necessary, {\em i.e.}, $\PP\notin \ball{\Pem}{\eps_N(\beta)}$ with probability $\ll \beta$.
	\item Even if $\PP\notin \ball{\Pem}{\eps_N(\beta)}$, the optimal value $\Jdd$ of \eqref{DRO} may still provide an upper bound on $\J$.
	\item The formula for $\eps_N(\beta)$ in \eqref{eps_N} is independent of the training data. Allowing for random Wasserstein radii, however, results in a more efficient use of the available training data. 
\end{itemize}

While Theorems~\ref{thm:fin} and \ref{thm:convergence} provide strong theoretical justification for using Wasserstein ambiguity sets, in practice, it is prudent to calibrate the Wasserstein radius via bootstrapping or cross-validation instead of using the conservative a priori bound $\eps_N(\beta)$; see Section~\ref{sec:simulation} for further details. A similar approach has been advocated in \cite{ref:BertSAA-14} to determine the sizes of ambiguity sets that are constructed via goodness-of-fit tests. 

}

So far we have seen that the Wasserstein metric allows us to construct ambiguity sets with favorable asymptotic and finite sample guarantees. In the remainder of the paper we will further demonstrate that the distributionally robust optimization problem \eqref{DRO} with a Wasserstein ambiguity set \eqref{eq:wasserstein-ball} is not significantly harder to solve than the corresponding SAA problem \eqref{Ex_emp}.

\section{Solving Worst-Case Expectation Problems} 
\label{sec:dist}

	We now demonstrate that the inner worst-case expectation problem in \eqref{DRO} over the Wasserstein ambiguity set \eqref{eq:wasserstein-ball} can be reformulated as a finite convex program for many loss functions $h(x,\xi)$ of practical interest. For ease of notation, throughout this section we suppress the dependence on the decision variable $x$. Thus, we examine a generic worst-case expectation problem
	\begin{align}
	\label{dist-rob-Ex}
		\sup\limits_{\Q \in \ball{\Pem}{\eps}} \EE^\Q \big[ \ell(\xi) \big] 
	\end{align}
	involving a decision-{\em in}dependent loss function $\ell(\xi) \Let \max_{k \le K}\ell_k(\xi)$, which is defined as the pointwise maximum of more elementary measurable functions $\ell_k:\R^m \ra \Ru$, $k\leq K$. 
	The focus on loss functions representable as pointwise maxima is non-restrictive unless we impose some structure on the functions $\ell_k$. Many tractability results in the remainder of this paper are predicated on the following convexity assumption.

	\begin{As}[Convexity]
	\label{a:ell} 
	The uncertainty set $\Xi\subseteq \R^m$ is convex and closed, and the negative constituent functions $-\ell_k$ are proper, convex, and lower semicontinuous for all $k\leq K$. Moreover, we assume that $\ell_k$ is not identically $-\infty$ on $\Xi$ for all $\le K$.
	\end{As}

	Assumption~\ref{a:ell} essentially stipulates that $\ell(\xi)$ can be written as a maximum of concave functions. As we will showcase in Section~\ref{sec:cases}, this mild restriction does not sacrifice much modeling power. Moreover, generalizations of this setting will be discussed in Section~\ref{sec:exten}. We proceed as follows. Subsection~\ref{subsec:worst-case} addresses the reduction of \eqref{dist-rob-Ex} to a finite convex program, while Subsection~\ref{subsec:ext-distr} describes a technique for constructing worst-case distributions.

	\subsection{Reduction to a Finite Convex Program}
	\label{subsec:worst-case}
	
	The worst-case expectation problem~\eqref{dist-rob-Ex} constitutes an infinite-dimensional optimization problem over probability distributions and thus appears to be intractable. {However, we will now demonstrate that \eqref{dist-rob-Ex} can be re-expressed as a finite-dimensional convex program by leveraging tools from robust optimization}.
		
%

	\begin{Thm}[Convex reduction]
		\label{thm:dist-rob-opt}
		If the convexity Assumption \ref{a:ell} holds, then for any $\eps \ge0 $ the worst-case expectation~\eqref{dist-rob-Ex} equals the optimal value of the finite convex program
		\begin{align}
		\label{eq:thm-dual:2}
		\left\{
			\begin{array}{clll} \inf\limits_{\lambda,s_i, z_{ik},\nu_{ik}} & \lambda \eps + {1 \over N}\sum\limits_{i = 1}^{N} s_i && \\
					\st & [-\ell_k]^*(z_{ik} - \nu_{ik}) + \sigma_{\Xi}(\nu_{ik}) - \inner{z_{ik}}{\data_i} \le s_i & \forall i \le N, & \forall k \le K \\
					& \|z_{ik}\|_* \le \lambda &\forall i \le N, & \forall k \le K.
			\end{array}
		\right.
		\end{align}
	\end{Thm}

	{ Recall that $[-\ell_k]^*(z_{ik} - \nu_{ik})$ denotes the conjugate of $-\ell_k$ evaluated at $z_{ik} - \nu_{ik}$ and $\|z_{ik}\|_*$ the dual norm of $z_{ik}$. Moreover, $\chi_\Xi$ represents the characteristic function of $\Xi$ and $\sigma_\Xi$ its conjugate, that is, the support function of $\Xi$. }

	
	\begin{proof}[Proof of Theorem~\ref{thm:dist-rob-opt}]
	{By using Definition~\ref{def:wass} we can re-express the worst-case expectation~\eqref{dist-rob-Ex} as}
		\begin{align*}
			\sup\limits_{\Q \in \ball{\Pem}{\eps}} \EE^\Q \big[ \ell(\xi) \big]  
			&= \left\{ 
			\begin{array}{cl} 
			\sup\limits_{\Pi,\Q} & \int_{\Xi} \ell(\xi) \, \Q(\diff \xi) \\ \st & \int_{\Xi^2} \|\xi -\xi'\| \, \Pi(\diff \xi, \diff \xi') \le \eps\\[1ex]
			& \left\{ \begin{array}{l} \mbox{$\Pi$ is a joint distribution of $\xi$ and $\xi'$}\\
			\mbox{with marginals $\Q$ and $\Pem$, respectively}
			\end{array}\right.			
			\end{array} 
			\right.\\
			& = \left\{ 
			\begin{array}{cl} 
			\sup\limits_{\Q_i \in \M(\Xi)} & {1 \over N}\sum\limits_{i = 1}^{N} \int_{\Xi} \ell(\xi) \, \Q_i(\diff \xi) \\ \st & {1 \over N}\sum\limits_{i = 1}^{N} \int_{\Xi} \|\xi -\data_i\| \, \Q_i(\diff \xi) \le \eps.
			\end{array} 
			\right.
		\end{align*}
		The second equality follows from the law of total probability, which asserts that any joint probability distribution $\Pi$ of $\xi$ and $\xi'$ can be constructed from the marginal distribution $\Pem$ of $\xi'$ and the conditional distributions $\Q_i$ of $\xi$ given $\xi'=\data_i$, $i\leq N$, that is, we may write $\Pi = {1 \over N}\sum_{i = 1}^{N} \dir{\data_i}\otimes \Q_i$. The resulting optimization problem represents a generalized moment problem in the distributions $\Q_i$, $i\leq N$. Using a standard duality argument, we obtain
		\begin{subequations}
		\label{eq:pf:thm-dual}
		\begin{align}
			\sup\limits_{\Q \in \ball{\Pem}{\eps}} \EE^\Q \big[ \ell(\xi) \big]  
			&= \sup\limits_{\Q_i \in \M(\Xi)}  \inf\limits_{\lambda \ge 0} {1 \over N}\sum\limits_{i = 1}^{N} \int_{\Xi} \ell(\xi)\, \Q_i(\diff \xi) + \lambda \Big( \eps - {1 \over N}\sum\limits_{i = 1}^{N} \int_{\Xi} \|\xi -\data_i\|\, \Q_i(\diff \xi) \Big) 
			\notag \\
			& \le \inf\limits_{\lambda \ge 0} \sup\limits_{\Q_i \in \M(\Xi)}  \lambda \eps + {1 \over N}\sum\limits_{i = 1}^{N} \int_{\Xi} \left( \ell(\xi) - \lambda \|\xi -\data_i\|\right)\Q_i(\diff \xi) 
			\label{eq:1} \\
			& = \inf\limits_{\lambda \ge 0} \lambda \eps + {1 \over N}\sum\limits_{i = 1}^{N} \sup_{\xi \in \Xi} \left (\ell(\xi) - \lambda \|\xi - \data_i\| \right ),
			\label{eq:1-}
		\end{align}
		where {\eqref{eq:1} follows from the max-min inequality}, and \eqref{eq:1-} follows from the fact that $\M(\Xi)$ contains all the Dirac distributions supported on $\Xi$. Introducing epigraphical auxiliary variables $s_i$, $i\leq N$, allows us to reformulate \eqref{eq:1-} as
		\begin{align}
			& \left\{ 
			\begin{array}{clll} \inf\limits_{\lambda, s_i} & \lambda \eps + {1 \over N}\sum\limits_{i = 1}^{N} s_i &&\\
			\st & \sup\limits_{\xi \in \Xi} \Big(\ell(\xi) - \lambda \|\xi - \data_i\| \Big)\le s_i & \forall i\le N&\\
						& \lambda \ge 0 &&
			\end{array}
			\right. \label{eq:2-} \\
			=~& \left\{ 
			\begin{array}{clll} \inf\limits_{\lambda, s_i} & \lambda \eps + {1 \over N}\sum\limits_{i = 1}^{N} s_i &&\\
			\st & \sup\limits_{\xi \in \Xi} \Big(\ell_k(\xi) - \max\limits_{\|z_{ik}\|_* \le \lambda} \inner{z_{ik}}{\xi - \data_i} \Big)\le s_i &\forall i\le N,& \forall k\le K \\
						& \lambda \ge 0 &&
			\end{array}			
			\right. \label{eq:2} \\
			\le~& \left\{ 
			\begin{array}{clll} \inf\limits_{\lambda, s_i} & \lambda \eps + {1 \over N}\sum\limits_{i = 1}^{N} s_i &&\\
			\st & \min\limits_{\|z_{ik}\|_* \le \lambda} \sup\limits_{\xi \in \Xi} \Big(\ell_k(\xi) - \inner{z_{ik}}{\xi - \data_i} \Big)\le s_i &\forall i\le N,& \forall k\le K \\
						& \lambda \ge 0. &&
			\end{array}
			\right. \label{eq:3}
		\end{align}
		Equality~\eqref{eq:2} exploits the definition of the dual norm and the decomposability of $\ell(\xi)$ into its constituents $\ell_k(\xi)$, $k\le K$. Interchanging the maximization over $z_{ik}$ with the minus sign (thereby converting the maximization to a minimization) and then with the maximization over $\xi$ leads to a restriction of the feasible set of \eqref{eq:2}. The resulting upper bound \eqref{eq:3} can be re-expressed as
		\begin{align}
			& \notag \left\{ 
			\begin{array}{clll} \inf\limits_{\lambda, s_i,z_{ik}} & \lambda \eps + {1 \over N}\sum\limits_{i = 1}^{N} s_i &&\\
			\st & \sup\limits_{\xi \in \Xi} \Big(\ell_k(\xi) - \inner{z_{ik}}{\xi} \Big) + \inner{z_{ik}}{\data_i}\le s_i &\forall i\le N,& \forall k\le K \\
						& \|z_{ik}\|_* \le \lambda & \forall i\le N, & \forall k\le K
			\end{array}
			\right. \\
			= ~& \left\{ 
			\begin{array}{clll} \inf\limits_{\lambda, s_i,z_{ik}} & \lambda \eps + {1 \over N}\sum\limits_{i = 1}^{N} s_i &&\\
			\st & [-\ell_k + \indI{\Xi}]^*(z_{ik}) - \inner{z_{ik}}{\data_i}\le s_i &\forall i\le N,& \forall k\le K \\
						& \|z_{ik}\|_* \le \lambda & \forall i\le N, & \forall k\le K,
			\end{array}
			\right. \label{eq:4}
		\end{align}
		\end{subequations}
		where \eqref{eq:4} follows from the definition of conjugacy, our conventions of extended arithmetic, and the substitution of $z_{ik}$ with $-z_{ik}$. Note that \eqref{eq:4} is already a finite convex program.
		
		Next, we show that Assumption~\ref{a:ell} reduces the inequalities \eqref{eq:1} and \eqref{eq:3} to equalities. Under Assumption~\ref{a:ell}, the inequality \eqref{eq:1} is in fact an equality for any $\eps > 0$ by virtue of an extended version of {a} well-known strong duality result for moment problems \cite[Proposition~3.4]{ref:Shap-dual-01}. One can show that \eqref{eq:1} continues to hold as an equality even for  $\eps = 0$, in which case the Wasserstein ambiguity set \eqref{eq:wasserstein-ball} reduces to the singleton $\{\Pem\}$, while \eqref{dist-rob-Ex} reduces to the sample average $\frac{1}{N}\sum_{i=1}^N \ell(\data_i)$. Indeed, for $\eps=0$ the variable $\lambda$ in \eqref{eq:1-} can be increased indefinitely at no penalty. As $\ell(\xi)$ constitutes a pointwise maximum of upper semicontinuous concave functions, an elementary but tedious argument shows that \eqref{eq:1-} converges to the sample average $\frac{1}{N}\sum_{i=1}^N \ell(\data_i)$ as $\lambda$ tends to infinity. 
			
		The inequality \eqref{eq:3} also reduces to an equality under Assumption \ref{a:ell} thanks to the classical minimax theorem \cite[Proposition 5.5.4]{ref:Bert-09}, which applies because the set $\{z_{ik} \in \R^m : \|z_{ik}\|_* \le \lambda\}$ is compact for any finite $\lambda\geq 0$. Thus, the optimal values of \eqref{dist-rob-Ex} and \eqref{eq:4} coincide. 
		
		Assumption~\ref{a:ell} further implies that the function $-\ell_k+\indI{\Xi}$ is proper, convex and lower semicontinuous. Properness holds because $\ell_k$ is not identically $-\infty$ on $\Xi$. By \cite[Theorem~11.23(a), p.~493]{ref:Rockafellar-10}, its conjugate essentially coincides with the \emph{epi-addition} (also known as \emph{inf-convolution}) of the conjugates of the functions $-\ell_k$ and $\sigma_{\Xi}$. Thus,
		\begin{align*}
			{[-\ell_k + \indI{\Xi}]^*(z_{ik})} & = \inf_{\nu_{ik}} \Big([-\ell_k]^*(z_{ik} - \nu_{ik}) + [\indI{\Xi}]^*(\nu_{ik}) \Big) \\
			& = \cl\Big[\inf_{\nu_{ik}} \Big([-\ell_k]^*(z_{ik} - \nu_{ik}) + \sigma_{\Xi}(\nu_{ik}) \Big)\Big],
		\end{align*}
		where $\cl[\cdot]$ denotes the closure operator that maps any function to its largest lower semicontinuous minorant. As $\cl[f(\xi)]\leq 0$ if and only if $f(\xi)\leq 0$ for any function $f$, we may conclude that \eqref{eq:4} is indeed equivalent to \eqref{eq:thm-dual:2} under Assumption~\ref{a:ell}. 
	\end{proof}

Note that the semi-infinite inequality in \eqref{eq:2-} generalizes the nonlinear uncertain constraints studied in \cite{ref:BenHerVial-15} because it involves an additional norm term and as the loss function $\ell(\xi)$ is not necessarily concave under Assumption~\ref{a:ell}. As in \cite{ref:BenHerVial-15}, however, the semi-infinite constraint admits a robust counterpart that involves the conjugate of the loss function and the support function of the uncertainty set.

From the proof of Theorem~\ref{thm:dist-rob-opt} it is immediately clear that the worst-case expectation~\eqref{dist-rob-Ex} is conservatively approximated by the optimal value of the finite convex program \eqref{eq:4} even if Assumption~\ref{a:ell} fails to hold. In this case the sum $-\ell_k + \indI{\Xi}$ in \eqref{eq:4} must be evaluated under our conventions of extended arithmetics, whereby $\infty - \infty = \infty$. These observations are formalized in the following corollary.

\begin{Cor}[Approximate convex reduction]  
\label{cor:approx}
For any $\eps \ge0$, the worst-case expectation~\eqref{dist-rob-Ex} is smaller or equal to the optimal value of the finite convex program~\eqref{eq:4}.
\end{Cor}

	\subsection{Extremal Distributions}
	\label{subsec:ext-distr}

	Stress test experiments are instrumental to assess the quality of candidate decisions in stochastic optimization. Meaningful stress tests require a good understanding of the extremal distributions from within the Wasserstein ball that achieve the worst-case expectation \eqref{dist-rob-Ex} for various loss functions. We now show that such extremal distributions can be constructed systematically from the solution of a convex program akin to~\eqref{eq:thm-dual:2}.	
	\begin{Thm}[Worst-case distributions]
		\label{thm:dual-dual}
		If Assumption \ref{a:ell} holds, then the worst-case expectation~\eqref{dist-rob-Ex} coincides with the optimal value of the finite convex program
		\begin{align}
		\label{dual-dual}
		\left\{
			\begin{array}{clll} \Sup{\alpha_{ik}, q_{ik}} & {1 \over N} \sum\limits_{i = 1}^{N} \sum\limits_{k = 1}^{K} \alpha_{ik}\ell_k\big( \data_i - {q_{ik} \over \alpha_{ik}}\big) \\
			\st & {1 \over N}\sum\limits_{i =1}^{N} \sum\limits_{k =1}^{K} \|q_{ik}\| \le \eps \\
						& \sum\limits_{k = 1}^{K} \alpha_{ik} = 1 &\forall i \le N\\
						& \alpha_{ik} \ge 0  &\forall i \le N, \quad \forall k\le K \\
						& \data_i - {q_{ik} \over \alpha_{ik}} \in \Xi &\forall i \le N, \quad \forall k\le K
			\end{array}
		\right.
		\end{align}
		irrespective of $\eps \ge 0$. Let $\big\{\alpha_{ik}(r), q_{ik}(r)\big\}_{r \in \N}$ be a sequence of feasible decisions whose objective values converge to the supremum of \eqref{dual-dual}. Then, the discrete probability distributions
		$$\Q_r \Let {1 \over N}\sum_{i = 1}^{N}\sum_{k = 1}^{K} \alpha_{ik}(r)\dir{\xi_{ik}(r)} \qquad \mbox{with}\qquad \xi_{ik}(r) \Let \data_i - {q_{ik}(r) \over \alpha_{ik}(r)}$$ 
		belong to the Wasserstein ball $\ball{\Pem}{\eps}$ and attain the supremum of \eqref{dist-rob-Ex} asymptotically, i.e., 
		\begin{align*}
			\sup\limits_{\Q \in \ball{\Pem}{\eps}} \EE^\Q \big[ \ell(\xi) \big] =  \lim\limits_{r \ra \infty} \EE^{\Q_r} \big[ \ell(\xi) \big] = 
			\lim\limits_{k \ra \infty} {1 \over N} \sum\limits_{i = 1}^{N} \sum\limits_{k = 1}^{K} \alpha_{ik}(r)\ell\big(\xi_{ik}(r)\big).
		\end{align*} 
	\end{Thm}

	We highlight that all fractions in \eqref{dual-dual} must again be evaluated under our conventions of extended arithmetics. Specifically, if $\alpha_{ik}=0$ and $q_{ik}\neq 0$, then $q_{ik}/\alpha_{ik}$ has at least one component equal to $+\infty$ or $-\infty$, which implies that $\data_i - q_{ik}/\alpha_{ik}\notin \Xi$. In contrast, if $\alpha_{ik}=0$ and $q_{ik}= 0$, then $\data_i - q_{ik} / \alpha_{ik}=\data_i \in \Xi$. Moreover, the $ik$-th component in the objective function of \eqref{dual-dual} evaluates to $0$ whenever $\alpha_{ik} =0$ regardless of $q_{ik}$. 
	
	The proof of Theorem \ref{thm:dual-dual} is based on the following technical lemma. 

	\begin{Lem}
		\label{lem:perspective}
		Define $F: \R^m \times \R_{+} \ra \Ru$ through $F(q,\alpha) = \inf_{z \in \R^m} \inner{z}{q - \alpha \widehat{\xi}} + \alpha f^*(z)$ for some proper, convex, and lower semicontinuous function $f:\R^m\ra \Ru$ and reference point $\widehat{\xi}\in\R^m$. Then, $F$ coincides with the (extended) perspective function of the mapping $q \mapsto -f(\widehat \xi - q)$, that is,
		\begin{align*}
			F(q, \alpha) = \left\{\begin{array}{cc}
			- \alpha f \big(\widehat{\xi} - q /\alpha\big)& \text{if }\alpha > 0, \\
			-\indI{\{0\}}(q) &  \text{if }\alpha = 0. 
			\end{array} \right.
		\end{align*}
	\end{Lem}

	\begin{proof}
		By construction, we have $F(q,0) = \inf_{z \in \R^m} \inner{z}{q} = - \indI{\{0\}}(q)$. For $\alpha > 0$, on the other hand, the definition of conjugacy implies that 
		\begin{align*}
			F(q,\alpha) = -[\alpha f^*]^*(\alpha\widehat{\xi} - q) = - \alpha [f^*]^* \big(\widehat{\xi} - {q / \alpha}\big). 
		\end{align*}
		The claim then follows because $[f^*]^* = f$ for any proper, convex, and lower semicontinuous function $f$ \cite[Proposition~1.6.1(c)]{ref:Bert-09}. 
		Additional information on perspective functions can be found in \cite[Section~2.2.3, p.~39]{ref:Boyd}. 
	\end{proof}

	\begin{proof}[Proof of Theorem \ref{thm:dual-dual}]
		By Theorem \ref{thm:dist-rob-opt}, which applies under Assumption~\ref{a:ell}, the worst-case expectation~\eqref{dist-rob-Ex} coincides with the optimal value of the convex program~\eqref{eq:thm-dual:2}. From the proof of Theorem \ref{thm:dist-rob-opt} we know that \eqref{eq:thm-dual:2} is equivalent to \eqref{eq:4}.
		\begin{subequations} 
		The Lagrangian dual of \eqref{eq:4} is given by
		\begin{align}
		\left\{
			\begin{array}{clll} \Sup{\beta_{ik}, \alpha_{ik}} &\Inf{\lambda, s_i, z_{ik}} \lambda \eps + \sum\limits_{i = 1}^{N} \Big[ {s_i \over N} + &\hspace{-3mm}\sum\limits_{k = 1}^{K} \big[\beta_{ik} \big(\|z_{ik}\|_* -\lambda \big) + \alpha_{ik}\big( [-\ell_k + \indI{\Xi}]^*(z_{ik}) - \inner{z_{ik}}{\data_i} - s_i\big)\big]\Big] \\
			\st & \alpha_{ik} \ge 0& \forall i \le N, \quad \forall k\le K \\
						& \beta_{ik} \ge 0  & \forall i \le N, \quad \forall k\le K,
			\end{array}
		\right. \notag 
		\end{align}
		where the products of dual variables and constraint functions in the objective are evaluated under the standard convention $0 \cdot \infty = 0$. Strong duality holds since the function $[-\ell_k+\indI{\Xi}]^*$ is proper, convex, and lower semicontinuous under Assumption~\ref{a:ell} and because this function appears in a constraint of \eqref{eq:4} whose right-hand side is a free decision variable. By explicitly carrying out the minimization over $\lambda$ and $s_i$, one can show that the above dual problem is equivalent to
		\begin{align}
		\left\{
		\begin{array}{clll} \Sup{\beta_{ik}, \alpha_{ik}} & \Inf{z_{ik}} ~ \sum\limits_{i = 1}^{N} \sum\limits_{k = 1}^{K} \beta_{ik} \|z_{ik}\|_* + &\hspace{-3mm}\alpha_{ik}[-\ell_k+ \indI{\Xi}]^*(z_{ik}) - \alpha_{ik}\inner{z_{ik}}{\data_i} \\
			\st & \sum\limits_{i =1}^{N} \sum\limits_{k = 1}^{K} \beta_{ik} = \eps \\
						& \sum\limits_{k = 1}^{K} \alpha_{ik} = {1 \over N} &\forall i \le N\\
						& \alpha_{ik} \ge 0 & \forall i \le N, \quad \forall k\le K \\
						& \beta_{ik} \ge 0  & \forall i \le N, \quad \forall k\le K.
			\end{array} \right. \label{eq:pf:dual-dual:1}
		\end{align}
		By using the definition of the dual norm, \eqref{eq:pf:dual-dual:1} can be re-expressed as
		\begin{align}
		& \left\{ \begin{array}{clll} \Sup{\beta_{ik}, \alpha_{ik}}& \Inf{z_{ik}} \sum\limits_{i = 1}^{N} \sum\limits_{k = 1}^{K} \Max{\|q_{ik}\| \le \beta_{ik}}\inner{z_{ik}}{q_{ik}} + &\hspace{-3mm} \alpha_{ik}[-\ell_k + \indI{\Xi}]^*(z_{ik}) - \alpha_{ik}\inner{z_{ik}}{\data_i} \Big] \\
			\st & \sum\limits_{i =1}^{N} \sum\limits_{k =1}^{K}\beta_{ik} = \eps \\
						& \sum\limits_{k = 1}^{K} \alpha_{ik} = {1 \over N} &\forall i \le N\\
						& \alpha_{ik} \ge 0 & \forall i \le N, \quad \forall k\le K \\
						& \beta_{ik} \ge 0  & \forall i \le N, \quad \forall k\le K
			\end{array} \right.  \\
		=~& \left\{ \begin{array}{clll} \Sup{\beta_{ik},\alpha_{ik}}& \Max{\|q_{ik}\| \le \beta_{ik}} \Inf{z_{ik}} \sum\limits_{i = 1}^{N} \sum\limits_{k =1}^{K} \inner{z_{ik}}{q_{ik}} + &\hspace{-3mm}\alpha_{ik}[-\ell_k+ \indI{\Xi}]^*(z_{ik}) - \alpha_{ik}\inner{z_{ik}}{\data_i} \\
			\st
						& \sum\limits_{i =1}^{N} \sum\limits_{k =1}^{K} \beta_{ik} = \eps \\
						& \sum\limits_{k = 1}^{K} \alpha_{ik} = {1 \over N} &\forall i \le N\\
						& \alpha_{ik} \ge 0 & \forall i \le N, \quad \forall k\le K \\
						& \beta_{ik} \ge 0  & \forall i \le N, \quad \forall k\le K,
			\end{array} \right. \label{eq:pf:dual-dual:2a} 
		\end{align}
		where \eqref{eq:pf:dual-dual:2a} follows from the classical minimax theorem and the fact that the $q_{ik}$ variables range over a non-empty and compact feasible set for any finite $\eps$; see \cite[Proposition 5.5.4]{ref:Bert-09}. Eliminating the $\beta_{ik}$ variables and using Lemma~\ref{lem:perspective} allows us to reformulate \eqref{eq:pf:dual-dual:2a} as
		\begin{align} 
		&\left\{ \begin{array}{clll} \Sup{\alpha_{ik}, q_{ik}} & \Inf{z_{ik}}~ \sum\limits_{i = 1}^{N} \sum\limits_{k =1}^{K} \inner{z_{ik}}{q_{ik} - \alpha_{ik}\data_i} + &\hspace{-3mm}  \alpha_{ik}[-\ell_k+ \indI{\Xi}]^*(z_{ik}) \\
			\st & \sum\limits_{i =1}^{N} \sum\limits_{k =1}^{K} \|q_{ik}\| \le \eps \\
						& \sum\limits_{k = 1}^{K} \alpha_{ik} = {1 \over N} &\forall i \le N\\
						& \alpha_{ik} \ge 0 &\forall i \le N,\quad \forall k\le K
			\end{array} \right. \label{eq:pf:dual-dual:2b}\\
			=~ & \left\{ \begin{array}{clll} \Sup{\alpha_{ik}, q_{ik}} & \sum\limits_{i = 1}^{N} \sum\limits_{k = 1}^{K} - \alpha_{ik} \Big(-\ell_k\big(\data_i - {q_{ik} \over \alpha_{ik}}\big) + &\hspace{-3mm} \indI{\Xi} \big(\data_i - {q_{ik} \over \alpha_{ik}}\big) \Big)\ind{\{\alpha_{ik}>0\}} - \indI{\{0\}}(q_{ik})\ind{\{\alpha_{ik} = 0\}}  \\
			\st & \sum\limits_{i =1}^{N} \sum\limits_{k =1}^{K} \|q_{ik}\| \le \eps \\
						& \sum\limits_{k = 1}^{K} \alpha_{ik} = {1 \over N} &\forall i \le N\\
						& \alpha_{ik} \ge 0 &\forall i \le N, \quad \forall k\le K.
			\end{array} \right. \label{eq:pf:dual-dual:3} 
		\end{align}
		Our conventions of extended arithmetics imply that the $ik$-th term in the objective function of problem~\eqref{eq:pf:dual-dual:3} simplifies to 
		\begin{equation}
		\label{nasty_alpha}
			\alpha_{ik} \ell_k\big(\data_i - {q_{ik} \over \alpha_{ik}}\big) - \indI{\Xi}\big(\data_i - {q_{ik} \over \alpha_{ik}}\big).
		\end{equation} 
		Indeed, for $\alpha_{ik}>0$, this identity trivially holds. For $\alpha_{ik}=0$, on the other hand, the $ik$-th objective term in  \eqref{eq:pf:dual-dual:3} reduces to $- \indI{\{0\}}(q_{ik})$. Moreover, the first term in \eqref{nasty_alpha} vanishes whenever $\alpha_{ik} = 0$ regardless of $q_{ik}$, and the second term in \eqref{nasty_alpha} evaluates to 0 if $q_{ik}=0$ (as $0/0=0$ and $\data_i \in \Xi$) and to $-\infty$ if $q_{ik}\neq 0$ (as $q_{ik}/0$ has at least one infinite component, implying that $\data_i+q_{ik}/0\notin \Xi$). Therefore, \eqref{nasty_alpha} also reduces to $- \indI{\{0\}}(q_{ik})$ when $\alpha_{ik}=0$. This proves that the $ik$-th objective term in \eqref{eq:pf:dual-dual:3} coincides with \eqref{nasty_alpha}. Substituting \eqref{nasty_alpha} into \eqref{eq:pf:dual-dual:3} and re-expressing $- \indI{\Xi}\big(\data_i - {q_{ik} \over \alpha_{ik}}\big)$ in terms of an explicit hard constraint yields
		\begin{align}
	 	\left\{
			\begin{array}{clll} \Sup{\alpha_{ik}, q_{ik}} & \sum\limits_{i = 1}^{N} \sum\limits_{k = 1}^{K} \alpha_{ik} \ell_k\big(\data_i - {q_{ik} \over \alpha_{ik}}\big) \\
						\st & \sum\limits_{i =1}^{N} \sum\limits_{k =1}^{K} \|q_{ik}\| \le \eps \\
						& \sum\limits_{k = 1}^{K} \alpha_{ik} = {1 \over N} &\forall i \le N\\
						& \alpha_{ik} \ge 0  &\forall i \le N, \quad \forall k\le K \\
						& \data_i - {q_{ik} \over \alpha_{ik}} \in \Xi&\forall i \le N, \quad \forall k\le K.
			\end{array}\right. \label{eq:pf:dual-dual:4}
		\end{align}
		Finally, replacing $\big\{\alpha_{ik}, q_{ik}\big\}$ with ${1 \over N}\big\{\alpha_{ik}, q_{ik}\big\}$ shows that \eqref{eq:pf:dual-dual:4} is equivalent to \eqref{dual-dual}. This completes the first part of the proof.
		
		As for the second claim, let $\{\alpha_{ik}(r), q_{ik}(r)\}_{r \in \N}$ be a sequence of feasible solutions that attains the supremum in~\eqref{dual-dual}, and set $\xi_{ik}(r) \Let \data_i - {q_{ik}(r) \over \alpha_{ik}(r)}\in\Xi$. Then, the discrete distribution
		$$\Pi_r \Let {1 \over N}\sum_{i = 1}^{N}\sum_{k = 1}^{K} \alpha_{ik}(r)\dir{\big(\xi_{ik}(r), \data_i\big)}$$ 
		has the distribution $\Q_r$ defined in the theorem statement and the empirical distribution $\Pem$ as marginals. {By the definition of the Wasserstein metric, $\Pi_r$ represents a feasible mass transportation plan that provides an upper bound on the distance between $\Pem$ and $\Q_r$; see Definition~\ref{def:wass}. Thus, we have }
		\begin{align*}
			\Wass{\Q_r}{\Pem} &\le \int_{\Xi^2} \|\xi - \xi'\| \, \Pi_r(\diff \xi, \diff \xi') = {1 \over N} \sum\limits_{i = 1}^{N} \sum\limits_{k= 1}^{K} \alpha_{ik}(r) \big\| \xi_{ik}(r) - \data_i \big\|  {= {1 \over N} \sum\limits_{i = 1}^{N} \sum\limits_{k= 1}^{K} \big\| q_{ik}(r)\big\| \le \eps,}
		\end{align*}
		where the last inequality follows readily from the feasibility of $q_{ik}(r)$ in \eqref{dual-dual}. We conclude that
		\begin{align*}
			\sup\limits_{\Q \in \ball{\Pem}{\eps}}\EE^{\Q}\big[\ell(\xi)\big] &\ge  \limsup_{k \ra \infty} \EE^{\Q_r}\big[\ell(\xi)\big] = \limsup_{k \ra \infty} {1 \over N} \sum\limits_{i = 1}^{N}\sum\limits_{k = 1}^{K}\alpha_{ik}(r) \ell\big(\xi_{ik}(r) \big)\\
			& \ge \limsup_{k \ra \infty} {1 \over N} \sum\limits_{i = 1}^{N}\sum\limits_{k = 1}^{K}\alpha_{ik}(r) \ell_k\big(\xi_{ik}(r) \big) = \sup\limits_{\Q \in \ball{\Pem}{\eps}}\EE^{\Q}\big[\ell(\xi)\big],
		\end{align*}
		\end{subequations} 
		where the first inequality holds as $\Q_r \in \ball{\Pem}{\eps}$ for all $k \in \N$, and the second inequality uses the trivial estimate $\ell \geq \ell_k$ for all $k\le K$. The last equality follows from the construction of $\alpha_{ik}(r)$ and $\xi_{ik}(r)$ and the fact that \eqref{dual-dual} coincides with the worst-case expectation~\eqref{dist-rob-Ex}. 
	\end{proof}	

	In the rest of this section we discuss some notable properties of the convex program \eqref{dual-dual}. 

	In the \emph{ambiguity-free} limit, that is, when the radius of the Wasserstein ball is set to zero, then the optimal value of the convex program~\eqref{dual-dual} reduces to the expected loss under the empirical distribution. Indeed, for $\eps = 0$ all $q_{ik}$ variables are forced to zero, and $\alpha_{ik}$ enters the objective only through $\sum_{k=1}^K \alpha_{ik}={1\over N}$. Thus, the objective function of \eqref{dual-dual} simplifies to $ \EE^{\Pem}[\ell(\xi)]$. 
	
	We further emphasize that it is not possible to guarantee the existence of a worst-case distribution that attains the supremum in \eqref{dist-rob-Ex}. In general, as shown in Theorem~\ref{thm:dual-dual}, we can only construct a sequence of distributions that attains the supremum asymptotically. The following example discusses an instance of \eqref{dist-rob-Ex} that admits no worst-case distribution.

\begin{figure}[t!]
	\begin{center}
		\includegraphics[scale = 0.8]{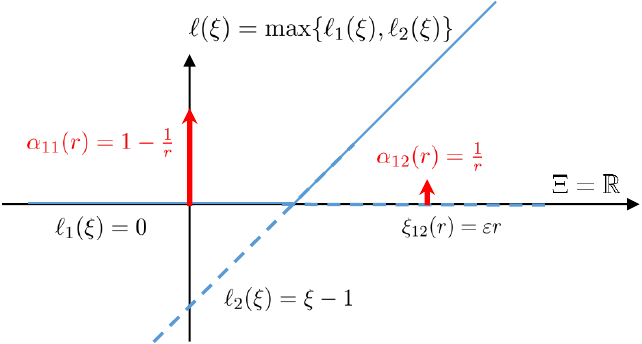}
	\end{center}
	\caption{Example of a worst-case expectation problem without a worst-case distribution}
	\label{fig:Ex}
\end{figure}

	\begin{Ex}[Non-existence of a worst-case distribution]
	\label{ex:worst-case}
		Assume that $\Xi = \R$, $N = 1$, $\data_1 = 0$, $K = 2$, $\ell_1(\xi) =0$ and $\ell_2(\xi) = \xi - 1$. In this case we have $\Pem=\dir{\{0\}}$, and problem \eqref{dual-dual} reduces to
		\begin{align*}
		\sup\limits_{\Q \in \ball{\dir{0}}{\eps}} \EE^\Q \big[ \ell(\xi) \big]  
			 = \left\{ \begin{array}{clll} \Sup{\alpha_{1j}, q_{1j}} & - q_{12} - \alpha_{12} \\
			\st & |q_{11}| + |q_{12}| \le \eps \\
						& \alpha_{11} + \alpha_{12} = 1 \\
						& {\alpha_{11} \ge 0,  \quad \alpha_{12} \ge 0.}
			\end{array} \right. 
		\end{align*}
		The supremum on the right-hand side amounts to $\eps$ and is attained, for instance, by the sequence $\alpha_{11}(r) = 1 - {1 \over k}$, $\alpha_{12}(r) = {1 \over k}$, $q_{11}(r) = 0$, $q_{12}(r) = - \eps$ for $k\in\mathbb N$. Define $$\Q_r = \alpha_{11}(r)\, \dir{\xi_{11}(r)}+\alpha_{12}(r)\, \dir{\xi_{12}(r)},$$ with $\xi_{11}(r) = \data_1 - {q_{11}(r) \over \alpha_{11}(r)}=0,$ and $\xi_{12}(r) = \data_1 - {q_{12}(r) \over \alpha_{12}(r)}=\eps k$. By Theorem \ref{thm:dual-dual}, the two-point distributions $\Q_r$ reside within the Wasserstein ball of radius $\eps$ around $\dir{0}$ and asymptotically attain the supremum in the worst-case expectation problem. However, this sequence has no weak limit as $\xi_{12}(r) = \eps k$ tends to infinity, see Figure~\ref{fig:Ex}. In fact, no single distribution can attain the worst-case expectation. Assume for the sake of contradiction that there exists $\Q^\star\in \ball{\dir{0}}{\eps}$ with $\EE^{\Q^\star}[\ell(\xi)]=\eps$. Then, we find $\eps= \EE^{\Q^\star}[\ell(\xi)]< \EE^{\Q^\star}[|\xi|]\leq \eps$, where the strict inequality follows from the relation $\ell(\xi)<|\xi|$ for all $\xi\neq 0$ and the observation that $\Q^\star\neq\dir{0}$, {while the second inequality follows from Theorem~\ref{thm:KantorovichRubinstein}.} Thus, $\Q^\star$ does not exist.
	\end{Ex}

	The existence of a worst-case distribution can, however, be guaranteed in some special cases.
	
	\begin{Cor}[Existence of a worst-case distribution] \label{cor:worst-case}
		Suppose that Assumption \ref{a:ell} holds. If the uncertainty set $\Xi$ is compact or the loss function is concave (i.e., $K=1$), then the sequence $\{\alpha_{ik}(r), \xi_{ik}(r)\}_{r \in \N}$ constructed in Theorem~\ref{thm:dual-dual} has an accumulation point $\{\alpha^\star_{ik}, \xi^\star_{ik}\}$, and 
		$$\Q^\star \Let {1 \over N}\sum_{i = 1}^{N}\sum_{k = 1}^{K} \alpha^\star_{ik}\dir{\xi^\star_{ik}}$$ 
		is a worst-case distribution achieving the supremum in \eqref{dist-rob-Ex}.
	\end{Cor}
	\begin{proof}
		If $\Xi$ is compact, then the sequence $\{\alpha_{ik}(r), \xi_{ik}(r)\}_{r \in \N}$ has a converging subsequence with limit $\{\alpha^\star_{ik},\xi^\star_{ik}\}$. Similarly, if $K = 1$, then $\alpha_{i1} = 1$ for all $i\le N$, in which case \eqref{dual-dual} reduces to a convex optimization problem with an upper semicontinuous objective function over a compact feasible set. Hence, its supremum is attained at a point $\{\alpha^\star_{ik},\xi^\star_{ik}\}$. In both cases, Theorem~\ref{thm:dual-dual} guarantees that the distribution $\Q^\star$ implied by $\{\alpha^\star_{ik},\xi^\star_{ik}\}$ achieves the supremum in \eqref{dist-rob-Ex}. 
	\end{proof}
	
	The worst-case distribution of Corollary~\ref{cor:worst-case} is discrete, and its atoms $\xi^\star_{ik}$ reside in the neighborhood of the given data points $\data_i$. By the constraints of problem~\eqref{dual-dual}, the probability-weighted cumulative distance between the atoms and the respective data points amounts to $$\sum_{i=1}^N\sum_{k=1}^K \alpha_{ik}\| \xi^\star_{ik}-\data_i \| = \sum_{i=1}^N\sum_{k=1}^K \|q_{ik}\|\leq \eps,$$ which is bounded above by the radius of the Wasserstein ball. The fact that the worst-case distribution $\Q^\star$ (if it exists) is supported outside of $\Xiem$ is a key feature distinguishing the Wasserstein ball from the ambiguity sets induced by other probability metrics such as the total variation distance or the Kullback-Leibler divergence; see Figure~\ref{fig:balls}. Thus, the worst-case expectation criterion based on Wasserstein balls advocated in this paper should appeal to decision makers who wish to immunize their optimization problems against perturbations of the data points. 
	
\begin{figure} [t!]
	\centering
	\subfigure[Empirical distribution on a training dataset with $N = 2$ samples]{\label{fig:emp} \includegraphics[width=0.3\columnwidth]{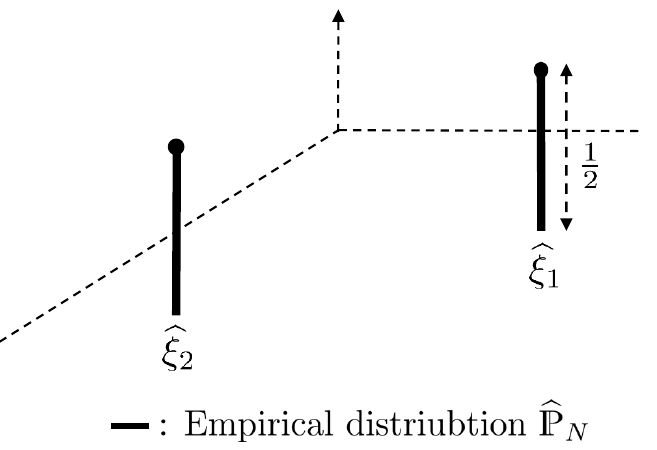}} \hspace{2mm}
	\subfigure[A representative discrete distribution in the total variation or the Kullback-Leiber ball]{\label{fig:TV-KL} \includegraphics[width=0.31\columnwidth]{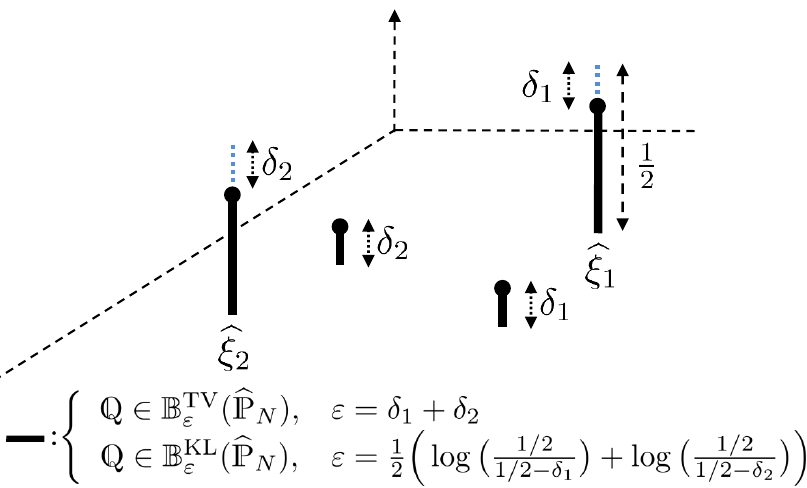}} 
	\hspace{2mm}
	\subfigure[A representative discrete distribution in the Wasserstein ball]{\label{fig:wass} \includegraphics[width=0.31\columnwidth]{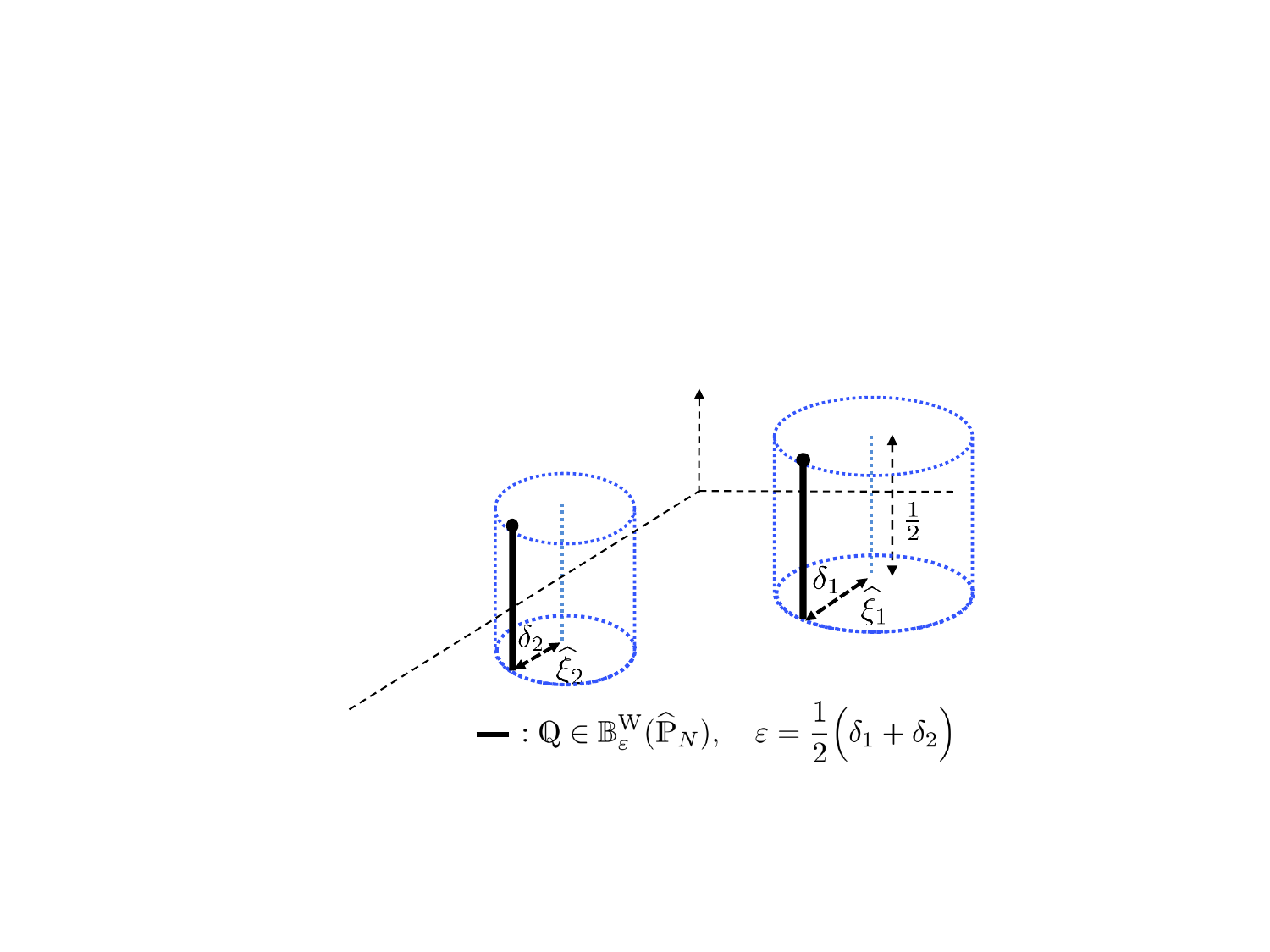}}
	\caption{Representative distributions in balls centered at $\Pem$ induced by different metrics}
	\label{fig:balls}
\end{figure}

	\begin{Rem}[Weak coupling]
		We highlight that the convex program \eqref{dual-dual} is amenable to decomposition and parallelization techniques as the decision variables associated with different sample points are only coupled through the norm constraint. We expect the resulting scenario decomposition to offer a substantial speedup of the solution times for problems involving large datasets. Efficient decomposition algorithms that could be used for solving the convex program \eqref{dual-dual} are described, for example, in \cite{ref:NeaBoyd-14} and \cite[Chapter 4]{ref:Bert-15}.
	\end{Rem}
	
	\section{Special Loss Functions}
	\label{sec:cases}

	We now demonstrate that the convex optimization problems \eqref{eq:thm-dual:2} and \eqref{dual-dual} reduce to computationally tractable conic programs for several loss functions of practical interest. 

	\subsection{Piecewise Affine Loss Functions} 
	We first investigate the worst-case expectations of convex and concave piecewise affine loss functions, which arise, for example, in option pricing \cite{ref:BertPop-00}, risk management \cite{ref:NatSimUich-10} and in generic two-stage stochastic programming~\cite{ref:BertDoaVinNat-10}. Moreover, piecewise affine functions frequently serve as approximations of {\em smooth} convex or concave loss functions. 
	\begin{Cor}[Piecewise affine loss functions]
	\label{cor:affine}
		Suppose that the uncertainty set is a polytope, that is, $\Xi = \{ \xi \in \R^m : C \xi \le d \}$ where $C$ is a matrix and $d$ a vector of appropriate dimensions. Moreover, consider the affine functions $a_k(\xi) \Let \inner{a_{k}}{\xi} + b_{k}$ for all $k\le K$.
		\begin{subequations}
		\label{affine-opt}
		\begin{itemize}
		\item[(i)] If $\ell(\xi)= \max_{k\le K}a_k(\xi)$, then the worst-case expectation \eqref{dist-rob-Ex} evaluates to
		\begin{align}
		\label{affine-max}
		\left\{
					\begin{array}{clll} \inf\limits_{\lambda,s_i, \gamma_{ik}} & \lambda \eps + {1 \over N}\sum\limits_{i = 1}^{N} s_i  \\
					\st & b_k +\inner{a_k}{\data_i}+ \inner{\gamma_{ik}}{d-C\data_i} \le s_i & \forall i \le N, & \forall k \le K\\
								& \|C\tr\gamma_{ik} - a_{k}\|_* \le \lambda & \forall i \le N, & \forall k \le K  \\
								& \gamma_{ik} \ge 0& \forall i \le N, & \forall k \le K .
					\end{array}\right.
		\end{align}
		\item[(ii)] If $\ell(\xi)= \min_{k\le K}a_k(\xi)$, then the worst-case expectation \eqref{dist-rob-Ex} evaluates to
		\begin{align}
		\label{affine-min}
		\left\{\begin{array}{clll} \inf\limits_{\lambda,s_i, \gamma_{i},\theta_{i}} \hspace{-1ex}& \lambda \eps + {1 \over N}\sum\limits_{i = 1}^{N} s_i  \\
					\st &\inner{\theta_i}{b+ A\data_i}+\inner{\gamma_{i}}{d- C\data_i} \le s_i & \forall i \le N\\
								& \|C\tr \gamma_i-A\tr\theta_i\|_* \le \lambda & \forall i \le N \\
								& \inner{\theta_{i}}{\One} = 1 & \forall i \le N\\
								& \gamma_{i}\ge 0& \forall i \le N\\
								& \theta_{i} \ge 0& \forall i \le N,
					\end{array}\right.
		\end{align}
		where $A$ is the matrix with rows $a\tr_k$, $k\le K$, $b$ is the column vector with entries $b_k$, $k\le K$, and $\One$ is the vector of all ones.
		\end{itemize}
		\end{subequations}
	\begin{proof}
		Assertion~(i) is an immediate consequence of Theorem~\ref{thm:dist-rob-opt}, which applies because $\ell(x)$ is the pointwise maximum of the affine functions $\ell_k(\xi)= a_k(\xi)$, $k\le K$, and thus Assumption~\ref{a:ell} holds for $J= K$. 
		By definition of {the} conjugacy operator, we have
		\begin{align*}
			[-\ell_k]^*(z) =[-a_k]^*(z) = \sup\limits_{\xi} \inner{z}{\xi} + \inner{a_k}{\xi} + b_k=\left\{\begin{array}{cl} 
							b_k & \text{if }z=-a_{k}, \\
							\infty & \text{else,}
						\end{array}
						\right.
			\end{align*}
			and
			\begin{align*}	
			\sigma_\Xi(\nu) = \left\{ \begin{array}{cl} \sup\limits_{\xi} & \inner{\nu}{\xi} \\ \st & C \xi \le d \end{array}\right. =  \left\{ \begin{array}{cl} \inf\limits_{\gamma\ge 0} & \inner{\gamma}{d} \\ \st & C\tr \gamma= \nu, \end{array} \right. 
		\end{align*}
		where the last equality follows from strong duality, which holds as the uncertainty set is non-empty. Assertion~(i) then follows by substituting the above expressions into \eqref{eq:thm-dual:2}. 
		
		Assertion~(ii) also follows directly from Theorem~\ref{thm:dist-rob-opt} because $\ell(\xi)=\ell_1(\xi)=  \min_{k\le K}a_j(\xi)$ is concave and thus satisfies Assumption~\ref{a:ell} for $J=1$. In this setting, we find
			\begin{align*}
			[-\ell]^*(z) &= \sup\limits_{\xi} \inner{z}{\xi} + \min_{k\le K}\Big\{ \inner{a_k}{\xi} + b_k\Big\} 
			= \left\{ \begin{array}{cl} \Sup{\xi,\tau} & \inner{z}{\xi} +\tau \\ \st & A{\xi} + b \geq \tau \One \end{array}\right. 
			= \left\{ \begin{array}{cl} \Inf{\theta \geq 0} & \inner{\theta}{b} \\ \st & A\tr \theta = -z \\ & \inner{\theta}{\One}  = 1 \end{array} \right.
			\end{align*}
			where the last equality follows again from strong linear programming duality, which holds since the primal maximization problem is feasible. Assertion~(ii) then follows by substituting $[-\ell]^*$ as well as the formula for $\sigma_\Xi$ from the proof of assertion~(i) into \eqref{eq:thm-dual:2}. 
	\end{proof}
	\end{Cor}

	 As a consistency check, we ascertain that in the {\em ambiguity-free limit}, the optimal value of \eqref{affine-max} reduces to the expectation of $\max_{k\le K}a_k(\xi)$ under the empirical distribution. Indeed, for $\eps = 0$, the variable $\lambda$ can be set to any positive value at no penalty. For this reason and because all training samples must belong to the uncertainty set ({\em i.e.}, $d-C\data_i\geq 0$ for all $i\le N$), it is optimal to set $\gamma_{ik}=0$. This in turn implies that $s_i= \max_{k\le K}a_k(\data_i)$ at optimality, in which case $\frac{1}{N}\sum_{i=1}^Ns_i$ represents the sample average of the convex loss function at hand. 

An analogous argument shows that, for $\eps=0$, the optimal value of \eqref{affine-min} reduces to the expectation of $\min_{k\le K}a_k(\xi)$ under the empirical distribution. As before, $\lambda$ can be increased at no penalty. Thus, we conclude that $\gamma_i=0$ and 
	\[
	s_i=\min\limits_{\theta_i\geq 0}\left\{\inner{\theta_i}{b+ A\data_i}:\inner{\theta_{i}}{\One} = 1 \right\} = \min_{k\le K}a_k(\data_i)
	\]
at optimality, in which case $\frac{1}{N}\sum_{i=1}^Ns_i$ is the sample average of the given concave loss function.

	\subsection{Uncertainty Quantification}
	\label{subsec:UQ}	

	A problem of great practical interest is to ascertain whether a physical, economic or engineering system with an uncertain state $\xi$ satisfies a number of safety constraints with high probability. In the following we denote by $\set{A}$ the set of states in which the system is safe. Our goal is to quantify the probability of the event $\xi\in\set A$ ($\xi\notin\set A$)  under an ambiguous state distribution that is only indirectly observable through a finite training dataset. More precisely, we aim to calculate the {\em worst-case} probability of the system being {\em unsafe}, {\em i.e.},
	\begin{subequations}
	\label{UQ}
	\begin{align}
	\label{worst-case-prob}
	\sup_{\Q \in \ball{\Pem}{\eps}} \Q \left[ \xi\notin \set{A}\right],
	\end{align}
as well as the {\em best-case} probability of the system being {\em safe}, that is,
	\begin{align}
	\label{best-case-prob}
	\sup_{\Q \in \ball{\Pem}{\eps}} \Q \left[ \xi\in \set{A}\right].
	\end{align}
	\end{subequations}

{ \begin{Rem}[Data-dependent sets]
		The set $\set A$ may even depend on the samples $\data_1,\ldots,\data_N$, in which case $\set A$ is renamed as $\wh{\set A}$. If the Wasserstein radius $\eps$ is set to $\eps_N(\beta)$, then we have $\PP\in \ball{\Pem}{\eps}$ with probability $1-\beta$, implying that \eqref{worst-case-prob} and \eqref{best-case-prob} still provide $1-\beta$ confidence bounds on $\PP[\xi\notin\wh{\set A}]$ and $\PP[\xi\in\wh{\set A}]$, respectively. 
		\end{Rem}}

	\begin{Cor}[Uncertainty quantification]
	\label{cor:chance}
		Suppose that the uncertainty set is a polytope of the form $\Xi = \{ \xi \in \R^m : C \xi \le d \}$ as in Corollary~\ref{cor:affine}. 	
		\begin{subequations}
		\label{chance}	
		\begin{itemize}
		\item[(i)] If $\set A = \{\xi \in \R^m: A\xi < b\}$ is an open polytope and the halfspace $\big\{\xi:\inner{a_k}{\xi}\geq b_k \big\}$ has a nonempty intersection with $\Xi$ for any $k\le K$, where $a_k$ is the $k$-th row of the matrix $A$ and $b_k$ is the $k$-th entry of the vector $b$, then the worst-case probability \eqref{worst-case-prob} is given by
		\begin{align}
		\label{chance-worst}
		\left\{
					\begin{array}{clll} \inf\limits_{\lambda,s_i, \gamma_{ik},\theta_{ik}} & \lambda \eps + {1 \over N}\sum\limits_{i = 1}^{N} s_i  \\
					\st &1-\theta_{ik}\big(b_k-\inner{a_k}{\data_i} \big) +\inner{\gamma_{ik}}{d-C\data_i} \le s_i & \forall i \le N, & \forall k \le K\\
								&  \|a_k\theta_{ik}-C\tr\gamma_{ik}\|_* \le \lambda & \forall i \le N, & \forall k \le K \\
								&  \gamma_{ik}\ge 0& \forall i \le N, & \forall k \le K\\
								&  \theta_{ik} \ge 0& \forall i \le N, & \forall k \le K\\
								& s_i \ge 0 &  \forall i \le N.
					\end{array}\right.
		\end{align}
		\item[(ii)] If $\set A = \{\xi  \in \R^m : A\xi \le b\}$ is a closed polytope that has a nonempty intersection with $\Xi$, then the best-case probability \eqref{best-case-prob} is given by
\begin{align}
		\label{chance-best}
		\left\{\begin{array}{clll} \inf\limits_{\lambda,s_i, \gamma_i, \theta_i} & \lambda \eps + {1 \over N}\sum\limits_{i = 1}^{N} s_i  \vspace{1mm} \\
				\st & 1+\inner{\theta_i}{b - A\data_i} + \inner{\gamma_{i}}{d - C\data_i} \le s_i & \forall i \le N \\
				& \|A\tr \theta_i+C\tr \gamma_{i}\|_* \le \lambda & \forall i \le N \\
				& \gamma_i \ge 0 & \forall i \le N\\
				& \theta_{i} \ge 0 & \forall i \le N\\
				& s_i\ge 0 & \forall i \le N.
			\end{array}\right.
		\end{align}
		\end{itemize}
		\end{subequations}
	\begin{proof}
		{The uncertainty quantification problems \eqref{worst-case-prob} and \eqref{best-case-prob} can be interpreted as instances of \eqref{dist-rob-Ex} with loss functions $\ell = 1 - \ind{\set A}$ and $\ell = \ind{\set A}$, respectively. In order to be able to apply Theorem~\ref{thm:dist-rob-opt}, we should represent these loss functions as finite maxima of concave functions as shown in Figure \ref{fig:ind}.}
		
		Formally, assertion~(i) follows from Theorem~\ref{thm:dist-rob-opt} for a loss function with $K+1$ pieces if we use the following definitions. For every $k\le K$ we define 
		$$\ell_{k}(\xi) =
		\left\{\begin{array}{cl} 1 & \text{if }\inner{a_k}{\xi} \ge b_k, \\ -\infty & \text{otherwise.} \end{array}\right.$$
		Moreover, we define $\ell_{K+1}(\xi) = 0$. {As illustrated in Figure~\ref{fig:ind:out}}, we thus have $\ell(\xi)=\max_{k\le K+1} \ell_k(\xi)= 1 - \ind{\set A}(\xi)$ and
		\[
		\sup\limits_{\Q \in \ball{\Pem}{\eps}} \Q \left[ \xi\notin \set{A}\right]~= \sup\limits_{\Q \in \ball{\Pem}{\eps}} \EE^\Q \left[\ell(\xi)\right].
		\]
		Assumption~\ref{a:ell} holds due to the postulated properties of $\set A$ and $\Xi$. In order to apply Theorem~\ref{thm:dist-rob-opt}, we must determine the support function $\sigma_\Xi$, which is already known from Corollary~\ref{cor:affine}, as well as the conjugate functions of $-\ell_k$, $k\le K+1$. A standard duality argument yields
		\begin{align*}
			[-\ell_k]^*(z) & = \left\{ \begin{array}{cl} \Sup{\xi} & \inner{z}{\xi} + 1 \\ \st & \inner{a_k}{\xi}\geq b_k  \end{array}\right. 
			= \left\{ \begin{array}{cl} \Inf{\theta \ge 0} & 1 - b_k\theta \\ \st & a_k \theta =-z, \end{array}\right. 
		\end{align*}
		for all $k\le K$. Moreover, we have {$[-\ell_{K+1}]^* = 0$ if $\xi=0$; $=\infty$ otherwise}. Assertion~(ii) then follows by substituting the formulas for $[-\ell_k]^*$, $k\le K+1$, and $\sigma_\Xi$ into \eqref{eq:thm-dual:2}. 
		
		Assertion~(ii) follows from Theorem~\ref{thm:dist-rob-opt} by setting $K= 2$, $\ell_1(\xi) =  1-\chi_{\set A}(\xi)$ and $\ell_2(\xi) = 0$. As illustrated in {Figure~\ref{fig:ind:in}}, this implies that $\ell(\xi)=\max\{\ell_1(\xi),\ell_2(\xi)\}=\ind{\set A}(\xi)$ and
		\[
		\sup\limits_{\Q \in \ball{\Pem}{\eps}} \Q \left[ \xi\in \set{A}\right]~= \sup\limits_{\Q \in \ball{\Pem}{\eps}} \EE^\Q \left[\ell(\xi)\right].
		\]
		Assumption~\ref{a:ell} holds by our assumptions on $\set A$ and $\Xi$. In order to apply Theorem~\ref{thm:dist-rob-opt}, we thus have to determine the support function $\sigma_\Xi$, which was already calculated in Corollary~\ref{cor:affine}, and the conjugate functions of $-\ell_1$ and $-\ell_2$. By the definition of the conjugacy operator, we find
		{
		\begin{align*}
			[-\ell_1]^*(z) &= \sup_{\xi \in \set{A}} \inner{z}{\xi} + 1 = \left\{ \begin{array}{cl} \Sup{\xi} & \inner{z}{\xi} + 1 \\ \st & A\xi \le b \end{array}\right. 
			= \left\{ \begin{array}{cl} \Inf{\theta_k \ge 0} & \inner{\theta}{b} + 1 \\ \st & A\tr \theta = z \end{array}\right.
		\end{align*}
		}where the last equality follows from strong linear programming duality, which holds as the safe set is non-empty. Similarly, we find {$[-\ell_{2}]^* = 0$ if $\xi=0$; $=\infty$ otherwise}. Assertion~(ii) then follows by substituting the above expressions into \eqref{eq:thm-dual:2}. 
	\end{proof}
	\end{Cor}

		\begin{figure}[t!]
		\centering
			\subfigure[Indicator function of the unsafe set]{\label{fig:ind:out}\includegraphics[scale = 0.55]{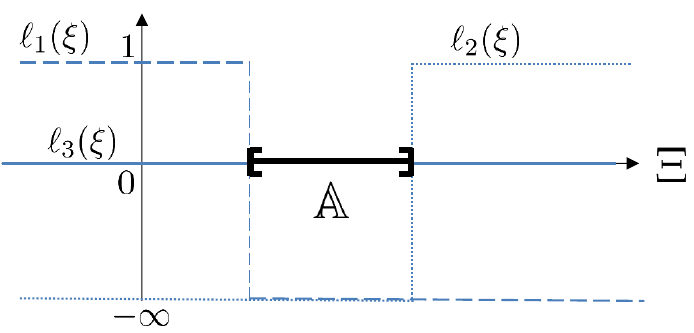}} \qquad
			\subfigure[Indicator function of the safe set]{\label{fig:ind:in}\includegraphics[scale = 0.55]{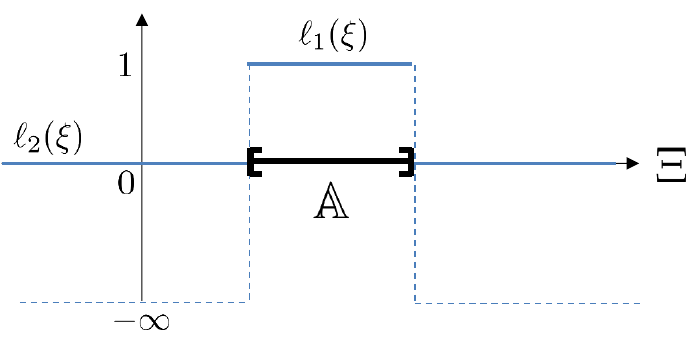}} 
		\caption{Representing the indicator function of {a convex set and its complement as a pointwise maximum of concave functions}}
		\label{fig:ind}
		\end{figure}

	 In the \emph{ambiguity-free limit} ({\em i.e.}, for $\eps = 0$) the optimal value of \eqref{chance-worst} reduces to the fraction of training samples residing outside of the open polytope $\set A=\{\xi:A\xi <b\}$. Indeed, in this case the variable $\lambda$ can be set to any positive value at no penalty. For this reason and because all training samples belong to the uncertainty set ({\em i.e.}, $d-C\data_i\geq 0$ for all $i\le N$), it is optimal to set $\gamma_{ik}=0$. If the $i$-th training sample belongs to $\set A$ ({\em i.e.}, $b_k-\inner{a_k}{\data_i}> 0$ for all $k\le K$), then $\theta_{ik}\geq1/(b_k-\inner{a_k}{\data_i})$ for all $k\le K$ and $s_i=0$ at optimality. Conversely, if the $i$-th training sample belongs to the complement of $\set A$, ({\em i.e.}, $b_k-\inner{a_k}{\data_i}\leq 0$ for some $k\le K$), then $\theta_{ik}=0$ for some $k\le K$ and $s_i=1$ at optimality. Thus, $\sum_{i=1}^Ns_i$ coincides with the number of training samples outside of $\set A$ at optimality. An analogous argument shows that, for $\eps=0$, the optimal value of \eqref{chance-best} reduces to the fraction of training samples residing inside of the closed polytope $\set A=\{\xi:A\xi \leq b\}$.

	\subsection{Two-Stage Stochastic Programming} 

	A major challenge in linear two-stage stochastic programming is to evaluate the expected recourse costs, which are only implicitly defined as the optimal value of a linear program whose coefficients depend linearly on the uncertain problem parameters \cite[Section~2.1]{ref:Shap&Dent&Rusz}. The following corollary shows how we can evaluate the worst-case expectation of the recourse costs with respect to an ambiguous parameter distribution that is only observable through a finite training dataset. For ease of notation and without loss of generality, we suppress here any dependence on the first-stage decisions.

	\begin{Cor}[Two-stage stochastic programming]
	\label{cor:2stage}
		Suppose that the uncertainty set is a polytope of the form $\Xi = \{ \xi \in \R^m : C \xi \le d \}$ as in Corollaries~\ref{cor:affine} and \ref{cor:chance}. 	
		\begin{subequations}
		\label{2stage}	
		\begin{itemize}
		\item[(i)] If $\ell(\xi) =\inf_{y} \left\{ \inner{y}{Q\xi} : Wy\ge h \right\}$ is the optimal value of a parametric linear program with objective uncertainty, and if the feasible set $\{y:Wy\ge h\}$ is non-empty and compact, then the worst-case expectation \eqref{dist-rob-Ex} is given by
		\begin{align}
		\label{objective-uncertainty}
		\left\{
			\begin{array}{clll} \inf\limits_{\lambda,s_i, \gamma_i, y_i} & \lambda \eps + {1 \over N}\sum\limits_{i = 1}^{N} s_i  \vspace{1mm} \\
				\st & \inner{y_i}{Q\data_i} + \inner{\gamma_{i}}{d - C\data_i} \le s_i & \forall i \le N \\
				& Wy_i\ge h & \forall i \le N\\
				& \|Q\tr y_i-C\tr \gamma_{i}\|_* \le \lambda & \forall i \le N \\
				& \gamma_i \ge 0 & \forall i \le N.
			\end{array}\right.
		\end{align}
		\item[(ii)] If $\ell(\xi) =\inf_{y} \left\{ \inner{q}{y} : Wy \ge H\xi + h \right\}$ is the optimal value of a parametric linear program with right-hand side uncertainty, and if the dual feasible set $\{\theta\ge 0:W\tr\theta=q\}$ is non-empty and compact with vertices $v_k$, $k\le K$, {then} the worst-case expectation \eqref{dist-rob-Ex} is given~by
		\begin{align}
		\label{rhs-uncertainty}
		\left\{
					\begin{array}{clll} \inf\limits_{\lambda,s_i, \gamma_{ik}} & \lambda \eps + {1 \over N}\sum\limits_{i = 1}^{N} s_i  \\
					\st & \inner{v_k}{h} + \inner{H\tr v_k}{\data_i}+ \inner{\gamma_{ik}}{d-C\data_i} \le s_i & \forall i \le N, & \forall k \le K\\
								& \|C\tr\gamma_{ik}-H\tr v_k\|_* \le \lambda & \forall i \le N, & \forall k \le K  \\
								& \gamma_{ik} \ge 0& \forall i \le N, & \forall k \le K.
					\end{array}\right.
		\end{align}
		\end{itemize}
		\end{subequations}
		\begin{proof}
		Assertion~(i) follows directly from Theorem~\ref{thm:dist-rob-opt} because $\ell(\xi)$ is concave as an infimum of linear functions in $\xi$. Indeed, the compactness of the feasible set $\{y: Wy\ge h\}$ ensures that Assumption~\ref{a:ell} holds for $K=1$. In this setting, we find
		\begin{align*}
			[-\ell]^*(z)&  = \sup\limits_{\xi} \left\{ \inner{z}{\xi} + \inf\limits_{y} \left\{  \inner{y}{Q\xi} : Wy\ge h\right\}\right\} \\
			&= \inf\limits_{y}\left\{  \sup\limits_{\xi}\left\{ \inner{z+Q\tr y}{\xi} \right\} : Wy\ge h \right\}\\
			&=  \left\{\begin{array}{cl} 0 & \text{if there exists $y$ with } Q\tr y=-z \text{ and }Wy\ge h,\\ \infty & \text{otherwise,} \end{array} \right.
		\end{align*}
		where the second equality follows from the classical minimax theorem \cite[Proposition 5.5.4]{ref:Bert-09}, which applies because $\{y: Wy\ge h\}$ is compact. Assertion~(i) then follows by substituting $[-\ell]^*$ as well as the formula for $\sigma_\Xi$ from Corollary~\ref{cor:affine} into \eqref{eq:thm-dual:2}. 

	Assertion~(ii) relies on the following reformulation of the loss function,		
	\begin{align*}
	\ell(\xi) &= \left\{ \begin{array}{cl} \inf\limits_{y} & \inner{q}{y} \\ \st & Wy \ge H\xi + h \end{array}\right. 
	= \left\{ \begin{array}{cl} \sup\limits_{\theta\geq 0} & \inner{\theta}{H\xi + h} \\ \st & W\tr \theta = q \end{array}\right. = \max\limits_{k\le K}  \inner{v_k}{H\xi + h} \\ &= \max\limits_{k\le K}  \inner{H\tr v_k}{\xi} + \inner{v_k}{h},
	\end{align*}
	where the first equality holds due to strong linear programming duality, which applies as the dual feasible set is non-empty. The second equality exploits the elementary observation that the optimal value of a linear program with non-empty, compact feasible set is always adopted at a vertex. As we managed to express $\ell(\xi)$ as a pointwise maximum of linear functions, assertion~(ii) follows immediately from Corllary~\ref{cor:affine}~(i).
	\end{proof}
	\end{Cor}
	
	As expected, in the {\em ambiguity-free limit}, problem~\eqref{objective-uncertainty} reduces to a standard SAA problem. Indeed, for $\eps=0$, the variable $\lambda$ can be made large at no penalty, and thus $\gamma_i=0$ and $s_i=\inner{y_i}{Q\data_i}$ at optimality. In this case, problem~\eqref{objective-uncertainty} is equivalent to
		\begin{align*}
			\inf\limits_{y_i} \left\{ {1 \over N}\sum\limits_{i = 1}^{N}  \inner{y_i}{Q\data_i} :  Wy_i\ge h \quad\forall i \le N\right\}.
		\end{align*}
Similarly, one can verify that for $\eps=0$, \eqref{rhs-uncertainty} reduces to the SAA problem
		\begin{align*}
			\inf\limits_{y_i} \left\{ {1 \over N}\sum\limits_{i = 1}^{N}  \inner{y_i}{q} :  Wy_i\ge H\data_i \quad\forall i \le N\right\}.
		\end{align*}

We close this section with a remark on the computational complexity of all the {convex optimization problems derived} in this section.

\begin{Rem}[Computational tractability]
\label{rem:comp}
\hfil
\begin{itemize}
	\item If the Wasserstein metric is defined in terms of the $1$-norm (i.e., $\|\xi\|=\sum_{k=1}^m|\xi_k|$) or the $\infty$-norm (i.e., $\|\xi\|=\max_{k\le m}|\xi_k|$), then the optimization problems \eqref{affine-max}, \eqref{affine-min}, \eqref{chance-worst}, \eqref{chance-best}, \eqref{objective-uncertainty} and \eqref{rhs-uncertainty} all reduce to linear programs whose sizes scale with the number $N$ of data points and the number $J$ of affine pieces of the underlying loss functions. 
	
	\item Except for the two-stage stochastic program with right-hand side uncertainty in \eqref{rhs-uncertainty}, the resulting linear programs scale polynomially in the problem description and are therefore computationally tractable. As the number of vertices $v_k$, $k\le K$, of the polytope $\{\theta\ge 0:W\tr\theta=q\}$ may be exponential in the number of its facets, however, the linear program \eqref{rhs-uncertainty} has generically exponential size. 
	
	\item {Inspecting~\eqref{affine-max}, one easily verifies that the distributionally robust optimization problem~\eqref{DRO} reduces to a finite convex program if $\X$ is convex and $h(x,\xi)= \max_{k\le K} \inner{a_{k}(x)}{\xi} + b_{k}(x)$, while the gradients $a_{k}(x)$ and the intercepts $b_{k}(x)$ depend linearly on $x$. Similarly, \eqref{DRO} can be reformulated as a finite convex program if $\X$ is convex and $h(x,\xi)=\inf_{y} \left\{ \inner{y}{Q\xi} : Wy\ge h(x) \right\}$ or $h(x,\xi)=\inf_{y} \left\{ \inner{q}{y} : Wy \ge H(x)\xi + h(x) \right\}$, while the right hand side coefficients $h(x)$ and $H(x)$ depend linearly on $x$; see \eqref{objective-uncertainty} and \eqref{rhs-uncertainty}, respectively. In contrast, problems \eqref{affine-min}, \eqref{chance-worst} and \eqref{chance-best} result in non-convex optimization problems when their data depends on $x$.}
	
	\item We emphasize that the computational complexity of all convex programs examined in this section is independent of the radius $\eps$ of the Wasserstein ball.
\end{itemize}
\end{Rem}

	\section{Tractable Extensions}
	\label{sec:exten}
	We now demonstrate that through minor modifications of the proofs, Theorems~\ref{thm:dist-rob-opt} and \ref{thm:dual-dual} extend to worst-case expectation problems involving even richer classes of loss functions. First, we investigate problems where the uncertainty can be viewed as a stochastic process and where the loss function is additively separable. Next, we study problems whose loss functions are convex in the uncertain variables {and are therefore} not necessarily representable as finite maxima of concave functions as postulated by Assumption~\ref{a:ell}.  

	\subsection{Stochastic Processes with a Separable Cost}
	Consider a variant of the worst-case expectation problem~\eqref{dist-rob-Ex}, where the uncertain parameters can be interpreted as a stochastic process {$\xi = \big(\xi_1,\ldots,\xi_T\big)$}, and assume that $\xi_t \in \Xi_t$, where $ \Xi_t  \subseteq \R^m$ is non-empty and closed for any $t\le T$. Moreover, assume that the loss function is additively separable with respect to the temporal structure of $\xi$, that is, 
	\begin{align}
	\label{proc_cost}
		\ell(\xi) \Let \sum\limits_{t = 1}^{T} \max_{k\le K}\ell_{tk} \big(\xi_t\big), 
	\end{align}
where $\ell_{tk}:\R^m\ra \Ru$ is a measurable function for any $k\le K$ and $t\le T$. Such loss functions appear, for instance, in open-loop stochastic optimal control or in multi-item newsvendor problems. Consider a process norm $\normT{\xi} = \sum_{t = 1}^{T} \|\xi_t\|$ associated with the base norm $\|\cdot\|$ on $\R^m$, and assume that its induced metric is the one used in the definition of the Wasserstein distance. Note that if $\|\cdot\|$ is the 1-norm on $\R^m$, then $\normT{\cdot}$ reduces to the 1-norm on $\R^{mT}$. 

By interchanging summation and maximization, the loss function~\eqref{proc_cost} can be re-expressed as
	\begin{align*}
		\ell(\xi)=  \max_{k_t \le K} \sum\limits_{t = 1}^{T} \ell_{tk_t} \big(\xi_t \big), 
	\end{align*}
where the maximum runs over all $K^T$ combinations of $k_1,\ldots, k_T\le K$. Under this representation, Theorem~\ref{thm:dist-rob-opt} remains applicable. However, the resulting convex optimization problem would involve $\mathcal O(K^T)$ decision variables and constraints, indicating that an efficient solution may not be available. Fortunately, this deficiency can be overcome by modifying Theorem~\ref{thm:dist-rob-opt}.
	
	\begin{Thm}[Convex reduction for separable loss functions]
		\label{thm:dist-rob-opt-separable}
		Assume that the loss function $\ell$ is of the form \eqref{proc_cost}, and the Wasserstein ball is defined through the process norm $\normT{\cdot}$. Then, for any $\eps \ge0 $, the worst-case expectation~\eqref{dist-rob-Ex} is smaller or equal to the optimal value of the  finite convex program
		\begin{align}
		\label{eq:thm-dual_process:1}
		\left\{
			\begin{array}{cllll} \inf\limits_{\lambda, s_{ti}, z_{tik}, \nu_{tik}} & \lambda \eps + {1 \over N}\sum\limits_{i = 1}^{N} \sum\limits_{t = 1}^{T} s_{ti} && \\
					\st & [-\ell_{tk}]^*\big(z_{tik} - \nu_{tik}\big) + \sigma_{\Xi_t}(\nu_{tik}) - \inner{z_{tik}}{\data_{ti}} \le s_{ti} & \forall i \le N, & \forall k\le K, & \forall t \le T,\\
					& \|z_{tik}\|_* \le \lambda &\forall i \le N, & \forall k\le K, & \forall t \le T.
					\end{array}\right.
		\end{align} 
If $\Xi_t$ and $\{\ell_{tk}\}_{k\le K}$ satisfy the convexity Assumption~\ref{a:ell} for every $t\le T$, then the worst-case expectation~\eqref{dist-rob-Ex} coincides exactly with the optimal value of problem~\eqref{eq:thm-dual_process:1}.
	\end{Thm}

	\begin{proof}
	Up until equation~\eqref{eq:2}, the proof of Theorem~\ref{thm:dist-rob-opt-separable} parallels that of Theorem~\ref{thm:dist-rob-opt}. Starting from \eqref{eq:2}, we then have
	\begin{align}
		\sup\limits_{\Q \in \ball{\Pem}{\eps}} 
			& \EE^\Q \big[ \ell(\xi) \big]  = \inf\limits_{\lambda \ge 0}~ \lambda \eps + {1 \over N}\sum\limits_{i = 1}^{N} \sup_{\xi} \left (\ell(\xi) - \lambda \normT{\xi - \data_i} \right ) \notag \\
			& = \inf\limits_{\lambda \ge 0} ~ \lambda \eps + {1 \over N} \sum\limits_{i = 1}^{N} \sum\limits_{t = 1}^{T} \sup_{\xi_t \in \Xi_t} \left (\max_{k\le K} \ell_{tk}\big(\xi_t \big) - \lambda \big \|\xi_t - \data_{ti}\big\| \right ) ,\notag 
			\end{align}
where the interchange of the summation and the maximization is facilitated by the separability of the overall loss function. Introducing epigraphical auxiliary variables yields	
			\begin{align}
			&\left\{ 
			\begin{array}{cllll} \inf\limits_{\lambda, s_{ti}} & \lambda \eps + {1 \over N}\sum\limits_{i = 1}^{N} \sum\limits_{t=1}^{T} s_{ti} \\
			\st & \sup\limits_{\xi_t\in \Xi_t} \Big(\ell_{tk}\big(\xi_t\big) - \lambda \big\|\xi_t - \data_{ti} \big\| \Big)\le s_{ti} & \forall i\le N, ~ \forall k\le K, ~ \forall t \le T \\ 							& \lambda \ge 0 
			\end{array}
			\right. \notag \\
			 \le& \left\{ 
			\begin{array}{cllll} \inf\limits_{\lambda, s_{ti},z_{tik}} & \lambda \eps + {1 \over N}\sum\limits_{i = 1}^{N} \sum\limits_{t = 1}^{T}s_{ti} \\
			\st & \sup\limits_{\xi_t\in \Xi_t} \Big(\ell_{tk}\big(\xi_t \big) - \inner{z_{tik}}{\xi_t} \Big) + \inner{z_{tik}}{\data_{ti}}\le s_{ti} &\forall i\le N, ~\forall k\le K, ~\forall t \le T \\
						& \|z_{tik}\|_* \le \lambda & \forall i\le N,~ \forall k\le K, ~ \forall t \le T
			\end{array}
			\right. \notag \\
			=  & \left\{ 
			\begin{array}{cllll} \inf\limits_{\lambda, s_{ti},z_{tik}} & \lambda \eps + {1 \over N}\sum\limits_{i = 1}^{N} \sum\limits_{t = 1}^{T} s_{ti} \\
			\st & [-\ell_{tk} + \indI{\Xi_t}]^*\big(-z_{tik}\big) + \inner{z_{tik}}{\data_{ti}}\le s_{ti}  &\forall i\le N,~ \forall k\le K,~  \forall t \le T \\
			& \|z_{tik}\|_* \le \lambda & \forall i\le N, ~ \forall k\le K, ~ \forall t \le T,
			\end{array}
			\right. \notag 
	\end{align} 
	where the inequality is justified in a similar manner as the one in \eqref{eq:3}, and it holds as an equality provided that $\Xi_t$ and $\{\ell_{tk}\}_{k\le K}$ satisfy Assumption~\ref{a:ell} for all $t \le T$. {Finally, by \cite[Theorem~11.23(a), p.~493]{ref:Rockafellar-10}, the conjugate of $-\ell_{tk} + \indI{\Xi_t}$ can be replaced by the inf-convolution of the conjugates of $-\ell_{tk}$ and $\indI{\Xi_t}$. This completes the proof.}
	\end{proof}

	Note that the convex program~\eqref{eq:thm-dual_process:1} involves only $\Oo(NKT)$ decision variables and constraints. Moreover, if $\ell_{tk}$ is affine for every $t\le T$ and $k\le K$, while $\|\cdot\|$ represents the $1$-norm or the $\infty$-norm on $\R^m$, then \eqref{eq:thm-dual_process:1} reduces to a tractable linear program (see also Remark~\ref{rem:comp}). A natural generalization of Theorem \ref{thm:dual-dual} further allows us to characterize the extremal distributions of the worst-case expectation problem~\eqref{dist-rob-Ex} with a separable loss function of the form \eqref{proc_cost}.

	\begin{Thm}[Worst-case distributions for separable loss functions]
		\label{prop:process-dual-dual-separable}
		Assume that the loss function $\ell$ is of the form \eqref{proc_cost}, and the Wasserstein ball is defined through the process norm $\normT{\cdot}$. If $\Xi_t$ and $\{\ell_{tk}\}_{k\le K}$ satisfy Assumption~\ref{a:ell} for all $t \le T$, then the worst-case expectation~\eqref{dist-rob-Ex} coincides with the optimal value of the finite convex program
		\begin{align}
		\label{dual-dual-separable}
		\left\{
			 \begin{array}{clll} \Sup{\alpha_{tik}, q_{tik}} & {1 \over N} \sum\limits_{i = 1}^{N} \sum\limits_{k = 1}^{K} \sum\limits_{t=1}^{T} \alpha_{tik}\ell_{tk}\Big( \data_{ti} - {q_{tik} \over \alpha_{tik}}\Big) \\
			\st & {1 \over N}\sum\limits_{i =1}^{N} \sum\limits_{k =1}^{K} \sum\limits_{t=1}^{T} \|q_{tik}\| \le \eps\\
						& \sum\limits_{k = 1}^{K} \alpha_{tik} = 1 &\forall i \le N, \quad \forall t \le T\\
						& \alpha_{tik} \ge 0 &\forall i \le N, \quad \forall t \le T, \quad \forall k\le K \\
						& \data_{ti} - {q_{tik} \over \alpha_{tik}} \in \Xi_t  &\forall i \le N, \quad \forall t \le T, \quad \forall k\le K
			\end{array} \right.
		\end{align}
		irrespective of $\eps \ge0 $. Let $\big\{\alpha_{tik}(r), q_{tik}(r)\big\}_{r \in \N}$ be a sequence {of} feasible decisions whose objective values converge to the supremum of \eqref{dual-dual-separable}. Then, the discrete (product) probability distributions 
		\begin{align*}
			\Q_r \Let  {1 \over N} \sum_{i = 1}^{N} \bigotimes_{t=1}^T \Big(\sum_{k = 1}^{K} \alpha_{tik}(r) \dir{\xi_{tik}(r)}\Big) \quad\mbox{with}\quad \xi_{tik}(r) \Let \data_{ti} - {q_{tik}(r) \over \alpha_{tik}(r)}
		\end{align*}
		belong to the Wasserstein ball $\ball{\Pem}{\eps}$ and attain the supremum of \eqref{dist-rob-Ex} asymptotically, i.e., 
		\begin{align*}
			\sup\limits_{\Q \in \ball{\Pem}{\eps}} \EE^\Q \big[ \ell(\xi) \big] =  \lim\limits_{r \ra \infty} \EE^{\Q_r} \big[ \ell(\xi) \big] = 
			\lim\limits_{r \ra \infty} {1 \over N} \sum\limits_{i = 1}^{N} \sum\limits_{k=1}^{K} \sum\limits_{t=1}^{T} \alpha_{tik}(r) \ell_{tk}\big(\xi_{tik}(r)\big) .
		\end{align*} 
	\end{Thm}
	\begin{proof} As in the proof of Theorem~\ref{thm:dual-dual}, the claim follows by dualizing the convex program~\eqref{eq:thm-dual_process:1}. Details are omitted for brevity of exposition. \end{proof}

We emphasize that the distributions $\Q_r$ from Theorem~\ref{prop:process-dual-dual-separable} can be constructed efficiently by solving a convex program of polynomial size even though they have $NK^T$ discretization points.

\subsection{Convex Loss Functions}
 
	Consider now another variant of the worst-case expectation problem~\eqref{dist-rob-Ex}, where the loss function $\ell$ is proper, convex and lower semicontinuous. Unless $\ell$ is piecewise affine, we cannot represent such a loss function as a pointwise maximum of finitely many concave functions, and thus Theorem~\ref{thm:dist-rob-opt} may only provide a loose upper bound on the worst-case expectation~\eqref{dist-rob-Ex}. The following theorem provides an alternative upper bound that admits new insights into distributionally robust optimization with Wasserstein balls and becomes exact for $\Xi=\R^m$.

	\begin{Thm}[Convex reduction for convex loss functions]
	\label{thm:convex}
		Assume that the loss function $\ell$ is proper, convex, and lower semicontinuous, and define $\kappa \Let \sup\big\{ \|\theta\|_* : \ell^*(\theta) < \infty \big \}$. Then, for any $\eps \ge0 $, the worst-case expectation~\eqref{dist-rob-Ex} is smaller or equal to 
		\begin{align}
			\label{kappa_bound}
			\kappa \eps + {1 \over N}\sum_{i = 1}^{N} \ell(\data_i).
		\end{align}
		If $\Xi=\R^m$, then the worst-case expectation~\eqref{dist-rob-Ex} coincides exactly with \eqref{kappa_bound}.
	\end{Thm}
	
	\begin{Rem}[Radius of effective domain]
	\label{rem:kappa}
	The parameter $\kappa$ can be viewed as the radius of the smallest ball containing the effective domain of the conjugate function $\ell^*$ in terms of the dual norm. By the standard conventions of extended arithmetic, the term $\kappa\eps$ in \eqref{kappa_bound} is interpreted as $0$ if $\kappa=\infty$ and $\eps=0$.
	\end{Rem}
	\begin{proof}
Equation~\eqref{eq:1-} in the proof of Theorem \ref{thm:dist-rob-opt} implies that
		\begin{align}
			\label{convex_loss}
		\sup\limits_{\Q \in \ball{\Pem}{\eps}} \EE^\Q \big[ \ell(\xi) \big] = \inf\limits_{\lambda \ge 0} ~\lambda \eps + {1 \over N}\sum\limits_{i = 1}^{N} \sup_{\xi \in \Xi} \left (\ell(\xi) - \lambda \|\xi - \data_i\| \right )
		\end{align}
		for every $\eps > 0$. As $\ell$ is proper, convex, and lower semicontinuous, it coincides with its bi-conjugate function $\ell^{**}$, see {\em e.g.}\ \cite[Proposition 1.6.1(c)]{ref:Bert-09}. Thus, we may write
	\begin{align*}
		\ell(\xi) = \sup_{\theta \in \Theta} \inner{\theta}{\xi} - \ell^*(\theta),
	\end{align*}
	where $\Theta \Let \{\theta\in\R^m : \ell^*(\theta) < \infty\}$ denotes the effective domain of the conjugate function $\ell^*$. Using this dual representation of $\ell$ in conjunction with the definition of the dual norm, we find
	\begin{align*}
		\Sup{\xi\in \Xi} \Big(\ell(\xi) - \lambda \|\xi-\data_i\|\Big) &= \Sup{\xi\in \Xi}~ \Sup{\theta\in \Theta} \Big(\inner{\theta}{\xi} - \ell^*(\theta) - \lambda \|\xi-\data_i\|\Big) \\
		& = \Sup{\xi\in \Xi}~ \Sup{\theta\in \Theta} \Inf{\|z\|_* \le \lambda}\Big(\inner{\theta}{\xi} - \ell^*(\theta) + \inner{z}{\xi} - \inner{z}{\data_i}\Big). 
	\end{align*}
	The classical minimax theorem \cite[Proposition 5.5.4]{ref:Bert-09} then allows us to interchange the maximization over $\xi$ with the maximization over $\theta$ and the minimization over $z$ to obtain
	\begin{align}
		\notag 
		\Sup{\xi\in \Xi} \Big(\ell(\xi) - \lambda \|\xi-\data_i\|\Big) & = \Sup{\theta\in \Theta} \Inf{\|z\|_* \le \lambda}  \Sup{\xi\in \Xi}\Big(\inner{\theta + z}{\xi} - \ell^*(\theta) - \inner{z}{\data_i}\Big) \\
		\label{sigma}
		& = \Sup{\theta\in \Theta} \Inf{\|z\|_* \le \lambda} \sigma_{\Xi}(\theta + z) - \ell^*(\theta) - \inner{z}{\data_i}.
	\end{align}
	Recall that $\sigma_\Xi$ denotes the support function of $\Xi$. It seems that there is no simple exact reformulation of \eqref{sigma} for arbitrary convex uncertainty sets $\Xi$. Interchanging the maximization over $\theta$ with the minimization over $z$ in \eqref{sigma} would lead to the conservative upper bound {of Corollary \ref{cor:approx}}. Here, however, we employ an alternative approximation. By definition of the support function, we have $\sigma_\Xi\leq \sigma_{\R^m} = \indI{\{0\}}$. Replacing $\sigma_\Xi$ with $ \indI{\{0\}}$ in \eqref{sigma} thus results in the conservative approximation
	\begin{align}
	\label{eq:ell-convex}
		\Sup{\xi\in \R^m} \Big(\ell(\xi) - \lambda \|\xi-\data_i\|\Big) &\leq 
		\left \{
		\begin{array}{cl}
			\ell(\data_i) &  \mbox{if }\sup \big\{\|\theta\|_* : \theta \in \Theta \big \} \le \lambda, \\
			\infty & \mbox{otherwise.}
		\end{array}
		\right.
	\end{align}
The inequality \eqref{kappa_bound} then follows readily by substituting \eqref{eq:ell-convex} into \eqref{convex_loss} and using the definition of $\kappa$ in the theorem statement. For $\Xi=\R^m$ we have $\sigma_\Xi= \indI{\{0\}}$, and thus the upper bound \eqref{kappa_bound} becomes exact. Finally, if $\eps=0$, then \eqref{dist-rob-Ex} trivially coincides with \eqref{kappa_bound} under our conventions of extended arithmetic. Thus, the claim follows.
	\end{proof}

	Theorem~\ref{thm:convex} asserts that for $\Xi=\R^m$, the worst-case expectation \eqref{dist-rob-Ex} of a convex loss function reduces the sample average of the loss adjusted by the simple correction term $\kappa\eps$. The following proposition highlights that $\kappa$ can be interpreted as a measure of maximum steepness of the loss function. This interpretation has intuitive appeal in view of Definition~\ref{def:wass}.

	\begin{Prop}[Steepness of the loss function]
	\label{prop:kappa}
		Let $\kappa$ be defined as in Theorem~\ref{thm:convex}. 
		\begin{enumerate}[label=(\roman*), itemsep = 1mm, topsep = 1mm] 
			\item \label{itm:lem:upper} 
			If $\ell$ is $\ol{L}$-Lipschitz continuous, i.e., if there exists $\xi' \in \R^m$ such that $\ell(\xi) - \ell(\xi') \le \ol{L}\|\xi-\xi'\|$ for all $\xi \in \R^m$, then~$\kappa \le \ol{L}$. 
			\item \label{itm:lem:lower}
			If $\ell$ majorizes an affine function, i.e., if there exists $\theta\in \R^m$ with $\|\theta\|_*=:\ul L$ and $\xi' \in \R^m$ such that $\ell(\xi) - \ell(\xi') \ge \inner{\theta}{\xi-\xi'}$ for all $\xi \in \R^m$, then $\kappa \ge \ul{L} $.
		\end{enumerate}
	\end{Prop}
	
	\begin{proof}
		The proof follows directly from the definition of conjugacy. As for \ref{itm:lem:upper}, we have 
		\begin{align*}
			\ell^*(\theta) = \sup_{\xi \in \R^m} \inner{\theta}{\xi} - \ell(\xi)~& \ge \sup_{\xi \in \R^m} \inner{\theta}{\xi} - \ol{L} \|\xi -\xi'\| -\ell(\xi')\\
			& = \sup_{\xi \in \R^m} \inf_{\|z\|_*\le \ol{L}} \inner{\theta}{\xi} - \inner{z}{\xi -\xi'}- \ell(\xi'),
		\end{align*}
		where the last equality follows from the definition of the dual norm. Applying the minimax theorem \cite[Proposition 5.5.4]{ref:Bert-09} and explicitly carrying out the maximization over $\xi$ yields
		\begin{align*}
			\ell^*(\theta) \ge 
			\left\{ 
			\begin{array}{cl}
				\inner{\theta}{\xi'}-\ell(\xi') & \mbox{if } \|\theta\|_* \le \ol{L}, \\
				\infty  & \mbox{otherwise.}
			\end{array}\right.
		\end{align*}
		Consequently, $\ell^*(\theta)$ is infinite for all $\theta$ with $\|\theta\|_*> \ol L$, which readily implies that the $\|\cdot\|_*$-ball of radius $\ol L$ contains the effective domain of $\ell^*$. Thus, $\kappa \le \ol{L}$. 
		
		As for \ref{itm:lem:lower}, we have 
		\begin{align*}
		\ell^*(\theta) = \sup_{\xi \in \R^m} \inner{\theta}{\xi} - \ell(\xi) & \le \sup_{\xi \in \R^m} \inner{\theta}{\xi} - \inner{z}{\xi-\xi'} - \ell(\xi') \\
		& = \sigma_{\R^m}(\theta - z)+ \inner{z}{\xi'} - \ell(\xi'), 
		\end{align*}
		which implies that $\ell^*(\theta) \le \inner{\theta}{\xi'} - \ell(\xi') < \infty$. Thus, $\theta$ belongs to the effective domain of $\ell^*$. We then conclude that $\kappa \ge \|\theta\|_* = \ul{L}$.
	\end{proof}
	
{ \begin{Rem}[Consistent formulations]
If $\Xi=\R^m$ and the loss function is given by $\ell(\xi) = \max_{k \le K}\{\inner{a_{k}}{\xi} + b_{k}\}$, then both Corollary~\ref{cor:affine} and Theorem~\ref{thm:convex} offer an exact reformulation of the worst-case expectation~\eqref{dist-rob-Ex} in terms of a finite-dimensional convex program. On the one hand, Corollary~\ref{cor:affine} implies that \eqref{dist-rob-Ex} is equivalent~to
\begin{align*}
\left\{
\begin{array}{clll} \Min{\lambda} & \lambda \eps + {1 \over N}\sum\limits_{i = 1}^{N} \ell(\data_i)\\
\st & \|a_k\|_* \le \lambda & \forall k \le K,
\end{array} \right.
\end{align*}
which is obtained by setting $C=0$ and $d=0$ in~\eqref{affine-max}. At optimality we have $\lambda^\star=\max_{k\le K} \|a_k\|_*$, which corresponds to the (best) Lipschitz constant of $\ell(\xi)$ with respect to the norm~$\|\cdot\|$. On the other hand, Theorem~\ref{thm:convex} implies that \eqref{dist-rob-Ex} is equivalent to \eqref{kappa_bound} with $\kappa=\lambda^\star$. Thus, Corollary~\ref{cor:affine} and Theorem~\ref{thm:convex} are consistent.
\end{Rem}}

{ \begin{Rem}[$\eps$-insensitive optimizers\footnote{{We are indepted to Vishal Gupta who has brought this interesting observation to our attention.}}]
Consider a loss function $h(x,\xi)$ that is convex in $\xi$, and assume that $\Xi=\R^m$. In this case Theorem~\ref{thm:convex} remains valid, but the steepness parameter $\kappa(x)$ may depend on $x$. For loss functions whose Lipschitz modulus with respect to $\xi$ is independent of $x$ (e.g., the newsvendor loss), however, $\kappa(x)$ is constant. In this case the distributionally robust optimization problem~\eqref{DRO} and the SAA problem~\eqref{Ex_emp} share the same minimizers irrespective of the Wasserstein radius~$\eps$. This phenomenon could explain why the SAA solutions tend to display a surprisingly strong out-of-sample performance in these problems.
\end{Rem}}

\section{Numerical Results}
\label{sec:num}
We validate the theoretical results of this paper in the context of a stylized portfolio selection problem. The subsequent simulation experiments are designed to provide additional insights into the performance guarantees of the proposed distributionally robust optimization scheme.

\subsection{Mean-Risk Portfolio Optimization}

Consider a capital market consisting of $m$ assets whose yearly returns are captured by the random vector $\xi = [\xi_1, \ldots, \xi_m]\tr$. If short-selling is forbidden, a portfolio is encoded by a vector of percentage weights $x=[x_1,\ldots,x_m]\tr$ ranging over the probability simplex $\X=\{x\in\mathbb R^m_+: \sum_{i=1}^{m}x_i = 1\}$. As portfolio $x$ invests a percentage $x_i$ of the available capital in asset $i$ for each $i=1,\ldots,m$, its return amounts to $\inner{x}{\xi}$. In the remainder we aim to solve the single-stage stochastic program
\begin{align}
    \J = \inf_{x\in\X} \bigg\{\EE^{\PP}\big[-\inner{x}{\xi}\big] + \rho\, \PP\text{-}\cvar_{\alpha}\big(-\inner{x}{\xi}\big) \bigg\},
    \label{portfolio}
\end{align}
which minimizes a weighted sum of the mean and the conditional value-at-risk (CVaR) of the portfolio loss $-\inner{x}{\xi}$, where $\alpha\in (0,1]$ is referred to as the confidence level of the CVaR, and $\rho\in\R_+$ quantifies the investor's risk-aversion. Intuitively, the CVaR at level $\alpha$ represents the average of the $\alpha\times 100\%$ worst (highest) portfolio losses under the distribution $\PP$. Replacing the CVaR in the above expression with its formal definition \cite{ref:Rock&Yry-00}, we obtain   
\begin{align*}
    \J & = \inf_{x\in\X}\bigg\{\EE^{\PP}\big[-\inner{x}{\xi}\big] + \rho\,\inf_{\tau\in\R} \EE^\PP \Big[ \tau + {1\over \alpha} \max\big\{ -\inner{x}{\xi} - \tau, 0 \big \}\Big] \bigg\} \\
    & =  \inf_{x\in\X, \tau\in\R} \EE^\PP \Big[ \max_{k\le K} \, a_k\inner{ x}{\xi}+b_k \tau \Big],
\end{align*}
where $K=2$, $a_1= -1$, $a_2= -1-\frac{\rho}{\alpha}$, $b_1=\rho$ and $b_2= \rho(1-\frac{1}{\alpha})$. An investor who is unaware of the distribution $\PP$ but has observed a dataset $\Xiem$ of $N$ historical samples from $\PP$ and knows that the support of $\PP$ is contained in $\Xi=\{\xi\in \R^m:C\xi\leq d\}$ might solve the distributionally robust counterpart of \eqref{portfolio} with respect to the Wasserstein ambiguity set $\ball{\Pem}{\eps}$, that is,
\begin{align*}
    {\Jdd(\eps)} \Let  \inf_{x\in\X, \tau\in\R} \sup\limits_{\Q \in \ball{\Pem}{\eps}} \EE^\Q \Big[ \max_{k\le K} \, a_k \inner{ x}{\xi}+b_k \tau \Big],
\end{align*}
{where we make the dependence on the Wasserstein radius $\eps$ explicit.} By Corollary~\ref{cor:affine} we know that
\begin{align}
    \label{dro:portfolio}
    {\Jdd (\eps)} = \left\{
    \begin{array}{lclll} & \inf\limits_{x,\tau,\lambda,s_i, \gamma_{ik}} & \lambda \eps + {1 \over N}\sum\limits_{i = 1}^{N} s_i  \\
    & \st & x\in\X\\
        & & b_k\tau +a_k\inner{x}{\data_i}+ \inner{\gamma_{ik}}{d-C\data_i} \le s_i & \forall i \le N, & \forall k \le K\\
        & & \|C\tr\gamma_{ik} - a_{k}x\|_* \le \lambda & \forall i \le N, & \forall k \le K  \\
        & & \gamma_{ik} \ge 0& \forall i \le N, & \forall k \le K.
    \end{array}
    \right.
\end{align}
Before proceeding with the numerical analysis of this problem, we provide some analytical insights into its optimal solutions when there is significant ambiguity. In what follows we keep the training data set fixed and let $\xdd(\eps)$ be an optimal distributionally robust portfolio corresponding to the Wasserstein ambiguity set of radius $\eps$. 
We will now show that, for natural choices of the ambiguity set, $\xdd(\eps)$ converges to the equally weighted portfolio $\frac{1}{m}e$ as $\eps$ tends to infinity, where $e \Let (1,\ldots,1)\tr$. The optimality of the equally weighted portfolio under high ambiguity has first been demonstrated in \cite{ref:PfluRichWoz-11} using analytical methods. We identify this result here as an immediate consequence of Theorem~\ref{thm:dist-rob-opt}, which is primarily a computational result.

For any non-empty set $S\subseteq \R^m$ we denote by $\mbox{recc}(S) \Let \{y\in\R^m:x+\lambda y\in S~\forall x\in S, ~\forall \lambda\geq 0\}$ the recession cone and by $S^\circ \Let \{y\in \R^m:\inner{y}{x}\leq 0~\forall x\in S\}$ the polar cone of $S$.

\begin{Lem}
\label{polarrecession}
If $\{\eps_k\}_{k\in\mathbb N}\subset \R_+$ tends to infinity, then any accumulation point $x^\star$ of $\big\{\xdd(\eps_k)\big\}_{k\in\mathbb N}$ is a portfolio that has minimum distance to $(\mbox{\em recc}(\Xi))^\circ$ with respect to $\|\cdot\|_*$. 
\end{Lem}
\begin{proof}
Note first that $\xdd(\eps_k)$, $k\in\mathbb N$, and $x^\star$ exist because $\X$ is compact. For large Wasserstein radii $\eps$, the term $\lambda\eps$ dominates the objective function of problem~\eqref{dro:portfolio}. Using standard epi-convergence results \cite[Section~7.E]{ref:Rockafellar-10}, one can thus show that
\begin{align*}
    x^\star\, & \in \arg\min_{x\in\X}~ \min_{\gamma_{ik}\geq 0}~ \max_{i\le N,\, k\le K}  \|C\tr\gamma_{ik} - a_{k}x\|_*\\
    & = \arg\min_{x\in\X}~ \max_{i\le N,\, k\le K}~ \min_{\gamma\geq 0} ~ \|C\tr\gamma + |a_{k}|\,x\|_* \\
    & = \arg\min_{x\in\X} ~ \min_{\gamma\geq 0} ~ \|C\tr\gamma + x\|_* ~ \max_{k\le K} |a_k| \\
    & = \arg\min_{x\in\X} ~ \min_{\gamma\geq 0} ~ \|C\tr\gamma + x\|_*,
\end{align*}
where the first equality follows from the fact that $a_k<0$ for all $k\le K$, the second equality uses the substitution $\gamma\rightarrow \gamma |a_k|$, and the last equality holds because the set of minimizers of an optimization problem is not affected by a positive scaling of the objective function. Thus, $x^\star$ is the portfolio nearest to the cone $\mathcal C=\{C\tr\gamma:\gamma\geq 0\}$. The claim now follows as the polar cone
\begin{align*}
\mathcal C^\circ \Let \{y\in\R^m:y\tr x\leq 0~\forall x\in\mathcal C\}= \{y\in\R^m:y\tr C\tr \gamma \leq 0~\forall \gamma\geq 0\}= \{y\in\R^m: Cy\geq 0\}
\end{align*}
is readily recognized as the recession cone of $\Xi$ and as $\mathcal C=(\mathcal C^\circ)^\circ$.
\end{proof}

\begin{Prop}[Equally weighted portfolio]
\label{prop:1/n}
Assume that the Wasserstein metric is defined in terms of the $p$-norm in the uncertainty space for some $p\in[1,\infty)$. If $\{\eps_k\}_{k\in\mathbb N}\subset \R_+$ tends to infinity, then $\big\{\xdd(\eps_k)\big\}_{k\in\mathbb N}$ converges to the equally weighted portfolio $x^\star=\frac{1}{m}e$ provided that the uncertainty set is given by
\begin{itemize}
\item[(i)] the entire space, i.e., $\Xi=\R^m$, or
\item[(ii)] the nonnegative orthant shifted by $-e$, i.e., $\Xi=\{\xi\in\R^m:\xi \geq-e\}$, which captures the idea that no asset can lose more than $100\%$ of its value.
\end{itemize}
\end{Prop}
\begin{proof}
(i) One easily verifies from the definitions that $(\mbox{recc}(\Xi))^\circ=\{0\}$. Moreover, we have $\|\cdot\|_*=\|\cdot\|_q$ where $\frac{1}{p}+\frac{1}{q}=1$. As $p\in [1,\infty)$, we conclude that $q\in (1,\infty]$, and thus the unique nearest portfolio to $(\mbox{recc}(\Xi))^\circ$ with respect to $\|\cdot\|_*$ is $x^\star=\frac{1}{m}e$. The claim then follows from Lemma \ref{polarrecession}. Assertion~(ii) follows in a similar manner from the observation that  $(\mbox{recc}(\Xi))^\circ$ is now the non-positive orthant. 
\end{proof}

{With some extra effort one can show that for every $p\in[1,\infty)$ there is a threshold $\bar \eps>0$ with $\xdd(\eps)=x^\star$ for all $\eps\geq \bar \eps$, see~\cite[Proposition~3]{ref:PfluRichWoz-11}. Moreover, for $p\in\{1,2\}$ the threshold $\bar \eps$ is known analytically.}

\subsection{Simulation Results: Portfolio Optimization}

\label{sec:simulation}

{Our experiments are based on a market with $m=10$ assets considered in~\cite[Section~7.5]{ref:BertSAA-14}.} In view of the capital asset pricing model we may assume that the return $\xi_i$ is decomposable into a systematic risk factor $\psi\sim \mathcal N(0,2\%)$ common to all assets and an unsystematic or idiosyncratic risk factor $\zeta_i\sim\mathcal N(i\times 3\%, i\times 2.5\%)$ specific to asset $i$. Thus, we set $\xi_i=\psi+\zeta_i$, where $\psi$ and the idiosyncratic risk factors $\zeta_i$, $i=1,\ldots,m$, constitute independent normal random variables. By construction, assets with higher indices promise higher mean returns at a higher risk. Note that the given moments of the risk factors completely determine the distribution $\PP$ of~$\xi$. This distribution has support $\Xi=\R^m$ and satisfies Assumption~\ref{a:exp} for the tail exponent $a=1$, say. We also set $\alpha=20\%$ and $\rho=10$ in all numerical experiments, and we use the $1$-norm to measure distances in the uncertainty space. Thus, $\|\cdot\|_*$ is the $\infty$-norm, whereby \eqref{dro:portfolio} reduces to a linear program.

\begin{figure*} [t]
	\centering
	\subfigure[$N=30$ training samples]{\label{fig:x:1} \includegraphics[width=0.31\columnwidth]{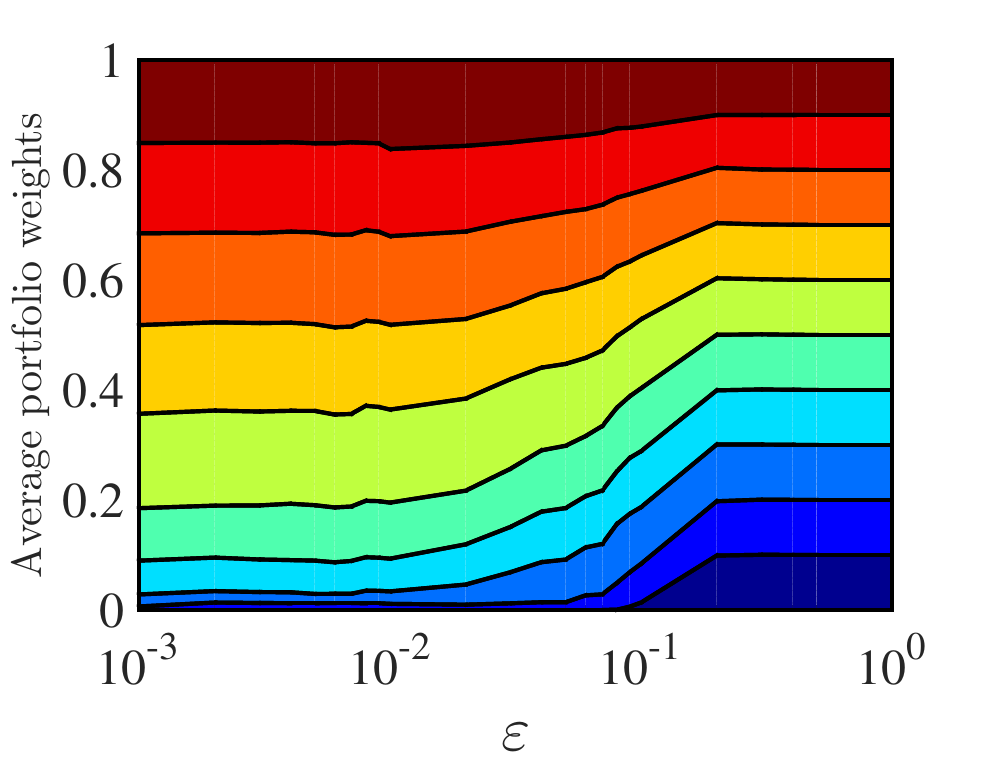}} \hspace{0mm}
	\subfigure[$N=300$ training samples]{\label{fig:x:2} \includegraphics[width=0.31\columnwidth]{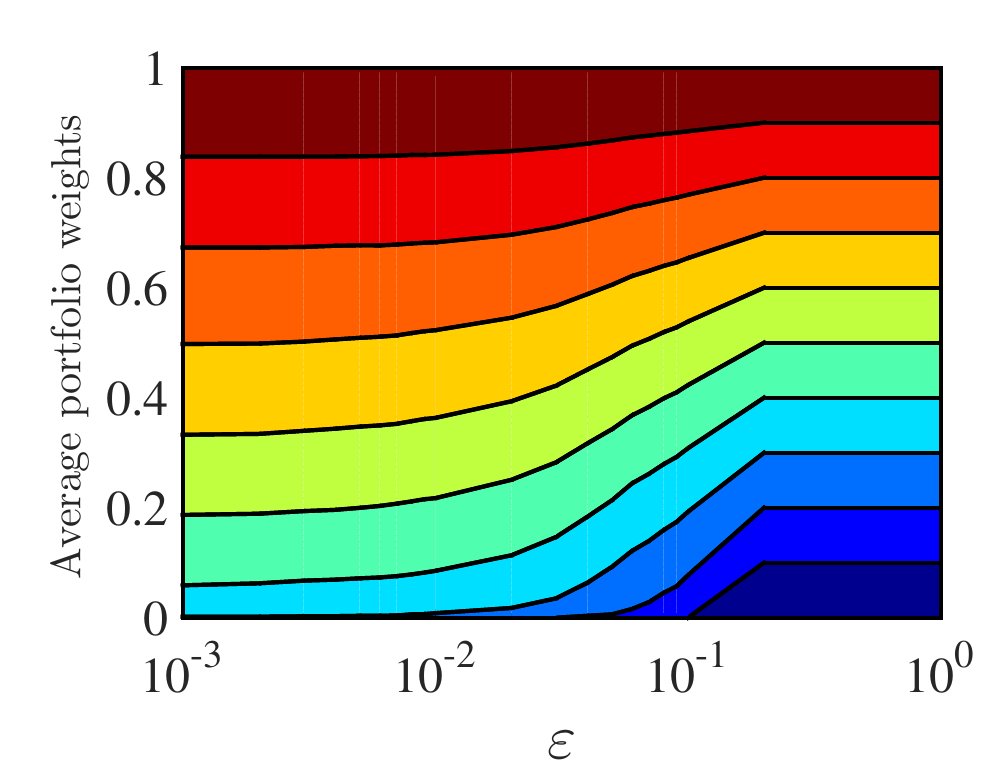}} \hspace{0mm}
	\subfigure[$N=3000$ training samples]{\label{fig:x:3} \includegraphics[width=0.31\columnwidth]{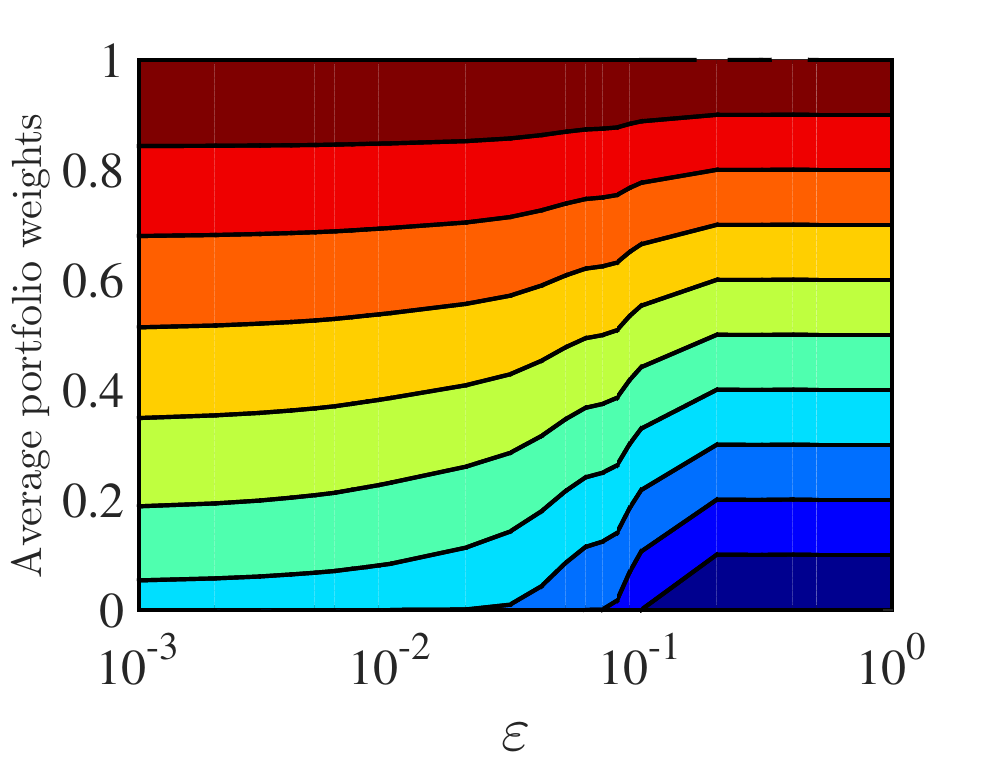}} \hspace{0mm}
	\caption{Optimal portfolio composition as a function of the Wasserstein radius~$\eps$ averaged over 200 simulations; the portfolio weights are depicted in ascending order, {\em i.e.}, the weight of asset~1 at the bottom (dark blue area) and that of asset~10 at the top (dark red area)}
	\label{fig:x}
\end{figure*}

	\subsubsection{\bf Impact of the Wasserstein Radius}
In the first experiment we investigate the impact of the Wasserstein radius $\eps$ on the optimal distributionally robust portfolios and their out-of-sample performance. We solve problem \eqref{dro:portfolio} using training datasets of cardinality $N \in \{30, 300, 3000\}$. Figure~\ref{fig:x} visualizes the corresponding optimal portfolio weights $\xdd(\eps)$ as a function of $\eps$, averaged over $200$ independent simulation runs. Our numerical results confirm the theoretical insight of Proposition~\ref{prop:1/n} that the optimal distributionally robust portfolios converge to the equally weighted portfolio as the Wasserstein radius $\eps$ increases; see also~\cite{ref:PfluRichWoz-11}.

The out-of-sample performance
\begin{align*}
J\big(\xdd(\eps)\big) \Let \EE^{\PP}\big[-\inner{\xdd(\eps)}{\xi}\big] + \rho\, \PP\text{-}\cvar_{\alpha}\big(-\inner{\xdd(\eps)}{\xi}\big)
\end{align*}
of any fixed distributionally robust portfolio $\xdd(\eps)$ can be computed analytically as $\PP$ constitutes a normal distribution by design, see, {\em e.g.}, \cite[p.~29]{ref:Rock&Yry-00}. 
Figure~\ref{fig:perf:eps} shows the tubes between the 20\% and 80\% quantiles (shaded areas) and the means (solid lines) of the out-of-sample performance $J\big(\xdd(\eps)\big)$ as a function of $\eps$---estimated using $200$ independent simulation runs. We observe that the out-of-sample performance improves (decreases) up to a critical Wasserstein radius $\eps_{\rm crit}$ and then deteriorates (increases). This stylized fact was observed consistently across all of simulations and provides an empirical justification for adopting a distributionally robust approach.

Figure~\ref{fig:perf:eps} also visualizes the reliability of the performance guarantees offered by our distributionally robust portfolio model. Specifically, the dashed lines represent the empirical probability of the event $J\big(\xdd(\eps)\big) \le \Jdd(\eps)$ with respect to $200$ independent training datasets. We find that the reliability is nondecreasing in $\eps$. This observation has intuitive appeal because $\Jdd(\eps) \ge J(\xdd(\eps))$ whenever $\PP\in \ball{\Pem}{\eps}$, and the latter event becomes increasingly likely as $\eps$ grows. Figure~\ref{fig:perf:eps} also indicates that the certificate guarantee sharply rises towards 1 near the critical Wasserstein radius $\eps_{\rm crit}$. Hence, the out-of-sample performance of the distributionally robust portfolios improves as long as the reliability of the performance guarantee is noticeably smaller than 1 and deteriorates when it saturates at 1. Even though this observation was made consistently across all simulations, we were unable to validate it theoretically.

\begin{figure*}[t]
	\centering
	\subfigure[$N=30$ training samples]{\label{fig:perf:1} \includegraphics[width=0.32\columnwidth]{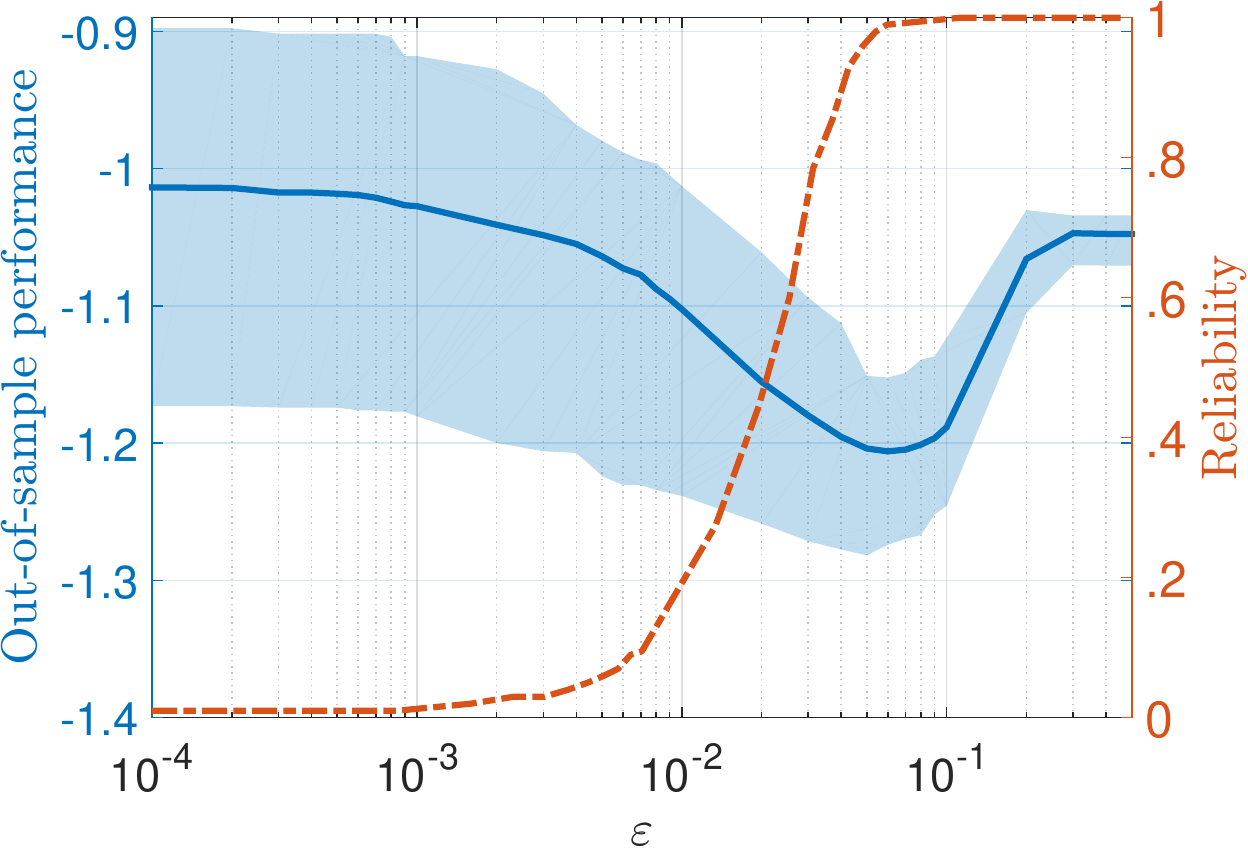}}
	\subfigure[$N=300$ training samples]{\label{fig:perf:2} \includegraphics[width=0.32\columnwidth]{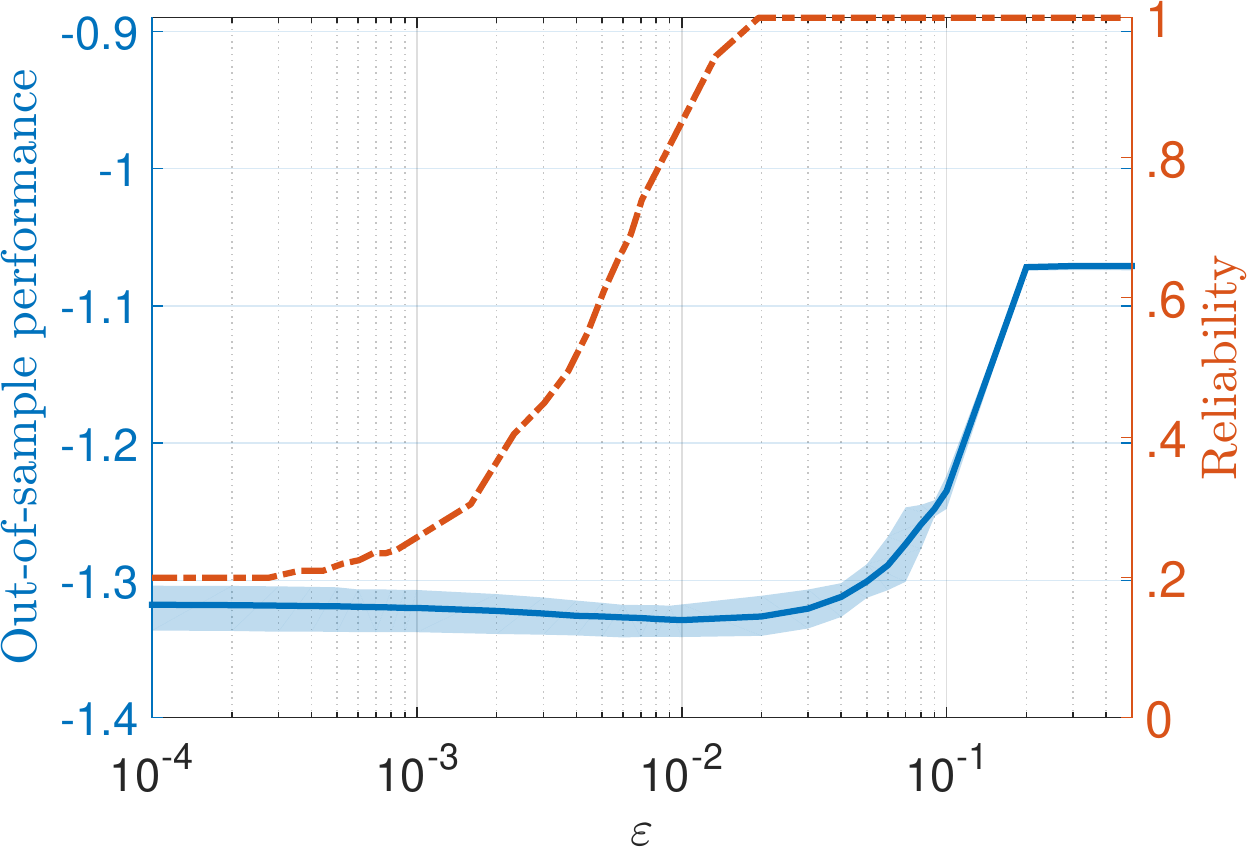}}
	\subfigure[$N=3000$ training samples]{\label{fig:perf:3} \includegraphics[width=0.32\columnwidth]{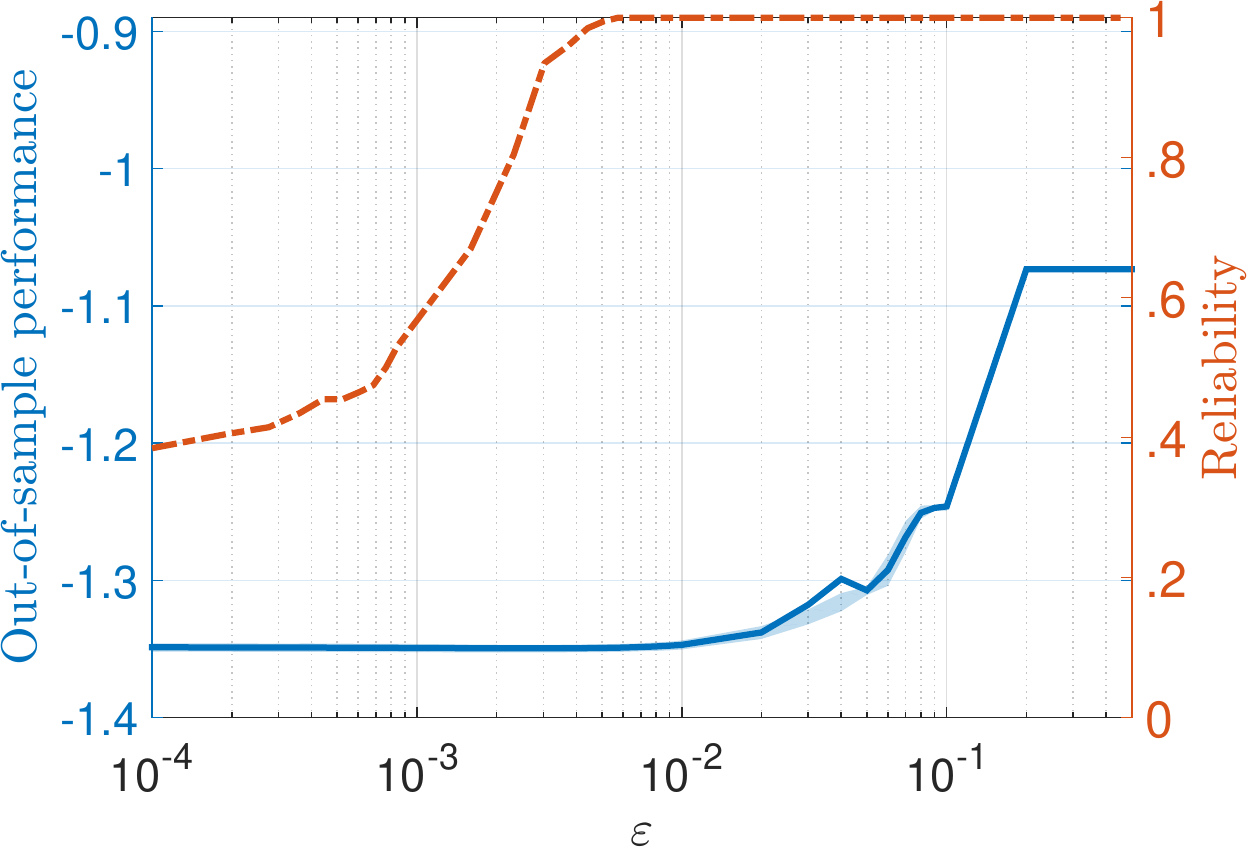}} 
	\caption{Out-of-sample performance $J(\xdd(\eps)) $ (left axis, solid line and shaded area) and reliability $\PP^N[J(\xdd(\eps)) \le \Jdd(\eps)]$ (right axis, dashed line) as a function of the Wasserstein radius $\eps$ and estimated on the basis of 200 simulations}
	\label{fig:perf:eps}
\end{figure*}

\subsubsection{\bf Portfolios Driven by Out-of-Sample Performance}
\label{subsub:sim:N}

Different Wasserstein radii $\eps$ may result in robust portfolios $\xdd(\eps)$ with vastly different out-of-sample performance $J(\xdd(\eps))$. Ideally, one should select the radius $\wh \eps_N^{\rm \; opt}$ that minimizes $J(\xdd(\eps))$ over all $\eps\geq 0$; note that $\wh \eps_N^{\rm \; opt}$ inherits the dependence on the training data from $J(\xdd(\eps))$. As the true distribution $\PP$ is unknown, however, it is impossible to evaluate and minimize $J(\xdd(\eps))$. In practice, the best we can hope for is to approximate $\wh \eps_N^{\rm \; opt}$ using the training data. Statistics offers several methods to accomplish this goal:
\begin{itemize}
	\item {\em Holdout method:} Partition $\data_1,\ldots,\data_N$ into a training dataset of size $N_T$ and a validation dataset of size $N_V=N-N_T$. Using only the training dataset, solve~\eqref{dro:portfolio} for a large but finite number of candidate radii $\eps$ to obtain $\wh x_{N_T}(\eps)$. Use the validation dataset to estimate the out-of-sample performance of $\wh x_{N_T}(\eps)$ via the sample average approximation. Set $\wh \eps_N^{\rm \; hm}$ to any $\eps$ that minimizes this quantity. Report $\xdd^{\rm \; hm}=\wh x_{N_T}(\wh \eps_N^{\rm \; hm})$ as the data-driven solution and $\Jdd^{\rm \; hm}=\wh J_{N_T}(\wh \eps_N^{\rm \; hm})$ as the corresponding certificate. 
\item {\em $k$-fold cross validation:} Partition $\data_1,\ldots,\data_N$ into $k$ subsets, and run the holdout method $k$ times. In each run, use exactly one subset as the validation dataset and merge the remaining $k-1$ subsets to a training dataset. Set $\wh \eps_N^{\rm \; cv}$ to the average of the Wasserstein radii obtained from the $k$ holdout runs. Resolve~\eqref{dro:portfolio} with $\eps=\wh \eps_N^{\rm \; cv}$ using all $N$ samples, and report $\xdd^{\rm \; cv}=\xdd(\wh \eps_N^{\rm \; cv})$ as the data-driven solution and $\Jdd^{\rm \; cv}=\Jdd(\wh \eps_N^{\rm \; cv})$ as the corresponding certificate. 
\end{itemize}
The holdout method is computationally cheaper, but cross validation has superior statistical properties. There are several other methods to estimate the best Wassertein radius $\wh \eps_N^{\rm \; opt}$. By construction, however, no method can provide a radius $\wh \eps_N$ such that $\xdd(\wh \eps_N)$ has a better out-of-sample performance than $\xdd(\wh \eps_N^{\rm \; opt})$.

In all experiments we compare the distributionally robust approach based on the Wasserstein ambiguity set with the classical sample average approximation (SAA) and with a state-of-the-art data-driven distributionally robust approach, where the ambiguity set is defined via a linear-convex ordering (LCX)-based goodness-of-fit test \cite[Section~3.3.2]{ref:BertSAA-14}. The size of the LCX ambiguity set is determined by a single parameter, which should be tuned to optimize the out-of-sample performance. While the best parameter value is unavailable, it can again be estimated using the holdout method or via cross validation. To our best knowledge, the LCX approach represents the only existing data-driven distributionally robust approach for {\em continuous} uncertainty spaces that enjoys strong finite-sample guarantees, asymptotic consistency as well as computational tractability.\footnote{{Much like worst-case expectations over Wasserstein balls, worst-case expectations over LCX ambiguity sets can be reformulated as finite convex programs whenever the underlying loss function represents a pointwise maximum of $K$ concave component functions. Unlike problem~\eqref{eq:thm-dual:2} in Theorem~\ref{thm:dist-rob-opt}, however, the resulting convex program scales exponentially with $K$.}}

To keep the computational burden manageable, in all experiments we select the Wasserstein radius as well as the LCX size parameter from within the discrete set $\mathcal E=\{\eps=b\cdot 10^c:b\in\{0,\ldots,9\},\; c\in\{-3,-2,-1\}\}$ instead of $\mathbb R_+$. We have verified that refining or extending $\mathcal E$ has only a marginal impact on our results, which indicates that $\mathcal E$ provides a sufficiently rich approximation of~$\R_+$.

\begin{figure*} [t]
	\centering
	\subfigure[Holdout method]{\label{fig:perf-val} 
		\includegraphics[width=0.31\columnwidth]{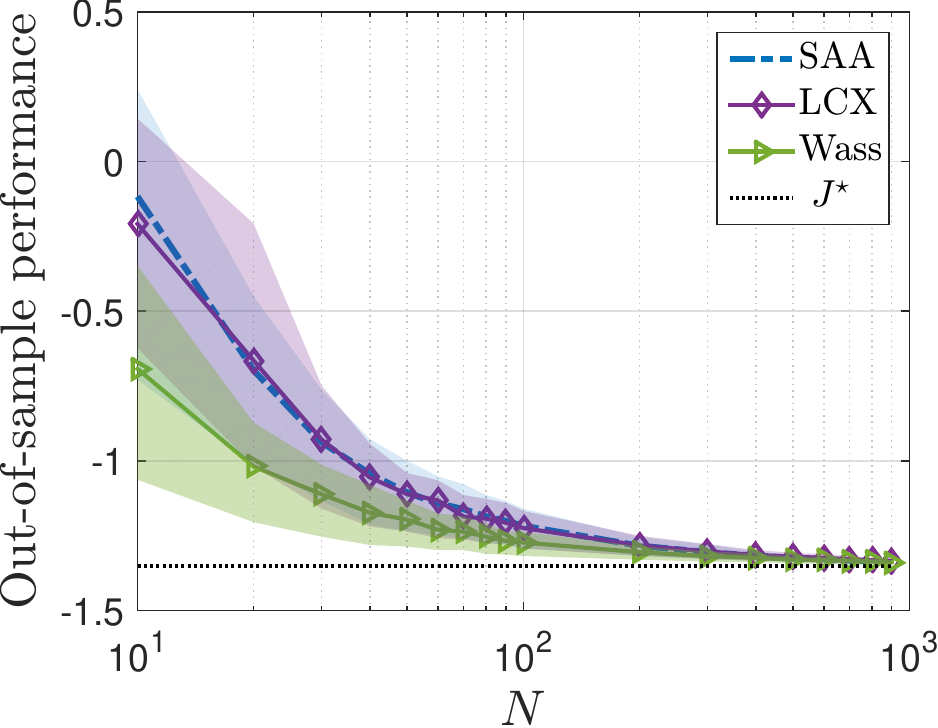}} \hspace{1mm}
	\subfigure[Holdout method]{\label{fig:cert-val}
		\includegraphics[width=0.31\columnwidth]{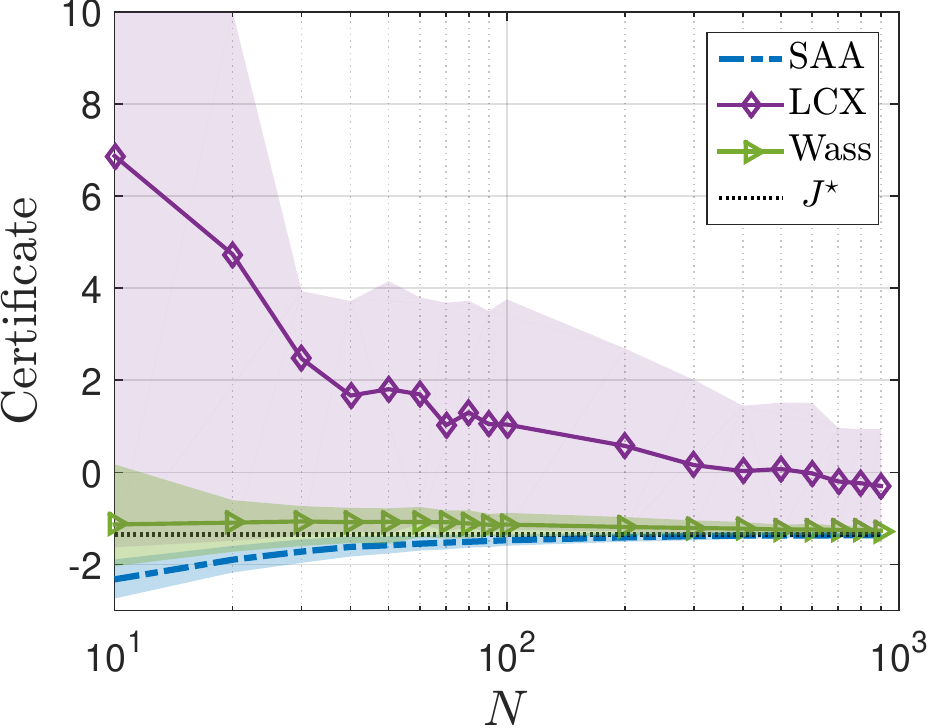}} \hspace{1mm}
	\subfigure[Holdout method]{\label{fig:beta-val}
		\includegraphics[width=0.31\columnwidth]{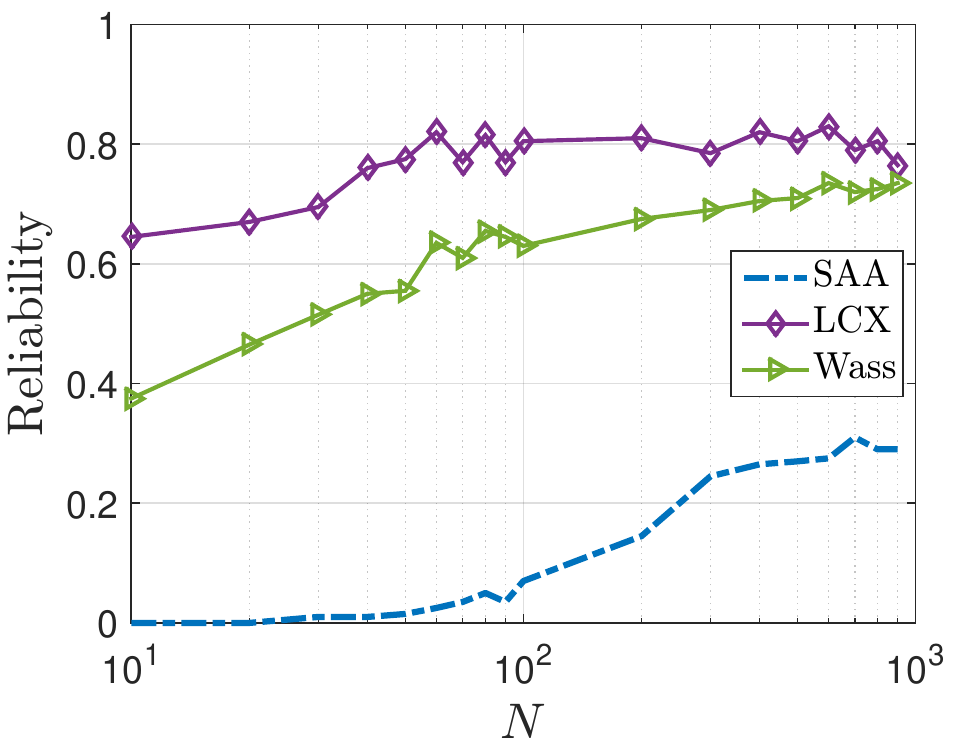}} \hspace{1mm}	
	\subfigure[$k$-fold cross validation]{\label{fig:perf-k}
		\includegraphics[width=0.31\columnwidth]{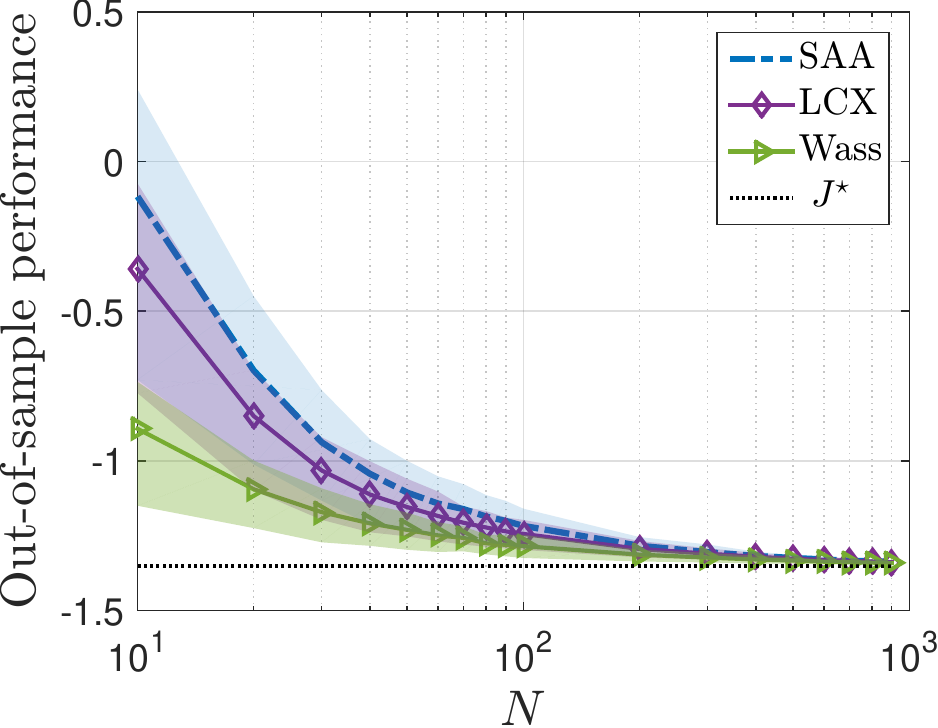}} \hspace{1mm}
	\subfigure[$k$-fold cross validation]{\label{fig:cert-k} 
		\includegraphics[width=0.31\columnwidth]{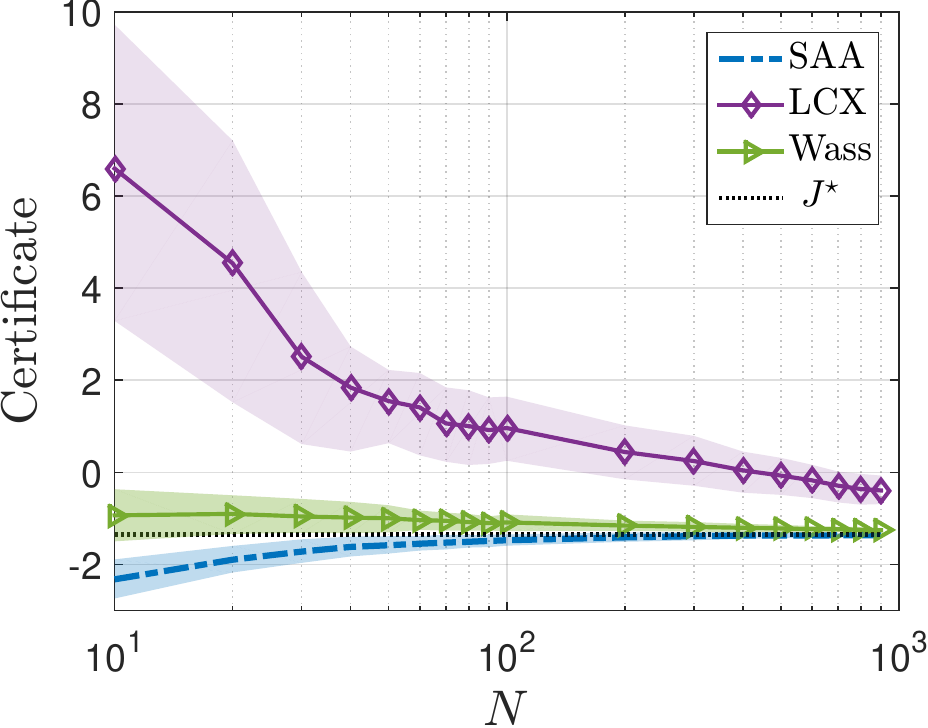}} \hspace{1mm}
	\subfigure[$k$-fold cross validation]{\label{fig:beta-k} 
		\includegraphics[width=0.31\columnwidth]{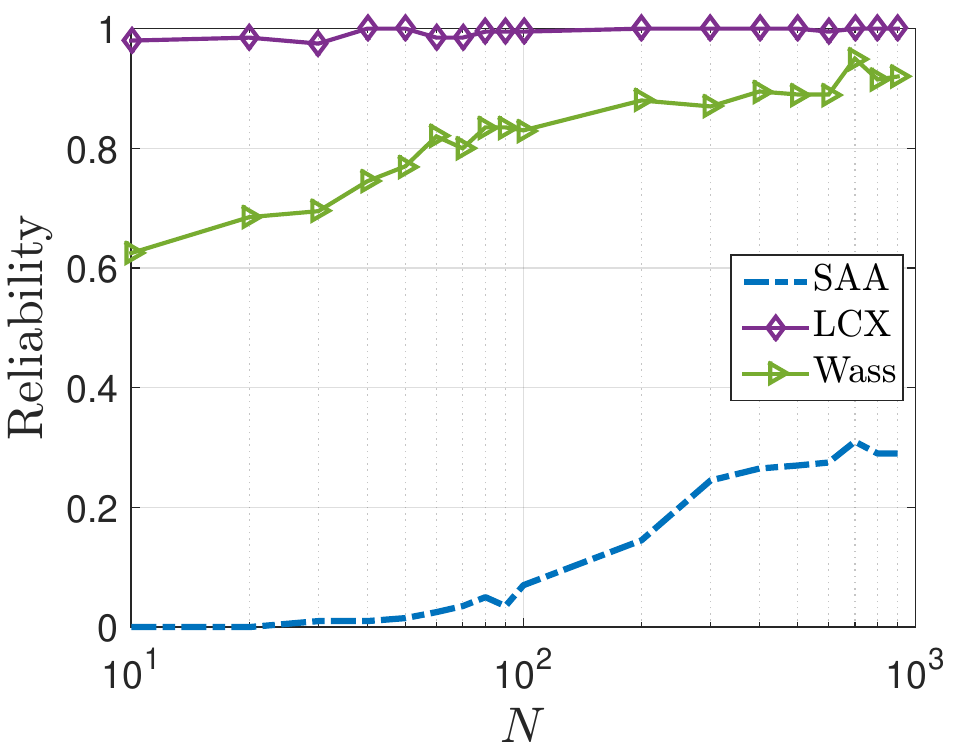}} \hspace{1mm}	
	\subfigure[Optimal size]{\label{fig:perf-best}
		\includegraphics[width=0.31\columnwidth]{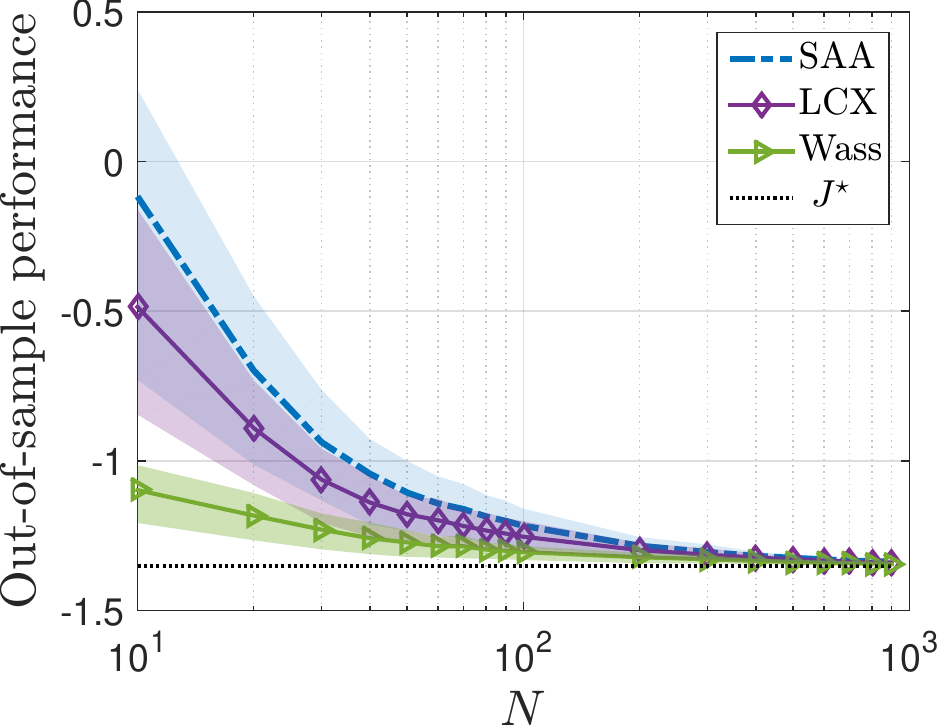}} \hspace{1mm}
	\subfigure[Optimal size]{\label{fig:cert-best}
		 \includegraphics[width=0.31\columnwidth]{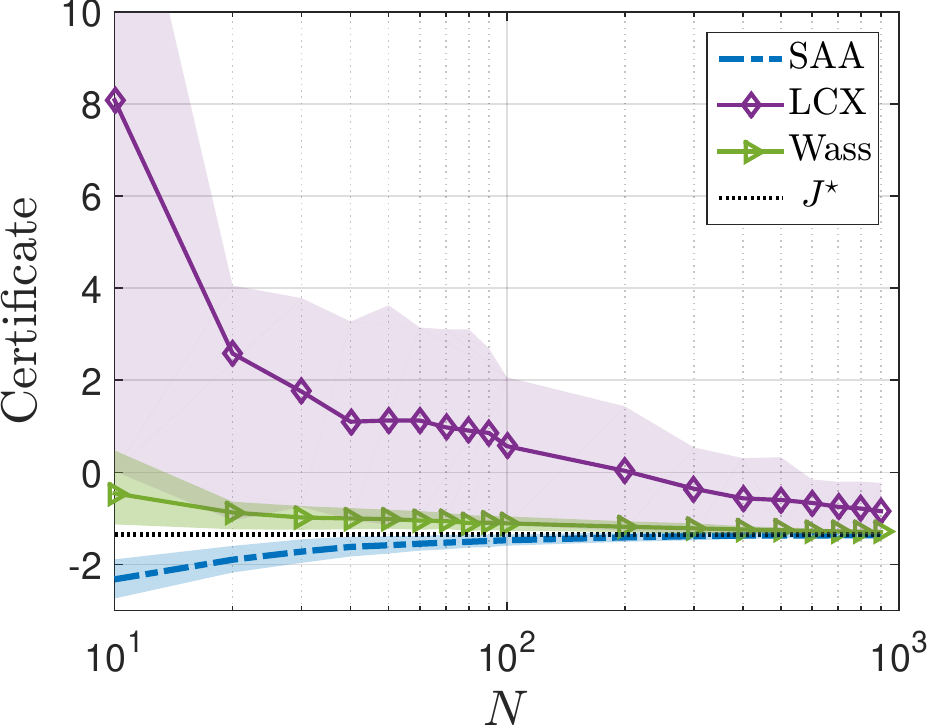}} \hspace{1mm}
	\subfigure[Optimal size]{\label{fig:beta-best}
		\includegraphics[width=0.31\columnwidth]{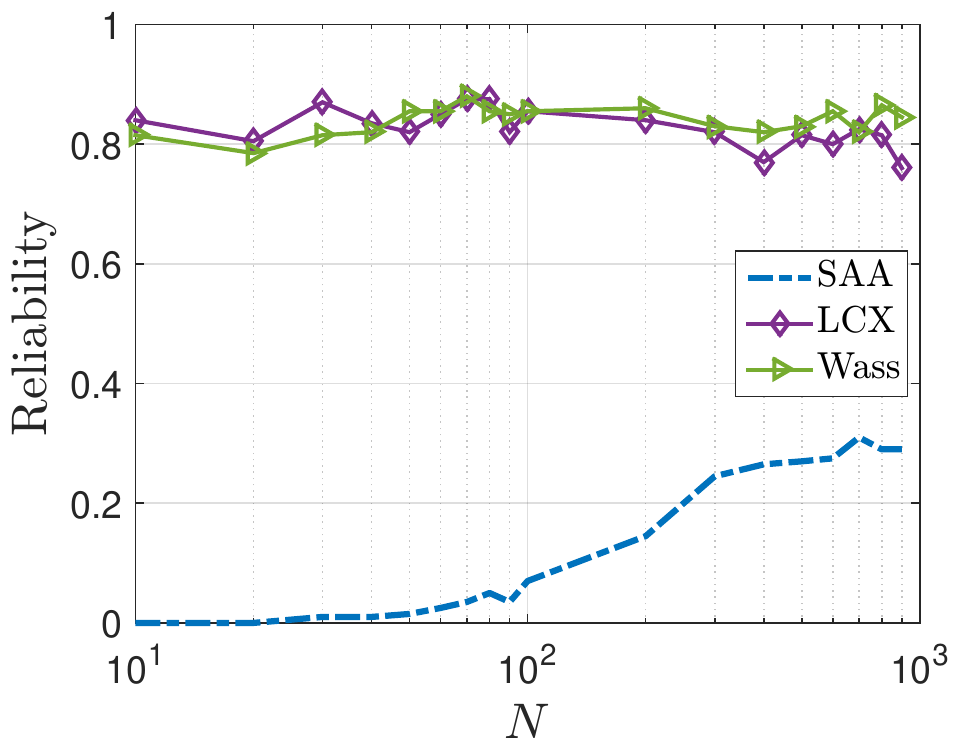}} 
	\caption{Out-of-sample performance $J(\xdd)$, certificate $\Jdd$, and certificate reliability $\PP^N\big[J(\xdd) \le \Jdd\big]$ for the performance-driven SAA, LCX and Wasserstein solutions as a function of $N$}
	\label{fig:perf:N}
\end{figure*}

In Figures~\ref{fig:perf-val}--\ref{fig:beta-val} the sizes of the (LCX and Wasserstein) ambiguity sets are determined via the holdout method, where $80\%$ of the data are used for training and $20\%$ for validation. Figure~\ref{fig:perf-val} visualizes the tube between the $20\%$ and $80\%$ quantiles (shaded areas) as well as the mean value (solid lines) of the out-of-sample performance $J(\xdd)$ as a function of the sample size $N$ and based on 200 independent simulation runs, where $\xdd$ is set to the minimizer of the SAA (blue), LCX (purple) and Wasserstein (green) problems, respectively. The constant dashed line represents the optimal value $\J$ of the original stochastic program~\eqref{Ex-true}, which is computed through an SAA problem with $N = 10^6$ samples. We observe that the Wasserstein solutions tend to be superior to the SAA and LCX solutions in terms of out-of-sample performance.

Figure~\ref{fig:cert-val} shows the optimal values $\Jdd$ of the SAA, LCX and Wasserstein problems, where the sizes of the ambiguity sets are chosen via the holdout method. Unlike Figure~\ref{fig:perf-val}, Figure~\ref{fig:cert-val} thus reports {\em in-sample} estimates of the achievable portfolio performance. As expected, the SAA approach is over-optimistic due to the optimizer's curse, while the LCX and Wasserstein approaches err on the side of caution. All three methods are known to enjoy asymptotic consistency, which is in agreement with all in-sample and out-of-sample results.

Figure~\ref{fig:beta-val} visualizes the reliability of the different performance certificates, that is, the empirical probability of the event $J(\xdd) \le \Jdd$ evaluated over 200 independent simulation runs. Here, $\xdd$ represents either an optimal portfolio of the SAA, LCX or Wasserstein problems, while $\Jdd$ denotes the corresponding optimal value. The optimal SAA portfolios display a disappointing out-of-sample performance relative to the optimistically biased mimimum of the SAA problem---particularly when the training data is scarce. In contrast, the out-of-sample performance of the optimal LCX and Wasserstein portfolios often undershoots~$\Jdd$.

Figures~\ref{fig:perf-k}--\ref{fig:beta-k} show the same graphs as Figures~\ref{fig:perf-val}--\ref{fig:beta-val}, but now the sizes of the ambiguity sets are determined via $k$-fold cross validation with $k=5$. In this case, the out-of-sample performance of both distributionally robust methods improves slightly, while the corresponding certificates and their reliabilities increase significantly with respect to the na\"ive holdout method. However, these improvements come at the expense of a $k$-fold increase in the computational cost.  

One could think of numerous other statistical methods to select the size of the Wasserstein ambiguity set. As discussed above, however, if the ultimate goal is to minimize the out-of-sample performance of $\xdd(\eps)$, then the best possible choice is $\eps=\wh \eps_N^{\rm \; opt}$. Similarly, one can construct a size parameter for the LCX ambiguity set that leads to the best possible out-of-sample performance of any LCX solution. We emphasize that these optimal Wasserstein radii and LCX size parameters are not available in practice because computing $J(\xdd(\eps))$ requires knowledge of the data-generating distribution. In our experiments we evaluate $J(\xdd(\eps))$ to high accuracy for every fixed $\eps\in\mathcal E$ using $2\cdot 10^5$ validation samples, which are independent from the (much fewer) training samples used to compute $\xdd(\eps)$. Figures~\ref{fig:perf-best}--\ref{fig:beta-best} show the same graphs as Figures~\ref{fig:perf-val}--\ref{fig:beta-val} for optimally sized ambiguity sets. By construction, no method for sizing the Wasserstein or LCX ambiguity sets can result in a better out-of-sample performance, respectively. In this sense, the graphs in Figure~\ref{fig:perf-best} capture the fundamental limitations of the different distributionally robust schemes.

\subsubsection{\bf Portfolios Driven by Reliability}
\label{subsub:sim:N2}

In Section~\ref{subsub:sim:N} the Wasserstein radii and LCX size parameters were calibrated with the goal to achieve the best out-of-sample performance. Figures~\ref{fig:perf:N}(c),~\ref{fig:perf:N}(f) and~\ref{fig:perf:N}(i) reveal, however, that by optimizing the out-of-sample performance one may sacrifice reliability. An alternative objective more in line with the general philosophy of Section~\ref{sec:prob} would be to choose Wasserstein radii that guarantee a prescribed reliability level. Thus, for a given $\beta\in[0,1]$ we should find the smallest Wasserstein radius $\eps\geq 0$ for which the optimal value $\Jdd(\eps)$ of~\eqref{dro:portfolio} provides an upper $1-\beta$ confidence bound on the out-of-sample performance $J(\xdd(\eps))$ of its optimal solution. As the true distribution $\PP$ is unknown, however, the optimal Wasserstein radius corresponding to a given~$\beta$ cannot be computed exactly. Instead, we must derive an estimator $\wh \eps_N^{\; \beta}$ that depends on the training data. We construct $\wh \eps_N^{\; \beta}$ and the corresponding reliability-driven portfolio via bootstrapping as follows:
\begin{enumerate}
\item Construct $k$ resamples of size $N$ (with replacement) from the original training dataset. It is well known that, as $N$ grows, the probability that any fixed training data point appears in a particular resample converges to $\frac{e-1}{e}\approx \frac{2}{3}$. Thus, about $\frac{N}{3}$ training samples are absent from any resample. We collect all unused samples in a validation dataset. 
\item For each resample $\kappa=1,\ldots, k$ and $\eps\geq 0$, solve problem~\eqref{dro:portfolio} using the Wasserstein ball of radius $\eps$ around the empirical distribution $\Pem^\kappa$ on the $\kappa$-th resample. The resulting optimal decision and optimal value are denoted as $\wh x_N^\kappa(\eps)$ and $\wh J_N^\kappa(\eps)$, respectively. Next, estimate the out-of-sample performance $J(\xdd^\kappa(\eps))$ of $\xdd^\kappa(\eps)$ using the sample average over the $\kappa$-th validation dataset.
\item Set $\wh \eps_N^{\; \beta}$ to the smallest $\eps\geq 0$ so that the certificate $\wh J_N^\kappa(\eps)$ exceeds the estimate of $J(\wh x_N^\kappa(\eps))$ in at least $(1-\beta)\times k$ different resamples. 
\item Compute the data-driven portfolio $\xdd=\xdd(\wh \eps_N^{\; \beta})$ and the corresponding certificate $\Jdd=\wh J_N(\wh \eps_N^{\; \beta})$ using the original training dataset.
\end{enumerate}
As in Section~\ref{subsub:sim:N}, we compare the Wasserstein approach with the LCX and SAA approaches. Specifically, by using bootstrapping, we calibrate the size of the LCX ambiguity set so as to guarantee a desired reliability level $1-\beta$. The SAA problem, on the other hand, has no free parameter that can be tuned to meet a prescribed reliability target. Nevertheless, we can construct a meaningful certificate of the form $\Jdd(\Delta):=\Jsaa+\Delta$ for the SAA portfolio by adding a non-negative constant to the optimal value of the SAA problem. Our aim is to find the smallest offset $\Delta\geq 0$ with the property that $\Jdd(\Delta)$ provides an upper $1-\beta$ confidence bound on the out-of-sample performance $J(\xsaa)$ of the optimal SAA portfolio $\xsaa$. The optimal offset corresponding to a given~$\beta$ cannot be computed exactly. Instead, we must derive an estimator $\wh \Delta_N^{\; \beta}$ that depends on the training data. Such an estimator can be found through a simple variant of the above bootstrapping procedure.

\begin{figure*} [t]
    \centering
    \subfigure[$\beta = 10\%$]{\label{fig:perf-boot-10}
    \includegraphics[width=0.31\columnwidth]{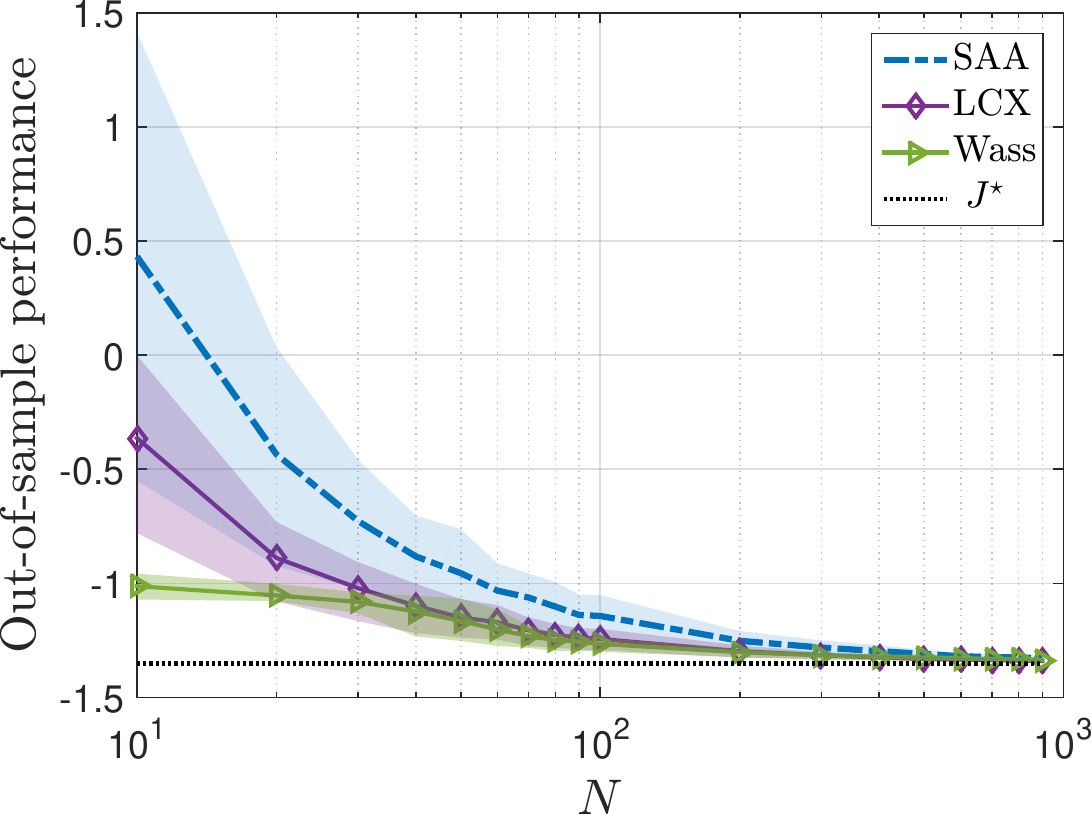}} \hspace{1mm}
    \subfigure[$\beta = 10\%$]{\label{fig:cert-boot-10}
    \includegraphics[width=0.31\columnwidth]{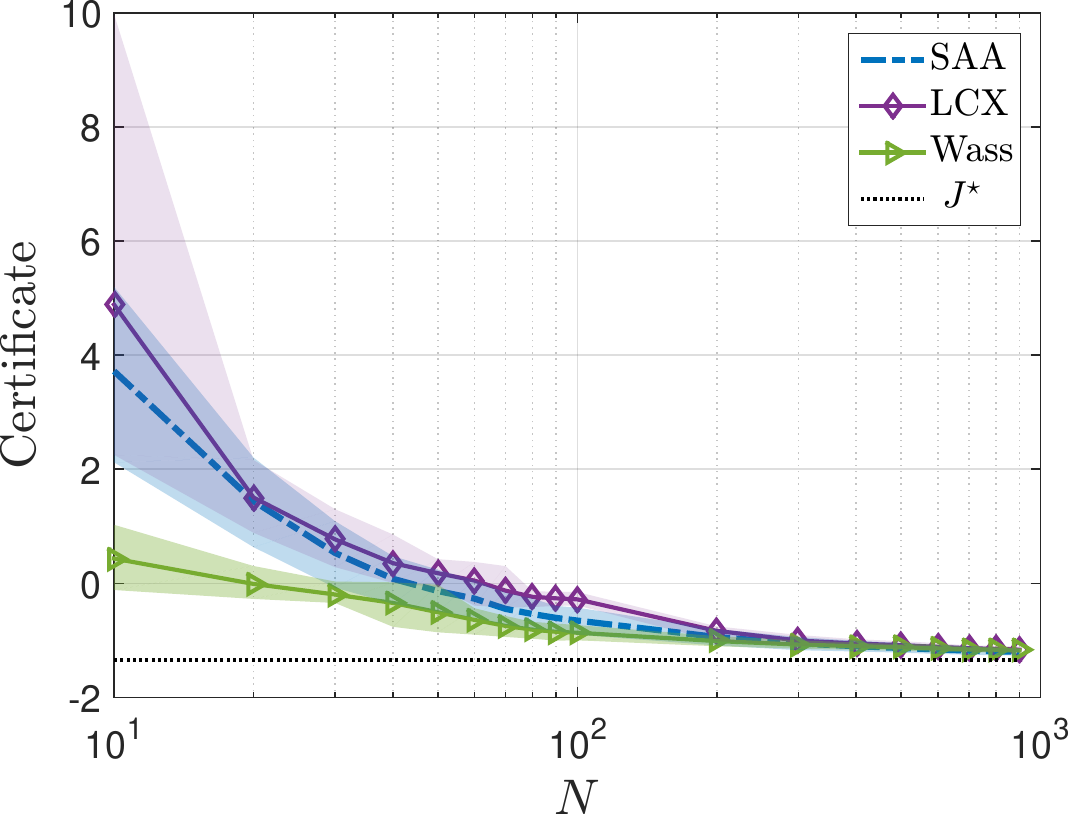}} \hspace{1mm}
    \subfigure[$\beta = 10\%$]{\label{fig:beta-boot-10}
    \includegraphics[width=0.31\columnwidth]{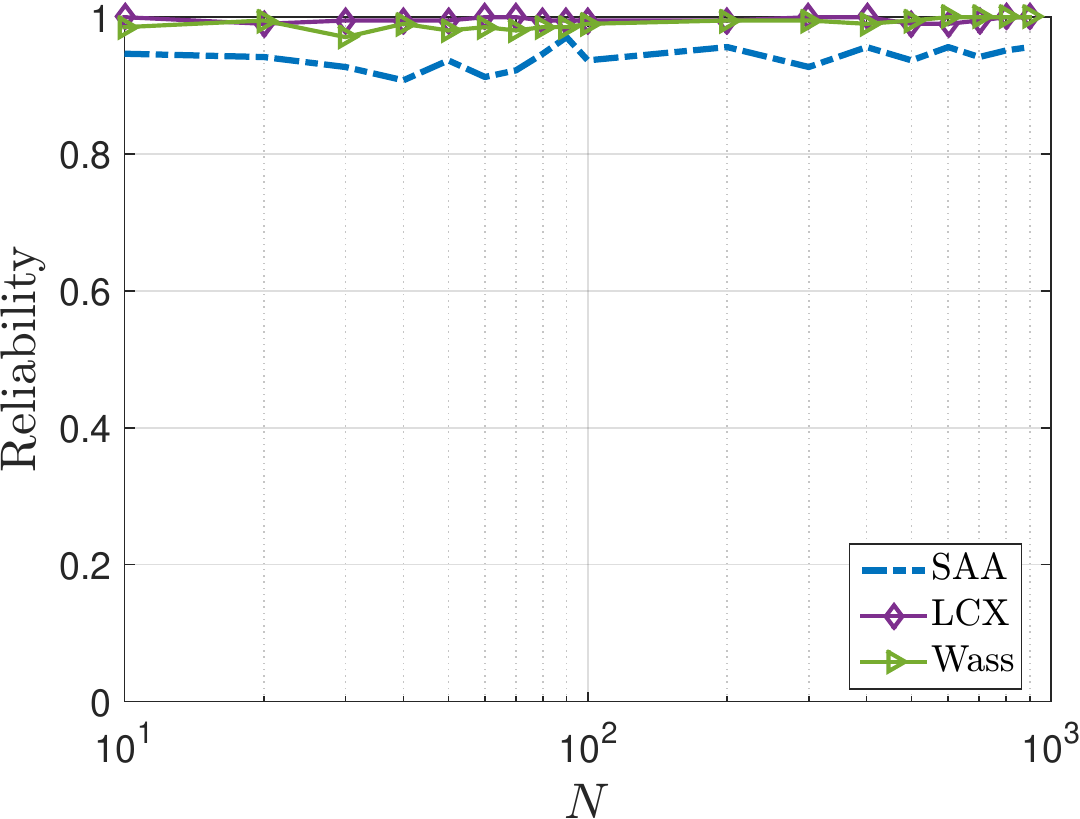}} \hspace{1mm}
    \subfigure[$\beta = 25\%$]{\label{fig:perf-boot-25}
    \includegraphics[width=0.31\columnwidth]{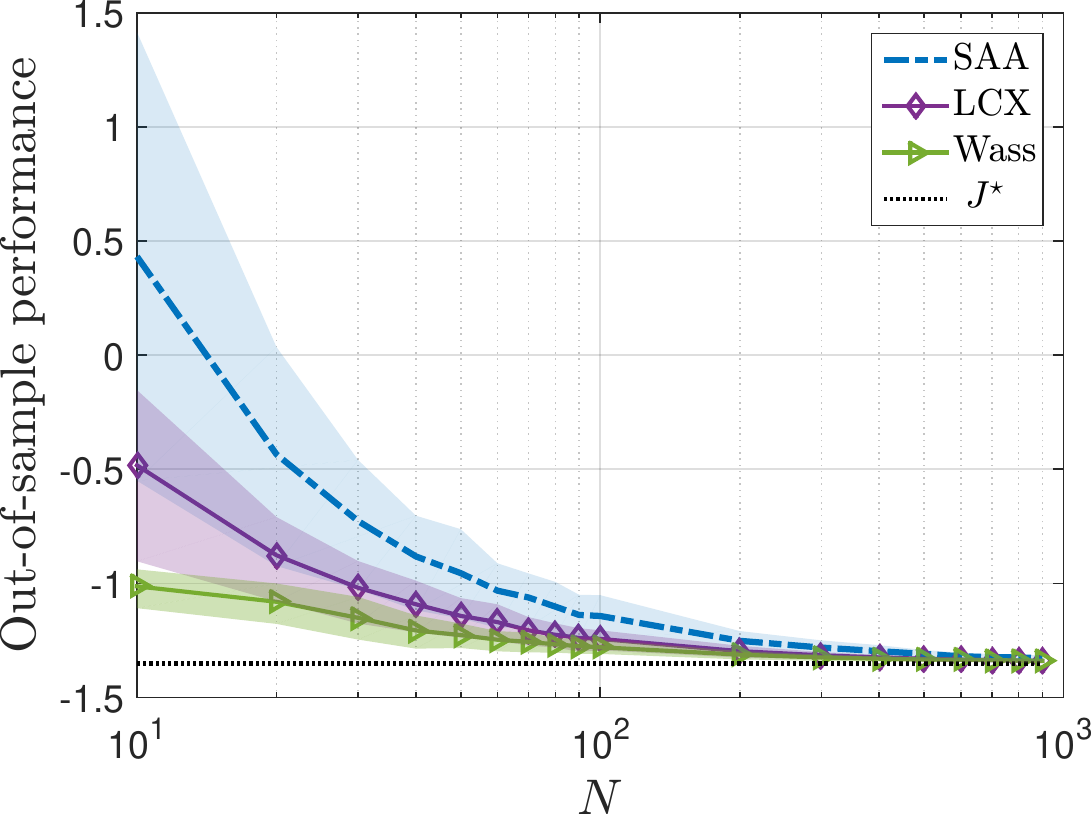}} \hspace{1mm}
    \subfigure[$\beta = 25\%$]{\label{fig:cert-boot-25}
    \includegraphics[width=0.31\columnwidth]{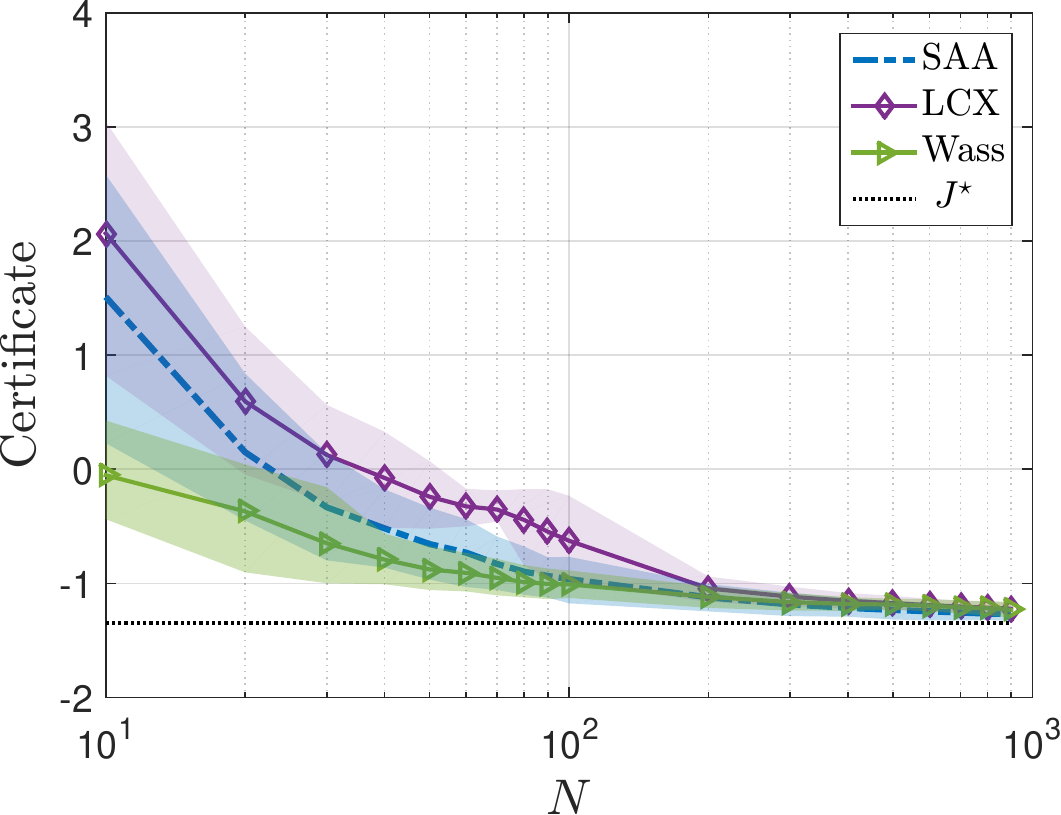}} \hspace{1mm}
    \subfigure[$\beta = 25\%$]{\label{fig:beta-boot-25}
    \includegraphics[width=0.31\columnwidth]{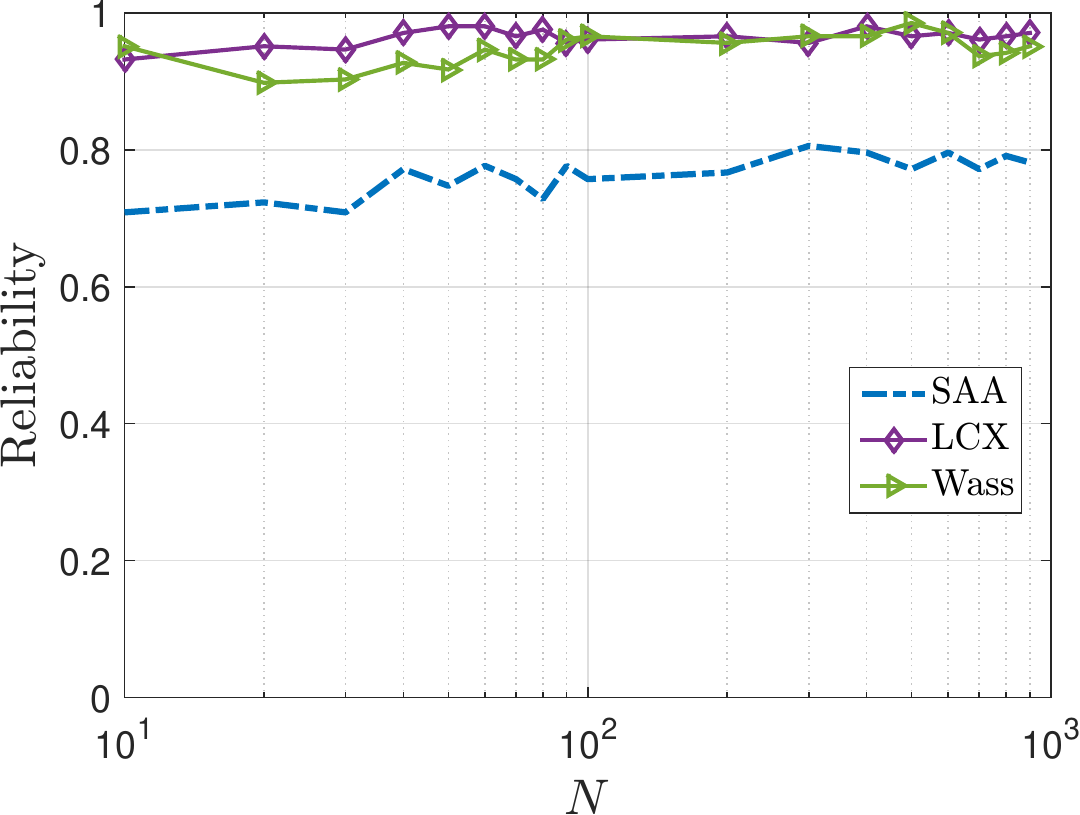}} \hspace{1mm}
    \caption{Out-of-sample performance $J(\xdd)$, certificate $\Jdd$, and certificate reliability $\PP^N\big[J(\xdd) \le \Jdd\big]$ for the reliability-driven SAA, LCX and Wasserstein portfolios as a function of $N$}
    \label{fig:boot}
\end{figure*}

In all experiments we set the number of resamples to $k=50$. Figures~\ref{fig:boot}(a)--\ref{fig:boot}(c) visualize the out-of-sample performance, the certificate and the empirical reliability of the reliability-driven portfolios obtained with the SAA, LCX and Wasserstein approaches, respectively, for the reliability target $1-\beta=90\%$ and based on 200 independent simulation runs. Figures~\ref{fig:boot}(d)--\ref{fig:boot}(f) show the same graphs as Figures~\ref{fig:boot}(a)--\ref{fig:boot}(c) but for the reliability target $1-\beta=75\%$. We observe that the new SAA certificate now overestimates the true optimal value of the portfolio problem. Moreover, while the empirical reliability of the SAA solution now closely matches the desired reliability target, the empirical reliabilities of the LCX and Wasserstein solutions are similar but noticeably exceed the prescribed reliability threshold. A possible explanation for this phenomenon is that the $k$ resamples generated by the bootstrapping algorithm are not independent, which may give rise to a systematic bias in estimating the Wasserstein radii required for the desired reliability levels.

\subsubsection{\bf Impact of the Sample Size on the Wasserstein Radius}
\label{subsub:sim:N3}

It is instructive to analyze the dependence of the Wasserstein radii on the sample size $N$ for different data-driven schemes. As for the performance-driven portfolios from Section~\ref{subsub:sim:N}, Figure~\ref{fig:eps:N} depicts the best possible Wasserstein radius $\wh\eps_N^{\rm \; opt}$ as well as the Wasserstein radii $\wh\eps_N^{\rm \; hm}$ and $\wh\eps_N^{\rm \; cv}$ obtained by the holdout method and via $k$-fold cross validation, respectively. As for the reliability-driven portfolios from Section~\ref{subsub:sim:N2}, Figure~\ref{fig:eps:N} further depicts the Wasserstein radii $\wh\eps_N^{\beta}$, for $\beta\in\{10\%,25\%\}$, obtained by bootstrapping. All results are averaged across 200 independent simulation runs. As expected from Theorem~\ref{thm:convergence}, all Wasserstein radii tend to zero as $N$ increases. Moreover, the convergence rate is approximately equal to $N^{-\frac{1}{2}}$. This rate is likely to be optimal. Indeed, if $\X$ is a singleton, then every quantile of the sample average estimator $\Jsaa$ converges to $\J$ at rate $N^{-\frac{1}{2}}$ due to the central limit theorem. Thus, if $\wh \eps_N= o(N^{-\frac{1}{2}})$, then $\Jdd$ also converges to $\J$ at leading order $N^{-\frac{1}{2}}$ by Theorem~\ref{thm:convex}, which applies as the loss function is convex. This indicates that the a~priori rate $N^{-\frac{1}{m}}$ suggested by Theorem~\ref{thm:concentration} is too pessimistic in~practice.

\begin{figure} [t]
	\centering
	\includegraphics[width=.43\columnwidth]{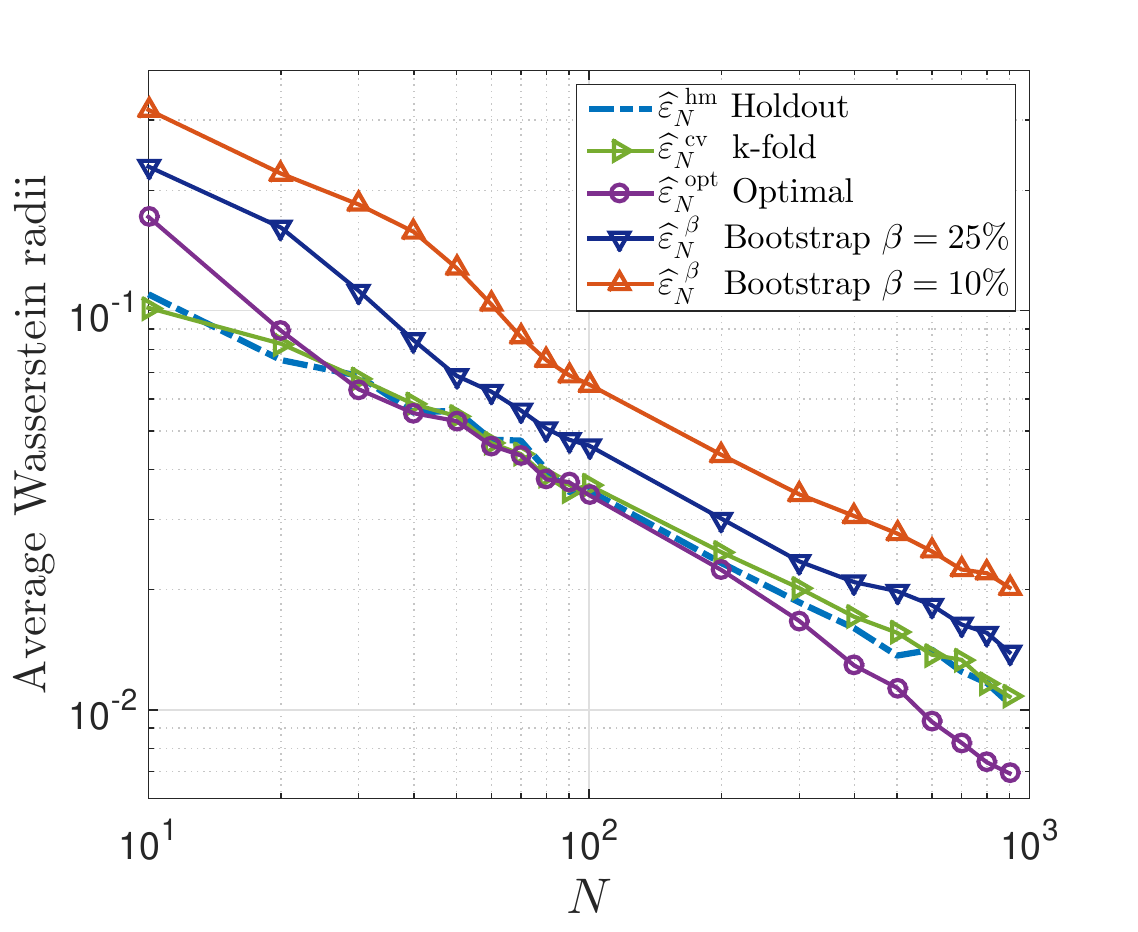} 
	\caption{Optimal performance-driven Wasserstein radius $\wh\eps_N^{\rm\; opt}$ and its estimates $\wh\eps_N^{\rm\; hm}$ and $\wh\eps_N^{\rm\; cv}$ obtained via the holdout method and $k$-fold cross validation, respectively, as well as the reliability-driven Wasserstein radius $\wh\eps_N^{\beta}$ for $\beta\in\{10\%,25\%\}$ obtained via  bootstrapping}
	\label{fig:eps:N}
\end{figure}

\subsection{Simulation Results: Uncertainty Quantification}

\label{sec:uqsimulation}

Investors often wish to determine the probability that a given portfolio will outperform various benchmark indices or assets. Our results on uncertainty quantification developed in Section~\ref{subsec:UQ} enable us to compute this probability in a meaningful way---solely on the basis of the training dataset. 

Assume for example that we wish to quantify the probability that any data-driven portfolio $\xdd$ outperforms the three most risky assets in the market {\em jointly}. Thus, we should compute the probability of the closed polytope
\begin{align*}
	\wh{\set{A}} = \Big\{\xi \in \R^m ~:~ \inner{\xdd}{\xi}\geq \xi_i ~ \forall i=8,9,10 \Big \}.
\end{align*}
As the true distribution $\PP$ is unknown, the probability $\PP[\xi \in \wh{\set A}]$ cannot be evaluated exactly. Note that $\wh{\set A}$ as well as $\PP[\xi \in \wh{\set A}]$ constitute random objects that depend on $\xdd$ and thus on the training data. Using the {\em same} training dataset that was used to compute $\xdd$, however, we may estimate $\PP[\xi \in \wh{\set A}]$ from above and below by 
\begin{align*}
 	\sup\limits_{\Q \in \ball{\Pem}{\eps}} \Q \left[ \xi\in \wh{\set A}\right]\qquad \text{and}\qquad 
 	\inf\limits_{\Q \in \ball{\Pem}{\eps}} \Q \left[ \xi\in \wh{\set A}\right] = 1 - \sup\limits_{\Q \in \ball{\Pem}{\eps}} \Q \left[ \xi\notin \wh{\set A}\right],
\end{align*}
respectively. Indeed, recall that the true data-generating probability distribution resides in the Wasserstein ball of radius $\eps_N(\beta)$ defined in~\eqref{eps_N} with probability $1-\beta$. Therefore, we have 
\begin{align*}
		1 - \beta  \le \PP^N\Big[\Xiem : \PP \in \ball{\Pem}{\eps_N(\beta)}\Big] & \le 	\PP^N\Big[\Xiem : ~ \sup_{\Q \in \ball{\Pem}{\eps_N(\beta)}} \Q\big[{\mathbb A}\big] \ge \PP\big[{\mathbb A}\big] \quad \forall \mathbb A \in \borel(\Xi) \Big] \\
		& = \PP^N\Big[\Xiem : ~ \inf_{\mathbb A \in \borel(\Xi)} \sup_{\Q \in \ball{\Pem}{\eps_N(\beta)}} \Q\big[{\mathbb A}\big] - \PP\big[{\mathbb A}\big] \ge 0 \Big],
\end{align*}
where $\borel(\Xi)$ denotes the set of all Borel subsets of $\Xi$. The data-dependent set $\wh{\set A}_N$ can now be viewed as a (measurable) mapping from $\Xiem$ to the subsets in $\borel(\Xi)$. The above inequality then implies
\begin{align*}
		\PP^N\Big[\Xiem : ~  \sup_{\Q \in \ball{\Pem}{\eps_N(\beta)}} \Q\big[{\wh{\set A}_N}\big] - \PP\big[{\wh{\set A}_N}\big] \ge 0 \Big]\ge 1-\beta.
		\end{align*}
Thus, $ \sup\{\Q [{\wh{\set A}_N}]:\Q \in \ball{\Pem}{\eps_N(\beta)}\}$ provides indeed an upper bound on $\PP[{\wh{\set A}_N}]$ with confidence $1-\beta$. Similarly, one can show that $ \inf \{\Q[{\wh{\set A}_N}]: \Q \in \ball{\Pem}{\eps_N(\beta)}\}$ provides a lower confidence bound on $\PP[{\wh{\set A}_N}]$.

The upper confidence bound can be computed by solving the linear program~\eqref{chance-worst}. Replacing $\wh{\set A}$ with its interior in the lower confidence bound leads to another (potentially weaker) lower bound that can be computed by solving the linear program~\eqref{chance-best}. We denote these computable bounds by $\Jdd^+(\eps)$ and $\Jdd^-(\eps)$, respectively. In all subsequent experiments $\xdd$ is set to a solution of the distributionally robust program~\eqref{dro:portfolio} calibrated via $k$-fold cross validation as described in Section~\ref{subsub:sim:N}. 

\subsubsection{\bf Impact of the Wasserstein Radius}\label{subsubsec:sim:UQ_eps}

\begin{figure}[t!]
	\centering
	\subfigure[$N=30$]{\label{fig:UQ:30} \includegraphics[width=0.43\columnwidth]{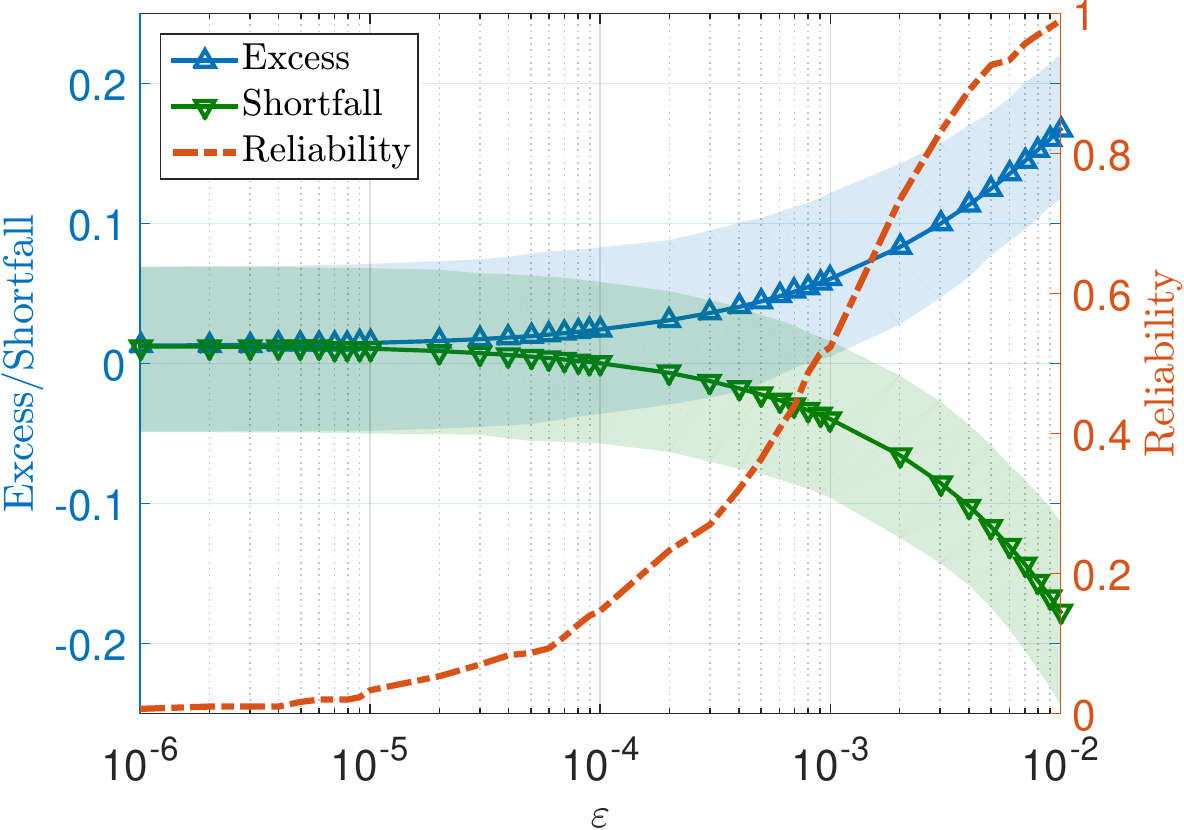}} \quad
	\subfigure[$N=300$]{\label{fig:UQ:300} \includegraphics[width=0.43\columnwidth]{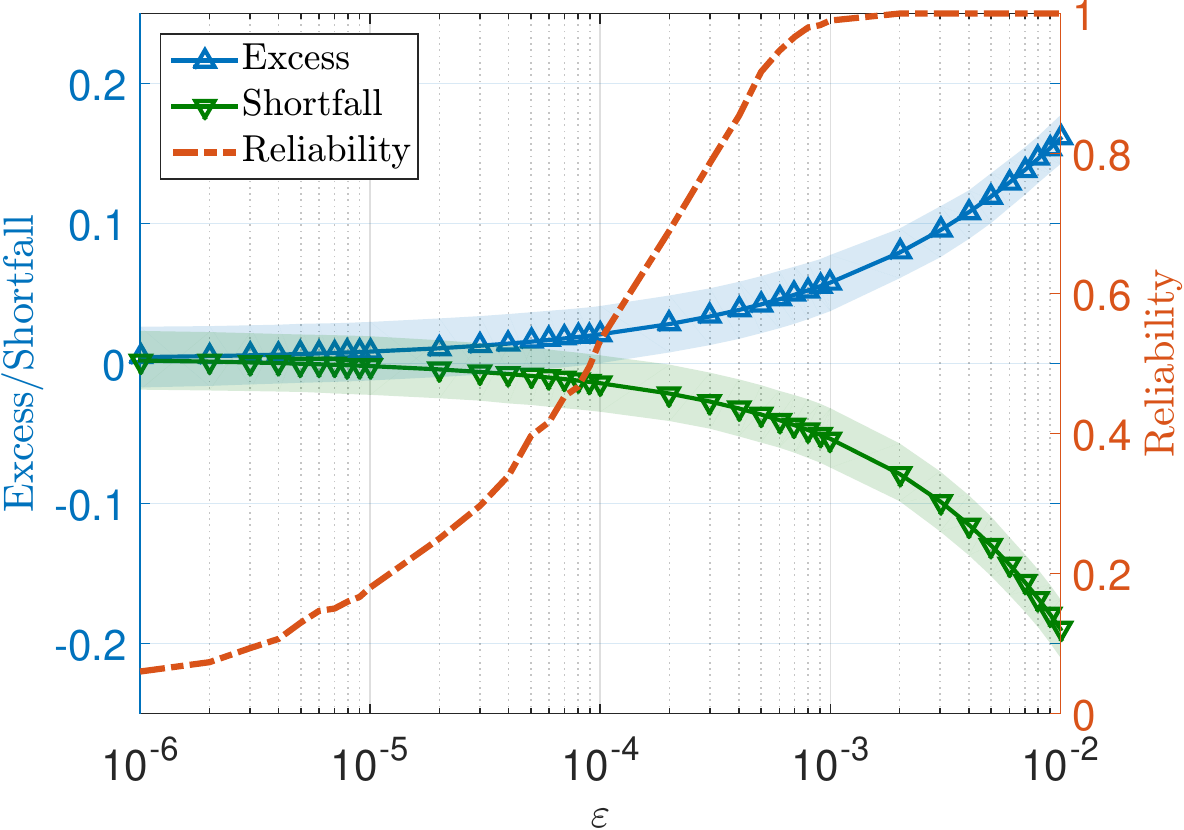}}
	\caption{Excess $\Jdd^+(\eps)- \PP[\wh{\set A}]$ and shortfall $\Jdd^-(\eps)- \PP[\wh{\set A}]$ (solid lines, left axis) as well as reliability $\PP^N[\Jdd^-(\eps) \leq \PP[\wh{\set A}] \leq \Jdd^+(\eps)]$ (dashed lines, right axis) as a function of $\eps$}
	\label{fig:UQ_radius}
\end{figure}

As $\Jdd^+(\eps)$ and $\Jdd^-(\eps)$ estimate a random target $\PP[\wh{\set A}]$, it makes sense to filter out the randomness of the target and to study only the differences $\Jdd^+(\eps)- \PP[\wh{\set A}]$ and $\Jdd^-(\eps)- \PP[\wh{\set A}]$. Figures~\ref{fig:UQ:30} and~\ref{fig:UQ:300} visualize the empirical mean (solid lines) as well as the tube between the empirical 20\% and 80\% quantiles (shaded areas) of these differences as a function of the Wasserstein radius $\eps$, based on 200 training datasets of cardinality $N = 30$ and $N=300$, respectively. 
Figure~\ref{fig:UQ_radius} also shows the empirical reliability of the bounds (dashed lines), that is, the empirical probability of the event $\Jdd^-(\eps) \leq \PP[\wh{\set A}] \leq \Jdd^+(\eps)$. Note that the reliability drops to 0 for $\eps=0$, in which case both $\Jdd^+(0)$ and $\Jdd^-(0)$ coincide with the SAA estimator for $\PP[\wh{\set A}]$. Moreover, at $\eps=0$ the set $\wh{\set A}$ is constructed from the SAA portfolio $\xdd$, whose performance is overestimated on the training dataset. Thus, the SAA estimator for $\PP[\wh{\set A}]$, which is evaluated using the same training dataset, is positively biased. For $\eps>0$, finally, the reliability increases as the shaded confidence intervals move away from 0.

\subsubsection{\bf Impact of the Sample Size}\label{subsubsec:sim:UQ_N}

We propose a variant of the $k$-fold cross validation procedure for selecting $\eps$ in uncertainty quantification. Partition $\data_1,\ldots,\data_N$ into $k$ subsets and repeat the following holdout method $k$ times. Select one of the subsets as the validation set of size $N_V$ and merge the remaining $k-1$ subsets to a training dataset of size $N_T=N-N_V$. Use the validation set to compute the SAA estimator of $\PP[\wh{\set A}]$, and use the training dataset to compute $\wh J_{N_T}^+(\eps)$ for a large but finite number of candidate radii $\eps$. Set $\wh \eps_N^{\; \rm hm}$ to the smallest candidate radius for which the SAA estimator of $\PP[\wh{\set A}]$ is not larger than $\wh J_{N_T}^+(\eps)$. Next, set $\wh \eps_N^{\rm \; cv}$ to the average of the Wasserstein radii obtained from the $k$ holdout runs, and report $\Jdd^+=\wh J_{N}^+(\wh \eps_N^{\rm \; cv})$ as the data-driven upper bound on $\PP[\wh{\set A}]$. The data-driven lower bound $\Jdd^-$ is constructed analogously in the obvious way. 

Figure~\ref{fig:UQ_N} visualizes the empirical means (solid lines) as well as the tubes between the empirical 20\% and 80\% quantiles (shaded areas) of $\Jdd^+-\PP[\wh{\set A}]$ and $\Jdd^--\PP[\wh{\set A}]$ as a function of the sample size $N$, based on 300 independent training datasets. As expected, the confidence intervals shrink and converge to $0$ as $N$ increases. We emphasize that $\Jdd^+$ and $\Jdd^-$ are computed solely on the basis of $N$ training samples, whereas the computation of $\PP[\wh{\set A}]$ necessitates a much larger dataset, particularly if $\wh{\set A}$ constitutes a rare event. 

\begin{figure*} [t]
	\centering
	\subfigure[Excess $\Jdd^+-\PP{[}\wh{\set A}{]}$ and shortfall $\Jdd^--\PP{[}\wh{\set A}{]}$ of the data-driven confidence bounds for $\PP{[}\wh{\set A}{]}$]{\label{fig:UQ_N} \includegraphics[width=0.465\columnwidth]{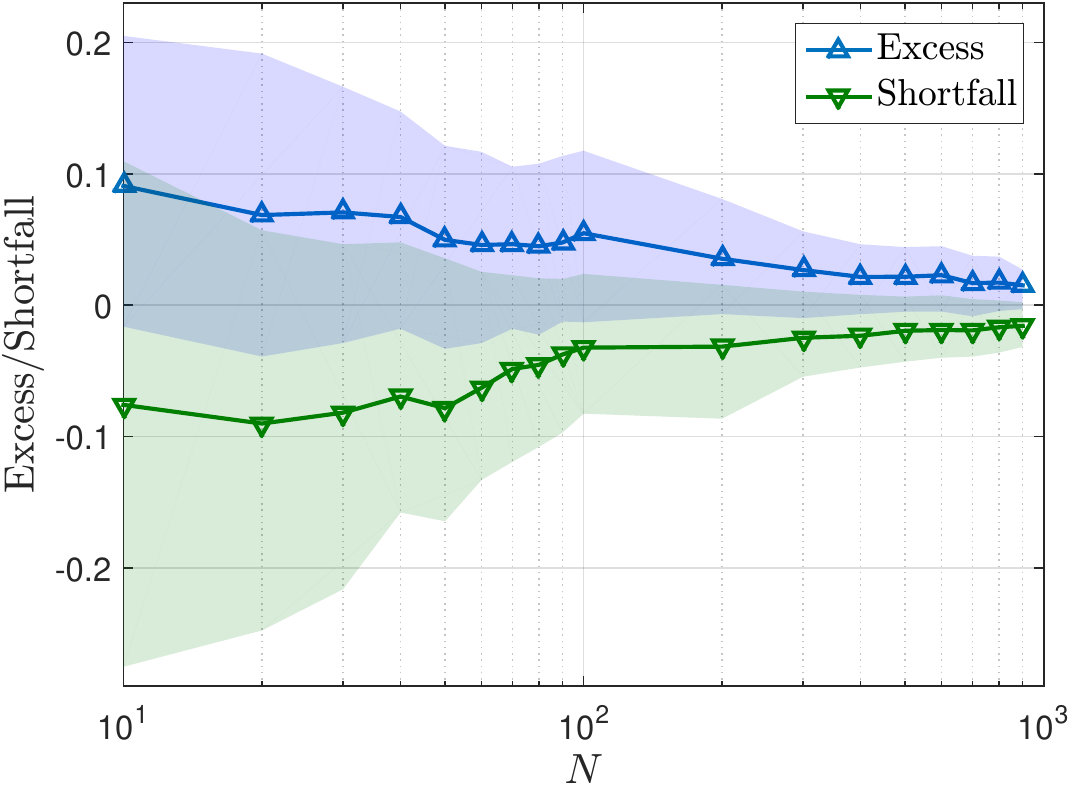}} \quad
	\subfigure[Data-driven Wasserstein radius $\wh\eps_N^{\rm \; cv}$ obtained via $k$-fold cross validation]{\label{fig:UQ_eps} \includegraphics[width=0.49\columnwidth]{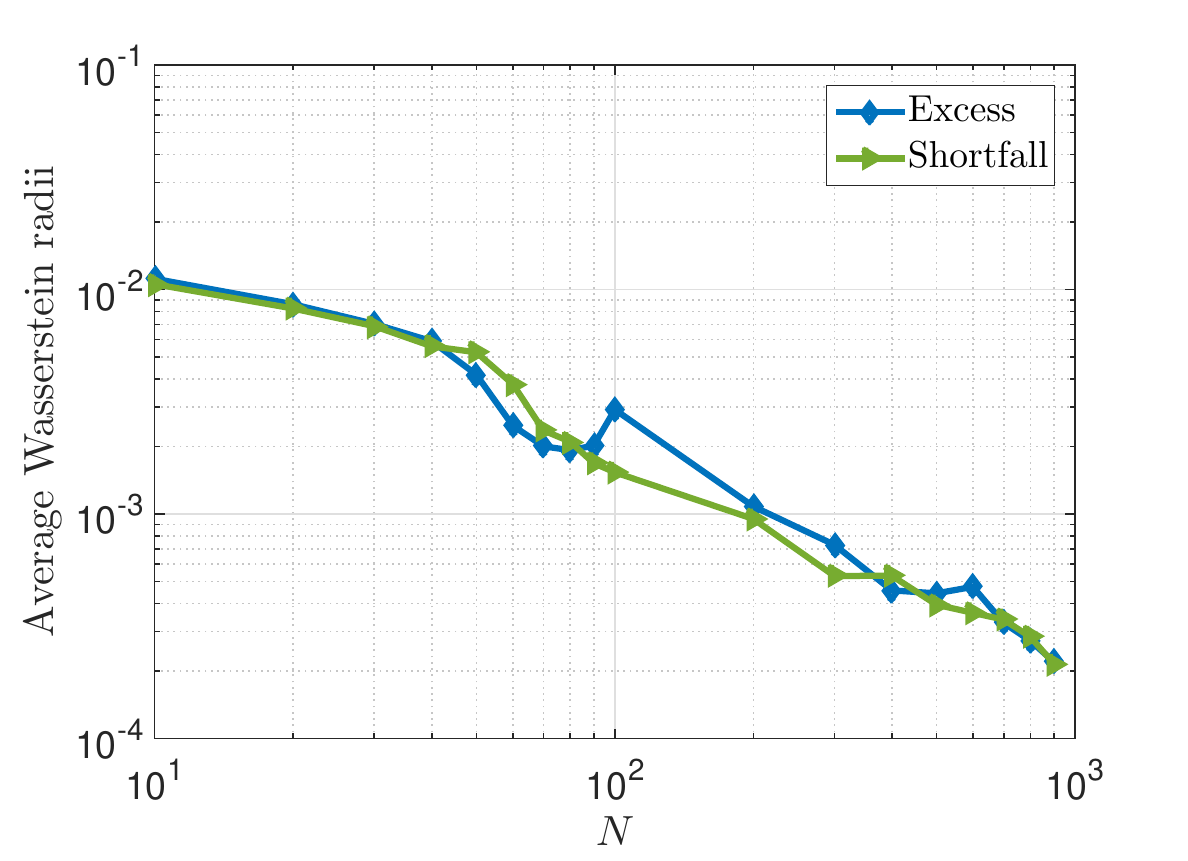}}
	\caption{Dependence of the confidence bounds and the Wasserstein radius on $N$}
	\label{fig:UQ_learning}
\end{figure*}

Figure~\ref{fig:UQ_eps} shows the Wasserstein radius $\wh\eps_N^{\rm \; cv}$ obtained via $k$-fold cross validation (both for $\Jdd^+$ and $\Jdd^-$). As usual, all results are averaged across 300 independent simulation runs. A comparison with Figure~\ref{fig:eps:N} reveals that the data-driven Wasserstein radii in uncertainty quantification display a similar but faster polynomial decay than in portfolio optimization. We conjecture that this is due to the absence of decisions, which implies that uncertainty quantification is less susceptible to the optimizer's curse. Thus, nature ({\em i.e.}, the fictitious adversary choosing the distribution in the ambiguity set) only has to compensate for noise but not for bias. A smaller Wasserstein radius seems to be sufficient for this purpose.

\paragraph*{\bf Acknowledgments} We thank Soroosh Shafieezadeh Abadeh for helping us with the numerical experiments. The authors are grateful to Vishal Gupta, Ruiwei Jiang and Nathan Kallus for their valuable comments. This research was supported by the Swiss National Science Foundation under Grant BSCGI0\_157733.

\appendix

\section{}\label{app}
The following technical lemma on the pointwise approximation of an upper semicontinuous function by a non-increasing sequence of Lipschitz continuous majorants strengthens~\cite[Theorem 4.2]{ref:Mash-09}, which focuses on bounded domains and continuous (but not necessarily Lipschitz continuous) majorants.

\begin{Lem}
	\label{lem:p.w.app}
	If $h:\Xi \ra \R$ is upper semicontinuous and satisfies $h(\xi) \le L(1+\|\xi\|)$ for some $L\geq 0$, then there exists a non-increasing sequence of Lipschitz continuous functions that converge pointwise to~$h$ on~$\Xi$. 
\end{Lem}

\begin{proof}
	The proof is constructive. Define the functions  
	\begin{align*}
		h_k(\xi) = \sup_{\xi' \in \Xi} h(\xi') - kL\|\xi - \xi'\|, \quad k \in \N,
	\end{align*}
	where $L$ is the linear growth rate of~$h$. Note that by construction $h_k(\xi)\le L(1+\|\xi\|)$. As $\xi'=\xi$ is feasible in the maximization problem defining $h_k(\xi)$, we have $h_k(\xi) \ge h(\xi)$ for all $\xi\in\Xi$ and $k\in\N$. Moreover, $h_k(\xi)$ is Lipschitz continuous with Lipschitz constant~$kL$ (as $h_k(\xi)$ constitutes a supremum of norm functions with this property). 
	Given any $\xi \in \Xi$, it remains to be shown that $\lim_{k\ra \infty}h_k(\xi) = h(\xi)$. Thus, choose $\xi'_k \in \Xi$ with 
	\begin{align*}
		h_k(\xi) = \sup_{\xi' \in \Xi} h(\xi') - kL\|\xi - \xi'\| \le h(\xi'_k) - kL\|\xi - \xi'_k\| + {1\over k}.
	\end{align*}
	We first show that $\xi_k$ converges to $\xi$ as $k$ tends to $\infty$. Indeed, we have 
	\begin{align*} 
		h(\xi) &\le h_k(\xi)  \le h(\xi'_k) - kL\|\xi - \xi'_k\| + {1\over k} \le L(1+\|\xi'_k\|) - kL\|\xi - \xi'_k\| + {1\over k}\\
		&  \le L(1+\|\xi - \xi'_k\| + \|\xi\|) - kL\|\xi - \xi'_k\| + {1\over k} = L(1 + \|\xi\|) + {1\over k} - (k-1)  	L\|\xi - \xi'_k\|,
	\end{align*} 
	which implies
	\[
		\|\xi - \xi'_k\|\leq  {1\over L(k-1)} \left( h(\xi) - L(1 + \|\xi\|) -  {1\over k}\right),
	\]
	that is, $\|\xi - \xi'_k\| \ra 0$ as $k \ra \infty$. Therefore, we find
	\begin{align*}
			h(\xi) &\le \lim_{k\ra \infty} h_k(\xi) \le \limsup_{k\ra \infty} h(\xi'_k) - kL\|\xi - \xi'_k\| + {1\over k} \le \limsup_{k\ra \infty} h(\xi'_k) \le h(\xi),
	\end{align*}
	where the last inequality is due to the upper semicontinuity of $h$. This concludes the proof.
\end{proof}

	\bibliographystyle{siam}
	\bibliography{ref}
\end{document}